\documentclass[11pt]{amsart}

\usepackage{mathrsfs}
\usepackage{amsfonts}
\usepackage{amssymb}
\usepackage{amsxtra}
\usepackage{dsfont}
\usepackage{color}
\usepackage[compress, sort]{cite}
\usepackage{graphicx}
\usepackage[english,polish]{babel}
\usepackage{enumerate}

\usepackage{newtxmath}
\usepackage{newtxtext}

\usepackage[margin=2.5cm, centering]{geometry}
\usepackage[colorlinks,citecolor=blue,urlcolor=blue,bookmarks=true]{hyperref}
\hypersetup{
pdfpagemode=UseNone,
pdfstartview=FitH,
pdfdisplaydoctitle=true,
pdfborder={0 0 0}, % No link borders
pdftitle={Subordinated Markov processes: sharp estimates for heat kernels and Green functions},
pdfauthor={Tomasz Grzywny and Bartosz Trojan},
pdflang=en-US
}

\reversemarginpar

\definecolor{tg}{rgb}{0.7,0.1,0.2}
\definecolor{bt}{rgb}{0, 0, 0.6}

\newcommand{\ZZ}{\mathbb{Z}}

\newcommand{\NN}{\mathbb{N}}
\newcommand{\RR}{\mathbb{R}}
\newcommand{\PP}{\mathbb{P}}
\newcommand{\EE}{\mathbb{E}}

\newcommand{\calC}{\mathcal{C}}
\newcommand{\calL}{\mathcal{L}}

\newcommand{\calF}{\mathcal{F}}

\newcommand{\calB}{\mathcal{B}}

\newcommand{\calD}{\mathcal{D}}

\newcommand{\calT}{\mathcal{T}}
\newcommand{\calE}{\mathcal{E}}

\newcommand{\Id}{\operatorname{Id}}

\newcommand{\WUSC}[3]{\operatorname{WUSC}_0({#1}, {#2}, {#3})}
\newcommand{\WLSC}[3]{\operatorname{WLSC}_0({#1}, {#2}, {#3})}

\newcommand{\WUSCINF}[3]{\operatorname{WUSC}_\infty({#1}, {#2}, {#3})}
\newcommand{\WLSCINF}[3]{\operatorname{WLSC}_\infty({#1}, {#2}, {#3})}

\newcommand{\pl}[1]{\foreignlanguage{polish}{#1}}

\newcommand{\norm}[1]{\lvert {#1} \rvert}
\newcommand{\abs}[1]{\lvert {#1} \rvert}
\newcommand{\sprod}[2]{\langle {#1}, {#2} \rangle}

\newcommand{\ind}[1]{{\mathds{1}_{{#1}}}}
\newcommand{\vphi}{\varphi}

\newcommand{\supp}{\operatornamewithlimits{supp}}

\newtheorem{claim}{Claim}
\newtheorem*{theorem*}{Theorem}

\newtheorem{theorem}{Theorem}[section]
\newtheorem{proposition}[theorem]{Proposition}
\newtheorem{lemma}[theorem]{Lemma}
\newtheorem{corollary}[theorem]{Corollary}

\numberwithin{equation}{section}
\theoremstyle{definition}

\newtheorem{remark}[theorem]{Remark}

\title{Subordinated Markov processes: sharp estimates for heat kernels and the Green functions}
\date{\today}

\author{Tomasz Grzywny}
\address{
        \pl{
        Tomasz Grzywny\\
        Wydzia\l{} Matematyki,
        Politechnika Wroc\l{}awska\\
        Wyb. Wyspia\'{n}skiego 27\\
        50-370 Wroc\l{}aw\\
        Poland}
}
\email{tomasz.grzywny@pwr.edu.pl}

\author{Bartosz Trojan}
\address{
	\pl{
		Bartosz Trojan\\
        Instytut Matematyczny\\
        Polskiej Akademii Nauk\\
        ul. \'Sniadeckich 8\\
        00-696 Warszawa\\
        Poland}
}
\email{btrojan@impan.pl}
\thanks{The research was supported in part by the National Science Centre, Poland, Grant 2016/23/B/ST1/01665}

\begin{document}
\selectlanguage{english}

\begin{abstract}
	We prove sharp estimates on heat kernels and Green functions for subordinate Markov processes with both 
	discrete an continuous time, under relatively weak assumptions about original processes as well as Laplace exponents
	of subordinators. We also show stability of heat kernel estimates in the discrete settings.
\end{abstract}

\maketitle

\section{Introduction}

The concept of subordination for continuous time Markov processes and corresponding semigroups was introduced by Bochner 
in \cite{bochner1} (see also \cite[Chapter 4.4]{bochner2}). Subordination provides method of constructing a large class of
semigroups from the original one. To be more precise: given $(P_t : t\geq 0)$ a strongly continuous semigroup of contraction
on a Banach space with infinitesimal generator $(\mathcal{A}, D(\mathcal{A}))$, and $(\mu_t : t\geq 0)$ a vaguely continuous
semigroup of sub-probability measures on the half-line $[0,\infty)$ with the Laplace exponent $e^{-t\varphi}$, the subordinated
semigroup $(P_t^\varphi : t\geq 0)$ is defined by the Bochner integral
\[
	P^\varphi_t(f)=\int_{[0,\infty)} P_s(f) \, \mu_t({\rm d} s).
\]
Furthermore, the Bochner--Phillips functional calculus implies that the generator of the subordinated semigroup has a form
$-\varphi(-\mathcal{A})$, see e.g. \cite{phillips}. From the probabilistic point of view, if $(P_t : t \geq 0)$ gives rise to
a transition probability of a Markov process $(X(t) : t \geq 0)$, the subordinated semigroup corresponds to Markov process
$(X(T(t)) : t \geq 0)$ where $T(t)$ is a L\'{e}vy process on the real line with nonnegative increments such that $\mu_t$
is a distribution of $T(t)$. Therefore the subordinated semigroups have an independent interest in potential theory
\cite{berg_forst}, functional analysis \cite{carasso} and financial analysis \cite{cont}.

Similarly, if $( P_n : n\in \NN_0)$ is a semigroup and $(\mu_n :  n\in\NN_0)$ is a semigroup on $\NN$, one can consider a
new semigroup by 
\[
	Q_n(f) = \sum_{k=1}^\infty P_k(f) \mu_n(\{k\}).
\] 
In discrete time settings, a counterpart of the generator of the semigroup is an operator $P_1-\Id$. In fact in 
\cite{bendikov_saloff}, it is proved that if $\vphi(1) = 1$, and
\[
	\mu_1(\{k\})=\frac{1}{k!} |\varphi^{(k)}(1)|,
\]
then the discrete generator of $(Q_n : n \in \NN_0)$ has a form $-\varphi(\Id-P_1)$.

The aim of this article is to obtain estimates for the heat kernels (i.e., $n$-step transition densities) for a large class of
subordinated Markov processes with both continuous and discrete time. Transition densities of stochastic processes
have been studied for many years resulting in a large number of articles, see e.g. \cite{Barlow2017,Boutayeb2015,MR2039952,MR1736868,MR2569498,MR1872526,MR2508847}. The most of them concern diffusion
processes or random walks with finite range, however there is an increasing interest in studying continuous time jump 
processes see e.g. \cite{bae2019,Bendikova2014,MR2357678,MR2443765,MR3165234,Chen2016,MR4249776,MR4157572,Grigoryan2014,GrigoryanHuLau2014,GrigoryanHuHu2018,MR4039012,MR3965398,Kulik}. Lastly, there are relatively few articles on heavy tailed discrete time Markov chains
\cite{Murugan2015,MuruganSaloff2015, Murugan2019, Cygan2021}.

Our main result provide estimates for the transition density and the Green function of subordinated continuous and discrete
time Markov processes on a metric measure space $(M, d, \mu)$. Recall, that $(M, d, \mu)$ is a metric measure space if 
$(M, d)$ is a metric space and $(M, \calB, \mu)$ is a measure space where $\calB$ denotes $\sigma$-algebra of Borel sets and 
$\mu$ is a Radon measure on $(M, \calB)$. By elementary methods and under relatively weak assumptions on the Markov process
and the Laplace transform of subordinator we prove upper and lower bounds for transition density, see Theorems \ref{thm:100},
\ref{thm:101} and for the Green function, see Propositions \ref{prop:102} and \ref{prop:103}. A standard procedure to obtain
such results uses bounds for the heat kernel of the original process together with estimates for transition density of
subordinator (see e.g. \cite{MR2013738,TGLLBT,bae2019}). In our approach we only use properties of the Laplace exponent
of subordinator. Moreover, the upper bounds hold even for trivial subordinators that is deterministic. Since very often
the space possess a Markov process which heat kernel has sub-Gaussian or stable-like estimates we used our general results
in this case to obtain sharp estimates for examined functions, see  Theorems \ref{thm:1}, \ref{thm:2} and \ref{thm:4}.
Furthermore, we show that estimates obtained in Theorem \ref{thm:1} are equivalent with estimates of the first step in the
discrete time case (and the jumping kernel in continuous time case) or scaling conditions for the Laplace exponent of
subordinator, see Theorems \ref{thm:6} and \ref{thm:7}. Similar equivalence was previously known only for unimodal L\'{e}vy
processes on the Euclidean space \cite{MR3165234}.

Since discrete time Markov chains with heavy tailed distribution are less studied we concentrate on this case. In Section
\ref{sec:3} we investigate Markov processes defined on a discrete metric measure space $(M, d, \mu)$. We assume that $(M, d)$
is uniformly discrete, that is there is $r_0 > 0$ such that for all $r_0 > r > 0$, and $x \in M$,
\[
	B(x, r) = \{x\}
\]
where $B(x, r)$ denote the closed ball in $M$ centered at $x$ and with radius $r > 0$. Moreover, we impose that the function
$V(x, r) = \mu(B(x, r))$ has both doubling and reverse doubling property, see \eqref{eq:A:44} and \eqref{eq:A:42}, respectively.
In this context, from estimates on the heat kernel of subordinate process we would like to deduce a cut-off Sobolev
inequality \eqref{cutoff} for a certain function $\psi$. In the case of continuous time Markov processes
it is a consequence of \cite[Theorem 1.13]{Chen2016}. In Section \ref{sec:3} we study discrete time case by showing
that heat kernel estimates of the form \eqref{Dheat} imply the cut-off Sobolev inequality \eqref{cutoff}.
This lets us to consider Markov chains such that the first step satisfies
\[
	h(1; x, y) \approx \frac{1}{V(x, d(x, y)) \psi(d(x, y))} \qquad \text{for all } x, y \in M.
\]
To do so we introduce new continuous time process having transition density of the form
\[
	p(t; x, y) = \sum_{k = 0}^\infty h(k; x, y) \frac{e^{-t} t^k}{k!}, \qquad x, y \in M,\, t > 0.
\]
Since the state space $M$ is discrete we cannot apply to the heat kernel $p(t; x, y)$ the main results of 
\cite{Chen2016} nor \cite{GrigoryanHuHu2018}. However, we managed to adapt arguments from 
\cite{Chen2016,Boutayeb2015,Grigoryan2014} to conclude that $p$ satisfies \eqref{eq:A:32}, 
see Theorem \ref{thm:A:1}. This allows us to show the upper near diagonal estimates on $h$,
see Theorem \ref{thm:A:2}. Then by modifying the arguments from \cite{BassLevin}, we prove the parabolic Harnack inequality on
$\NN_0 \times M$, see Theorem \ref{thm:A:3}. The latter allows us to deduce the sharp heat kernel estimates \eqref{Dheat},
see Theorems \ref{thm:A:5}, \ref{thm:A:4}, \ref{thm:A:6}. The final result is summarized in Theorem \ref{thm:A:7}.
Let us point out that the proof of Theorem \ref{thm:A:3} is the only place where the assumption of the uniform discreteness
of the space $(M, d)$ is used. It is interesting if the parabolic Harnack inequality \eqref{eq:40} holds true in the discrete
time case but without uniform discreteness of the state space. Let us also observe that the main result of Section \ref{sec:3}
is the stability of the heat kernel estimates in the discrete settings. Having proved Theorem \ref{thm:A:7} we use it together
with Theorem \ref{thm:6} or Theorem \ref{thm:7} to show that the existence of a Markov chain having heat kernel which satisfies
sub-Gaussian or stable-like estimates is sufficient for \eqref{Dheat}, see  Theorems \ref{thm:5} and \ref{thm:3}. 
By Theorem \ref{thm:A:7} together with Theorem 
\ref{thm:5}, or \cite[Theorem 1.13]{Chen2016} together with Theorem \ref{thm:tg5}, we can deduce that the decay at zero as well
as the growth at infinity of the function $\psi$ which appears in the cut-off Sobolev inequality \eqref{cutoff} is limited
by the scale in sub-Gaussian heat kernel estimates. For instance, in the Euclidean space with $\psi(r)=r^{\alpha}$
we get the well-known limitation $\alpha<2$. In Section \ref{sec:5} we present possible applications of 
Theorems \ref{thm:5} and \ref{thm:3}, which improves results obtained previously in
\cite{Murugan2015,MuruganSaloff2015, Murugan2019, Cygan2021}.

\section{Preliminaries}

\subsection{Markov processes}
\label{sec:1}
Let $(M, d, \mu)$ be a metric measure space which means that $(M, d)$ is a metric space, and $(M, \calB, \mu)$ is a measure
space where $\calB$ denotes the $\sigma$-algebra of Borel sets and $\mu$ is a Radon measure on $(M, \calB)$. Let us denote by
$B(x, r)$ the closed metric ball in $M$ centered at $x$ and having radius $r$, that is
\[
	B(x, r) = \left\{y \in M : d(x, y) \leq r \right\}.
\]
We set
\[
	V(x, r) = \mu\big(B(x, r)\big).
\]
To treat at the same time discrete and continuous time Markov processes we set $\calT$ to be the half-line of positive reals
$(0, \infty)$ or positive integers $\NN$. Let $\calT_0 = \calT \cup \{0\}$.

Let  $\mathbf{S} = (S_t : t \in \calT_0)$ be a Markov process on $M$, i.e. a stochastic process that satisfies
Markov property. Given a set $A \in \calB$, we define the first exit time of $\mathbf{S}$ from $A$ as
\[
	\tau_{\mathbf{S}}(A) = \inf\big\{t \in \calT : S_t \notin A \big\}.
\]
In this article we assume that the transition function of $\mathbf{S}$ is absolutely continuous with respect to $\mu$. Let
$p: \mathcal{T}\times M\times M \rightarrow [0,\infty)$ be its density, that is
\[
	\PP^x(S_t \in A) = \int_A p(t; x, y) \, \mu({\rm d} y)
\]
for all $x \in M$ and $A \in \calB$.

\subsection{Weak scaling}
We say that a function $f: (0, a) \rightarrow (0, \infty)$ has the weak upper scaling property at zero of
index $\beta \in \RR$, if there are $c \in (0, 1]$ and $\theta_0 \in (0, a]$ such that
\[
	f(\lambda \theta) \geq c \lambda^\beta f(\theta)
\]
for all $\lambda \in (0, 1]$ and $\theta \in (0, \theta_0)$. We shortly write that $f$ belongs to 
$\WUSC{\beta}{c}{\theta_0}$. Similarly, $f$ has the lower scaling property at
zero of index $\alpha \in \RR$, if there is $C \in [1, \infty)$ and $\theta_0 \in (0, a]$ such that
\[
	f(\lambda \theta) \leq C \lambda^\alpha f(\theta)
\]
for all $\lambda \in (0, 1]$ and $\theta \in (0, \theta_0)$. We write that $f$ belongs to $\WLSC{\alpha}{C}{\theta_0}$.

By \cite[Lemma 1.1]{Mimica2019}, it stems that if $\phi$ is a Bernstein function (see Section \ref{sec:6} for the definition)
with the weak scaling property for certain finite  $\theta_0>0$, then it has the same property with $\theta_0'$ larger than
$\theta_0$, the same exponent and modified constant $c$. Therefore with no lose of generality it is enough to consider two cases:
$\theta_0=1$ and $\theta_0=\infty$.

In this article we also need notion of the weak scaling at infinity. A function $f: [A, \infty) \rightarrow (0, \infty)$
has the weak upper scaling property at infinity of index $\beta \in \RR$, denoted as $f \in \WUSCINF{\beta}{c}{\theta_0}$,
if the function
\[
	g(x) = \frac{1}{f\left(\frac{1}{x}\right)}, \qquad x \in (0, A^{-1}],
\]
belongs to $\WUSC{\beta}{c^{-1}}{\theta_0^{-1}}$. Similarly, $f$ belongs to $\WLSCINF{\alpha}{C}{\theta_0}$ if
the function $g$ belongs to $\WLSC{\alpha}{C^{-1}}{\theta_0^{-1}}$.

\subsection{Subordinators}
\label{sec:6}
A smooth function $\varphi: [0, \infty) \rightarrow [0, \infty)$ is called a Bernstein function if for all $n \in \NN$,
\[
	(-1)^n \varphi^{(n)}(u) \leq 0 \qquad\text{for all } u > 0.
\]
In view of \cite[Theorem 3.2]{ssv}, there are constants $a$ and $b$, a positive Radon measure $\nu$ on $[0, \infty)$, so
that $\nu(\{0\}) = 0$, and
\[
	\int_{[0, \infty)} \min\{1, t\} \, \nu({\rm d} t) < \infty,
\]
such that
\[
	\varphi(u) = a + b u + \int_{[0, \infty)} \big(1 - e^{-u t} \big) \, \nu({\rm d} t) \qquad u\geq 0.
\]
We assume that $\varphi(0) = 0$, thus $a=0$. Since $\varphi$ is concave (see e.g. \cite[Lemma 1.1]{Mimica2019})
\begin{equation}
	\label{eq:10}
	\phi(\lambda \theta) \geq \lambda \phi(\theta) \qquad\text{for all } \theta>0 \text{ and } \lambda \in [0, 1].
\end{equation}
The Laplace transform provides one-to-one correspondence between convolution semigroups of probability measures 
$(\nu_t : t\geq 0)$ on $[0,\infty)$ and the set of Bernstein functions. Namely,
\[
	\mathcal{L}\{ \nu_t \}(\lambda) = e^{-t\varphi(\lambda)}, \qquad \lambda \geq 0,\, t > 0.
\]
Let $T_t$ be a random variable with a distribution $\nu_t$. 

There is also a concept of discrete subordination. In this case we additionally assume that $\varphi(1) = 1$, thus
$0 \leq b \leq 1$, and
\[
	0 < \varphi(u) < 1
	\qquad\text{for all } u \in (0, 1).
\]
For $k \in \NN$, we set
\begin{align*}
	c(\varphi, k) 
	&= \frac{1}{k!}| \varphi^{(k)}(1) |\\
	&= \int_{[0, \infty)} t^k e^{-t} \, \nu({\rm d} t) + b\ind{\{1\}}(k).
\end{align*}
Observe that
\[
	\sum_{k=1}^\infty c(\varphi,k)=\varphi(1)=1.
\]
Let us consider a sequence of independent and identically distributed integer-valued random variables $(R_k : k \in \NN)$
also independent of $\mathbf{S}$, and such that $\PP(R_j = k) = c(\phi, k)$. Then for every  $\lambda\geq 0$, we have
\begin{align*} 
	\EE e^{-\lambda R_1}
	&=\sum_{k = 1}^\infty c(\varphi, k) e^{-\lambda k} 
	=\sum_{k = 1}^\infty \frac{1}{k!} e^{-\lambda k}
	\int_{(0, \infty)} t^k e^{-t} \, \nu({\rm d} t)+ b e^{-\lambda} \\
	&= 1 - b - \int_{(0, \infty)} 1 - e^{-t} \, \nu({\rm d} t) + \sum_{k = 1}^\infty \frac{1}{k!} e^{-\lambda k}
	\int_{(0, \infty)} t^k e^{-t} \, \nu({\rm d} t) + b e^{-\lambda} \\
	& = 1 - b - \int_{(0, \infty)} 1 - e^{-t} \sum_{k = 0}^\infty \frac{1}{k!} t^k e^{-\lambda k} \, \nu({\rm d} t)
	+ b e^{-\lambda} \\
	& = 1 - b - \int_{(0, \infty)} 1 - e^{-t} e^{t e^{-\lambda}} \, \nu({\rm d} t) + b e^{-\lambda} \\
	& = 1 - b (1 - e^{-\lambda}) - \int_{(0, \infty)} 
	\Big(1 - e^{-(1 - e^{-\lambda}) t}\Big) \,
	\nu({\rm d} t) \\
	& = 1 - \varphi(1 - e^{-\lambda}).
\end{align*}
Therefore, for $T_n = R_1 + \ldots + R_n$, we obtain
\begin{align}
	\nonumber
	\EE e^{-\lambda T_n}
	&=\left(\EE e^{-\lambda R_1}\right)^n\\
	\label{eq:50}
	&=\left(1-\varphi\left(1-e^{-\lambda}\right)\right)^n, \qquad \lambda \geq 0.
\end{align}

\begin{lemma}
	\label{lem:4}
	For all $C > 0$,
	\begin{equation}
		\label{eq:46}
		\EE\min\big\{1, C T_t\big\} \leq \frac{e}{e-1} \min\big\{ 1, t\varphi(C)\big\}, \qquad t\in\calT.
	\end{equation}
	In particular, we have 
	\begin{equation}
		\label{eq:45}
		\PP(T_t\geq r)\leq \min\left\{1,\frac{e}{e-1} t \varphi(1/r)\right\}, \qquad t\in \mathcal{T},\, r>0.
	\end{equation}
	If we additionally assume that $\varphi$ belongs to $\WUSC{\beta}{c}{1/\inf\mathcal{T}}$ for certain
	$\beta \in [0, 1)$ and $c \in (0, 1]$, then
	\begin{equation}
		\label{eq:44}
		\PP(T_t\geq r)\geq  C\min\left\{1,t \varphi(1/r)\right\}, \qquad t\in \mathcal{T},\, r>0
	\end{equation}
	where $C$ depends on $c$ and $\beta$.
\end{lemma}
\begin{proof}
	Let $U_t(r)=\PP(T_t\geq r)$. We first estimate the Laplace transform of $U_t$. By Tonelli's theorem we have
	\begin{align}
		\lambda \mathcal{L}\{U_t\}(\lambda)
		&=\lambda\int^\infty_0 e^{-\lambda r}\int_{[r,\infty)}\PP(T_t\in {\rm d}u){\: \rm d} r\nonumber\\
		&= \lambda \int^\infty_0 \int_0^ue^{-\lambda r}{\rm d} r \: \PP(T_t\in {\rm d}u)\nonumber\\
		&= 1-\EE e^{-\lambda T_t}.\label{eq:100}
	\end{align}
	In particular, the Laplace transform $\calL \{U_t\}(\lambda)$ is bounded by $1/\lambda$. If $\mathcal{T} = \NN$, 
	the Bernoulli inequality and monotonicity of $\varphi$ easily leads to
	\begin{equation}
		\label{eq:20}
		\lambda\mathcal{L}\{U_t\}(\lambda)\leq \min\{1,t \varphi(\lambda)\}.
	\end{equation}
	In the case when $\mathcal{T}=(0,\infty)$, it is enough to apply the inequality $1-e^{-s}\leq s$. 

	Next, we infer
	\begin{align*}
		\EE\min\big\{1, C T_t\big\}
		&=
		\int_{\mathcal{T}} \min\{1,s C\} \: \PP(T_t\in {\rm d}s)\\
		&\leq  
		\frac{e}{e-1}\int_{\mathcal{T}} \left(1-e^{-sC}\right)\PP(T_t\in {\rm d} u) \\
		&\leq 
		\frac{e}{e-1} t \varphi(C)
	\end{align*}
	where in the last estimate we have used \eqref{eq:100} and \eqref{eq:20} proving \eqref{eq:46}.
	Since $\PP(T_t\geq r) \leq \EE \min\big\{1,T_t/r\big\}$, the inequality \eqref{eq:45} is a direct consequence of 
	\eqref{eq:46}.

	Now, to show the lower estimate \eqref{eq:44}, we it enough to consider $t \vphi(\lambda) \leq 1$. If
	$\mathcal{T}=(0,\infty)$, then
	\[
		\lambda\mathcal{L}\{U_t\}(\lambda) \geq (1-e^{-1}) t\varphi(\lambda). 
	\]
	If $\mathcal{T}=\NN$, then $t \geq 1$ thus $\lambda \leq 1$. Notice that for $k\in\NN$ and $0\leq s\leq k^{-1}$,
	\begin{align*}
		1-(1-s)^k 
		&=
		s\sum^k_{m=1}(1-s)^{m-1} \\
		&\geq 
		ks (1-k^{-1})^{k-1} \\
		&\geq (1-e^{-1})ks.
	\end{align*}
	Hence,
	\begin{align*}
		\lambda\mathcal{L}\{U_t\}(\lambda) 
		&\geq (1-e^{-1}) t\varphi(1-e^{-\lambda}) \\
		&\geq  (1-e^{-1}) t\varphi\big( (1-e^{-1})\lambda\big) \\
		&\geq  (1-e^{-1})^2 t\varphi(\lambda)
	\end{align*}
	where in the last inequality we have used \eqref{eq:10}. Lastly, since $U_t$ is non-increasing, the function
	$\lambda\mapsto \lambda\mathcal{L}\{U_t\}(\lambda)$ is also non-increasing. Therefore, for any $\lambda > 0$ and 
	$t \in \mathcal{T}$,
	\[
		(1-e^{-1})^2 \min\left\{1,t \varphi(\lambda)\right\}
		\leq \lambda\mathcal{L}\{U_t\}(\lambda) \leq \min\left\{1,t \varphi(\lambda)\right\}.
	\]
	By the monotonicity of $U_t$, \cite[Lemma 5]{MR3165234} implies that
	\[
		U_t(r) \leq \min\left\{1,\frac{e}{e-1}t \varphi(1/r)\right\}.
	\]
	For the lower bound we use the scaling property of $\varphi$: We first observe that for $s > 0$ and
	$\lambda \geq 1$,
	\begin{align*}
		\frac{\mathcal{L}\{U_t\}(\lambda s)}{\mathcal{L}\{U_t\}(s)}
		&\leq (1-e^{-1})^{-2} \lambda^{-1}\frac{ \min\{1,t \varphi(\lambda s)\}}{ \min\{1,t \varphi(s)\}}\\
		&=  (1-e^{-1})^{-2} \lambda^{-1}
		\begin{cases}
			1 &  \varphi^{-1}(1/t) \leq s,\\
			\frac{\varphi(\varphi^{-1}(1/t))}{\varphi(s)} &
			\varphi^{-1}(1/t)/\lambda < s < \varphi^{-1}(1/t),\\
			\frac{\varphi(\lambda s)}{\varphi(s)} &   
			0 < s \leq \varphi^{-1}(1/t)/\lambda.
		\end{cases}
	\end{align*}
	If $\varphi^{-1}(1/t)\lambda < s \leq \varphi^{-1}(1/t)$, then by the upper scaling property of $\varphi$, we get
	\begin{align*}
		\frac{\varphi(\varphi^{-1}(1/t))}{\varphi(s)}
		&\leq 
		c^{-1} \left(\frac{\varphi^{-1}(1/t)}{s}\right)^\beta \\
		&< c^{-1}\lambda^\beta.
	\end{align*}
	Analogously, if $s \leq \varphi^{-1}(1/t)/\lambda$ then
	\[
		\frac{\varphi(\lambda s)}{\varphi(s)} \leq c^{-1} \lambda^{\beta}.
	\]
	Hence,
	\[
		\frac{\mathcal{L}\{U_t\}(\lambda s)}{\mathcal{L}\{U_t\}(s)} 
		\leq 
		c^{-1}  (1-e^{-1})^{-2} \lambda^{\beta-1}.
	\]
	Since $\beta-1 < 0$, by \cite[Lemma 13]{MR3165234} we get
	\[
		U_t(r) \geq C_1 \min\{1,t \varphi(1/r)\} \qquad \text{for all } r>0,\, t\in\mathcal{T}
	\]
	where $C_1$ depends only on $c$ and $\beta$. This completes the proof.
\end{proof}

In the following proposition we prove that the exponent in scaling conditions for jump kernels are limited by
natural diffusion on the space.
\begin{proposition}
	\label{prop:1}
	Assume that the L\'{e}vy measure of the subordinator is absolutely continuous with non-increasing density $\nu$.
	Then there exists an absolute constant $c_1 > 0$ such that
	\[
		c(\varphi,k) \geq c_1 \nu(k), \qquad k\in\NN.
	\]
	If we additionally assume that $\nu$ satisfies the lower scaling condition at infinity then there is $c_2 > 0$,
	such that 
	\[
		c(\varphi,k)\leq c_2 \nu(k), \qquad k\in\NN.
	\]
\end{proposition}
\begin{proof}
	Since $\nu$ is non-increasing
	\begin{align*}
		c(\varphi,k)
		&\geq
		\frac{1}{k!} \int^k_0 e^{-t}t^k\nu(t) {\: \rm d} t \\
		&\geq 
		\nu(k) \frac{\Gamma(k+1) -\Gamma(k+1,k)}{\Gamma(k+1)}.
	\end{align*}
	Hence, the first statement is a consequence of \cite[Formula 6.5.35]{MR0167642}. Next, let us assume that
	\[
		\frac{\nu(t)}{\nu(s)} \leq c\left(\frac{t}{s}\right)^\eta, \qquad 1\leq t\leq s,
	\]
	for some $c>0$ and $\eta<0$. Then, for $k>-\eta-1$,
	\begin{align*}
		c(\varphi,k)
		&\leq
		\frac{1}{k!}\left( \int^1_0 e^{-t}t^k\nu(t){\: \rm d} t + 
		c \nu(k) \int_1^k e^{-t}t^k \left(\frac{t}{k}\right)^\eta {\: \rm d} t 
		+ \nu(k) \int^\infty_k e^{-t}t^k {\: \rm d} t\right) \\
		&\leq
		\frac{1}{k!}\left(\varphi(1)+c\nu(k)k^{-\eta}\Gamma(k+\eta+1)+\nu(k)\Gamma(k+1)\right).
	\end{align*}
	Since $\nu(k)\geq c^{-1}\nu(1)k^{\eta}$, the claim is a consequence of the Stirling formula.
\end{proof}
Let us denote by $U$ the potential measure of the subordinator. Namely, for $x \geq 0$,
\[
	\begin{aligned}
	U([0, x]) 
	&= \sum_{n = 0}^\infty \PP(T_n \in [0, x]), &\text{if } \calT &= \NN, \\
	&=
	\int_0^\infty \PP(T_t \in [0, x]) {\: \rm d} t,  &\text{if } \calT &= (0, \infty).
	\end{aligned}
\]
\begin{lemma}
	\label{lem:22}
	We have
	\[
		U\big([0,x]\big) \approx \frac{1}{\varphi(1/x)}
	\]
	uniformly with respect to $x \geq \inf\mathcal{T}$.
\end{lemma}
\begin{proof}
	Thanks to \cite[Proposition III.1]{MR1406564} it is enough to consider $\mathcal{T}=\NN$. Let $F(x)= U([0,x])$, $x\geq 0$.
	By the Tonelli's theorem and \eqref{eq:50}, for $\lambda > 0$ we have
	\begin{align*}
		\lambda \calL \{F\}(\lambda)
		&=
		\sum_{n=0}^\infty \lambda\int^\infty_0 e^{-\lambda x} \int_{[0,x]}\PP(T_n \in {\rm d} u){\: \rm d} x \\
		&= 
		\sum_{n=0}^\infty \EE e^{-\lambda T_n} \\
		&=
		\sum_{n=0}^\infty \big(1-\varphi(1-e^{-\lambda})\big)^n \\
		&=\frac{1}{\varphi(1-e^{-\lambda})}.
	\end{align*}
	Observe that the Markov property implies that $F$ is subadditive, thus $F(\lambda x)\leq 2\lambda F(x)$ for all $x \geq 0$ 
	and $\lambda\geq1$. Hence, by \cite[Propositions 2.2 and 2.3]{MR3098066} we obtain
	\[
		\frac{1}{2(1+e^{-1})} \frac{1}{\varphi(1-e^{-1/x} )} \leq 
		F(x)
		\leq e \frac{1}{\varphi(1-e^{-1/x})},
		\qquad x>0.
	\]
	Since for $x \geq 1$, 
	\[
		\frac{1-e^{-1}}{x}
		\leq 
		1-e^{-\frac{1}{x}}
		\leq \frac{1}{x},
	\]
	by \eqref{eq:10} we get
	\[
		\frac{1}{2(1+e^{-1})} \frac{1}{\varphi(1/x)}
		\leq 
		F(x)
		\leq 
		\frac{e}{1-e^{-1}} \frac{1}{\varphi(1/x)},
		\qquad x\geq 1
	\]
	which completes the proof.
\end{proof}

\section{Subordinate processes}
\label{sec:2}
In this section we concentrate on subordinate processes on a metric measure space $(M, d, \mu)$. Let 
$\mathbf{S} = (S_t : t \in \calT_0)$ be a Markov process and let $(T_t : t \in \calT_0)$ be a subordinator independent
of $\mathbf{S}$ and corresponding to a Bernstein function $\varphi$. Then the subordinate process
$\mathbf{S}^\varphi = (S^\varphi_t : t \in \calT_0)$ is defined as
\[
	 S^\varphi_t=S_{T_t} \qquad\text{for } t\in\mathcal{T}.
\]
In particular, for all $t \in \calT$, $A \in \calB$ and $x \in M$,
\[
	\PP^x(S^\varphi_t \in A) = \int_A \PP^x(S_s \in A) \PP(T_t \in {\rm d} s).
\]
We assume that $\mathbf{S}$ has transition function which is absolutely continuous with respect to $\mu$. Let $p: \calT \times
M \times M \rightarrow [0, \infty)$ denote its density. Then there is a function $p_\varphi: \calT \times M \times M
\rightarrow [0, \infty)$, such that for all $t \in \calT$, $A \in \calB$ and $x \in M$,
\[ 
	\PP^x(S^\varphi_t \in A) = \int_A p_\varphi(t; x, y) \, \mu({\rm d} y) 
	+
	\lim_{\lambda \to +\infty} e^{-t\varphi(\lambda)} \delta_x(A) \ind{\calT=(0,\infty)}.
\]
The aim of this section is to find the sharp lower and upper estimates on $p_\varphi$ under certain
scaling conditions provided that some information about $p(t; \cdot, \cdot)$ is known. Then we study the estimate on the
corresponding Green function.

\subsection{Heat kernel: Upper bounds}

\begin{proposition}
	\label{prop:100}
	Fix $x_0, y_0 \in M$, and assume that there are two positive constants $C$ and $c$ such that
	\begin{equation}
		\label{eq:21}
		p(s; x_0, y_0) \leq C \min\left\{1, c s \right\} \qquad \text{for all } s\in\mathcal{T}.
	\end{equation}
	Then
	\[
		p_\varphi(t; x_0, y_0) \leq 
		\frac{e}{e-1} C \min\left\{1, t \varphi(c) \right\}
		\qquad \text{for all } t \in \mathcal{T}.
	\]
\end{proposition}
\begin{proof}
	In view of Lemma \ref{lem:4}, the statement is a consequence of \eqref{eq:21}.
\end{proof}

In the following lemma we study the subordinate Markov chain obtained from the Markov chain given by the sequence of independent
random variables having the standard normal distribution. 
\begin{lemma}
	\label{lem:7}
	Let $(X_n : n \in \NN)$ be a sequence of independent random variables with the standard normal distribution on $\RR^d$. Let
	$(S_n : n \in \NN_0)$ where $S_n = X_1 + \ldots + X_n$. Assume that $\varphi$ belongs to $\WLSC{\beta}{C}{1}$ for
	certain $\beta \in (0, 1]$ and $C \in [1, \infty)$. Then there is $C_0 > 0$ such that for all $n \in \NN$ and
	$x,y \in \RR^d$,
	\[
		p_{\varphi}(n;x,y) \leq C_0 \big(\varphi^{-1}(n^{-1})\big)^{\frac{d}{2}}.
	\]
\end{lemma}
\begin{proof}
	Using the Fourier inversion formula, we have
	\begin{align*}
		{p}_\varphi(n; x,y) &= (2\pi)^{-\frac{d}{2}}  
		\int_{\RR^d} \bigg(1 - \varphi\Big(1 - e^{-\frac{\norm{\xi}^2}{2}}\Big)\bigg)^n 
		e^{-i \sprod{\xi}{y-x}} {\: \rm d} \xi   \\&\leq(2\pi)^{-\frac{d}{2}}
		\int_{\RR^d} \bigg(1 - \varphi\Big(1 - e^{-\frac{\norm{\xi}^2}{2}}\Big)\bigg)^n 
		 {\: \rm d} \xi.
	\end{align*}
	By \eqref{eq:10}, we get
	\[
		\varphi\Big(1 - e^{-\frac{\norm{\xi}^2}{2}}\Big)
		\geq
		1 - e^{-\frac{\norm{\xi}^2}{2}},
	\]
	thus
	\[
		1 - \varphi\Big(1 - e^{-\frac{\norm{\xi}^2}{2}}\Big)
		\leq
		e^{-\frac{\norm{\xi}^2}{2}}.
	\]
	Hence,	
	\begin{align*}
		\int_{\{\norm{\xi} \geq 1\}}
		\bigg(1 - \varphi\Big(1 - e^{-\frac{\norm{\xi}^2}{2}}\Big)\bigg)^n
		 {\: \rm d} \xi
			&\leq
		\int_{\{\norm{\xi} \geq 1\}}
		e^{-n \frac{\norm{\xi}^2}{2}}
		{\: \rm d}\xi \\
		&\leq
		C_1
		e^{-\frac{n}{4}}.
	\end{align*}
	Next, we estimate the integral in the neighborhood of the origin. If $\norm{\xi} \leq 1$, then
	\[
		1 - e^{-\frac{\norm{\xi}^2}{2}} \geq \frac{\sqrt{e} - 1}{2 \sqrt{e}} \norm{\xi}^2,
	\]
	thus by \eqref{eq:10}
	\[
		\varphi\Big(1 - e^{-\frac{\norm{\xi}^2}{2}}\Big) \geq \frac{\sqrt{e} - 1}{2 \sqrt{e}} \varphi(\norm{\xi}^2).
	\]
	Let
	\[
		a_n = \sqrt{\varphi^{-1}(1/n)}.
	\]
	If $\norm{\xi} \leq a_n$, then again by \eqref{eq:10},
	\begin{align*}
		\varphi\Big(1 - e^{-\frac{\norm{\xi}^2}{2}}\Big) 
		&\geq 
		\frac{\sqrt{e} - 1}{2 \sqrt{e}} \varphi\big(a_n^2 a_n^{-2} \norm{\xi}^2\big) \\
		&\geq
		\frac{\sqrt{e} - 1}{2 \sqrt{e}} n^{-1} a_n^{-2} \norm{\xi}^2. 
	\end{align*}
	If $a_n < \norm{\xi} \leq 1$, then using the weak lower scaling property of $\varphi$ we get
	\[
		n^{-1} = \varphi(a_n^2)  \leq C \big(a_n^2 \norm{\xi}^{-2}\big)^\beta \varphi(\norm{\xi}^2).
	\]
	Hence,
	\begin{align*}
		\varphi\Big(1 - e^{-\frac{\norm{\xi}^2}{2}}\Big)
		&\geq
		\frac{\sqrt{e} - 1}{2C\sqrt{e}} n^{-1} \big(a_n^{-2} \norm{\xi}^2\big)^\beta. 
	\end{align*}
	Therefore,
	\begin{align*}
		\int_{\{\norm{\xi} \leq 1\}}
		\bigg(1 - \varphi\Big(1 - e^{-\frac{\norm{\xi}^2}{2}}\Big)\bigg)^n
	 	{\: \rm d} \xi
		&\leq
		a_n^{-d}
		\int_{\{\norm{\xi} \leq 1/a_n\}}
		\exp\Big(-n \varphi\Big(1 - e^{-\frac{\norm{a_n \xi}^2}{2}}\Big) \Big)
		{\: \rm d} \xi \\
		&\leq
		a_n^{-d}
		\int_{\RR^d}
		\exp\Big(-C_2 \min\big\{\norm{\xi}^2, \norm{\xi}^{2\beta} \big\}\Big)
		{\: \rm d} \xi
	\end{align*}
	and the claim follows, since $\varphi$ decays at most polynomially.
\end{proof}
As an immediate consequence of Lemma \ref{lem:7} we obtain the following corollary.
\begin{corollary}
	\label{cor:3}
 	Assume that $\varphi$ belongs to $\WLSC{\beta}{C}{1}$ for certain $\beta \in (0, 1]$ and $C \in [1, \infty)$.
	Then for each $d\in\NN$ there is $C_d > 0$ such that for all $n \in \NN$,
	\[
		\sum_{k = 1}^\infty \PP(T_n=k) k^{-\frac{d}{2}} \leq C_d \big(\varphi^{-1}(n^{-1})\big)^{\frac{d}{2}}.
	\]
\end{corollary}

\begin{lemma}
	\label{lem:12}
	Assume that $\varphi$ belongs to $\WLSC{\beta}{C}{0}$ for certain $\beta \in (0, 1]$ and $C \in [1, \infty)$. 
	Then for each $d \in \NN$ there is $C_d > 0$ such that for all $t>0$,
	\[
		\int_{[0, \infty)} s^{-\frac{d}{2}}\: \PP(T_t \in {\rm d} s) \leq 
		C_d \big(\varphi^{-1}(t^{-1})\big)^{\frac{d}{2}}.
	\]
\end{lemma}
\begin{proof}
	Let $(S_t : t \geq 0)$ be a Brownian motion on $\RR^d$. Then $(S^\varphi_t : t > 0)$ is a L\'{e}vy process with
	the characteristic exponent $\psi(x)=\varphi(|x|^2)$. In particular, $\psi\in \WLSC{2\beta}{C}{0}$. Since
	\[
		p_\varphi(t; 0,0) = (2\pi)^{-\frac{d}{2}} \int_{[0, \infty)} s^{-\frac{d}{2}}\: \PP(T_t\in {\rm d} s)
	\]
	the lemma is a consequence of \cite[Theorem 3.1]{MR4140542}.
\end{proof}

\begin{proposition}
	\label{prop:101}
	Fix $x_0, y_0 \in M$, and assume that there exists a non-decreasing function $W:[\inf\mathcal{T},\infty) \rightarrow
	[0,\infty)$ belonging to $\WUSCINF{\delta}{C_1}{\inf\mathcal{T}}$ for certain $C_1 \in [1, \infty)$ and $\delta > 0$,
	such that 
	\[
		p(s; x_0, y_0) \leq  \frac{1}{W(s)} \qquad\text{for all } s\in\mathcal{T}.
	\]
	Suppose that $\varphi$ belongs to $\WLSC{\beta}{C_2}{1/\inf\mathcal{T}}$ for some $\beta \in (0, 1]$ and
	$C_2 \in [1, \infty)$. Then there is $C > 0$, such that
	\[
		p_\varphi(t; x_0, y_0) \leq \frac{C}{ W\left(\frac{1}{\varphi^{-1}(t^{-1})}\right)}
		\qquad\text{for all } t \in \calT.
	\]
	The constant $C$ depends only on $C_1$, $C_2$ and $\delta$, $\beta$.
\end{proposition}
\begin{proof} 
	Notice that if $W(r)=0$ for certain $r>0$, then $W \equiv 0$. Therefore, we may and do assume that $W(s) > 0$ for
	all $s \in \mathcal{T}$. By the scaling condition on $W$, for all $s, t \in \mathcal{T}$,
	\begin{align*} 
		W\left(\frac{1}{\varphi^{-1}(t^{-1})}\right) p(s; x_0, y_0) 
		&\leq  
		\frac{W\left(\frac{1}{\varphi^{-1}(t^{-1})}\right)}{W(s)}\\
		&\leq
		C_1\max\left\{1,\left(\frac{1}{s \varphi^{-1}(t^{-1})}\right)^{\delta}\right\}\\
		&\leq
		C_1\left(1+\left(\frac{1}{s \varphi^{-1}(t^{-1})}\right)^{d/2}\right)
	\end{align*}
	where we have set $d = \lceil 2\delta\rceil$. By Corollary \ref{cor:3} in the case $\mathcal{T}=\NN$, or Lemma \ref{lem:12} 
	in the case $\mathcal{T}=(0,\infty)$, we obtain
	\begin{align*}
 		W\left(\frac{1}{\varphi^{-1}(t^{-1})}\right)p_\varphi(t; x_0, y_0)
		&\leq 
		C_1\int_{\mathcal{T}}\left(1+\left(\frac{1}{s \varphi^{-1}(t^{-1})}\right)^{d/2}\right)\PP(T_t\in {\rm d} s) \\
		&\leq  C_1(1+C_d)
	\end{align*}
	proving the lemma.
\end{proof}

\begin{theorem}
	\label{thm:100}
	Fix $x_0, y_0 \in M$ and assume that there exist $c > 0$ and a non-decreasing function
	$W: [\inf\mathcal{T},\infty) \rightarrow [0,\infty)$ belonging to $\WUSCINF{\delta}{C_1}{\inf \mathcal{T}}$ for some
	$C_1 \in [1, \infty)$ and $\delta > 0$, such that
	\[
		p(s; x_0, y_0)
		\leq 
		\min\left\{\frac{1}{W(s)}, s \frac{c}{W(1/c)}\right\} \qquad\text{for all } s\in\mathcal{T}.
	\]
	Suppose that $\varphi$ belongs to $\WLSC{\beta}{C_2}{1/\inf\mathcal{T}}$ for some 
	$\beta \in (0, 1]$ and $C_2 \in [1, \infty)$. Then there is $C \geq 1$, such that
	\[
		p_\varphi(t; x_0, y_0) 
		\leq  
		C \min\left\{\frac{1}{ W\left(\frac{1}{\varphi^{-1}(t^{-1})}\right)},
		\, t \frac{\varphi(c)}{W(1/c)}\right\}
		\qquad\text{for all } t \in \calT.
	\]
	The constant $C$ depends only on $C_1$, $C_2$ and $\delta$, $\beta$.
\end{theorem}
\begin{proof}
	In view of the monotonicity of $W$, we have
	\[
		p(s; x_0, y_0)
		\leq 
		\frac{1}{W(1/c)} \min\left\{1,\,c s \right\} \qquad\text{for all } s\in\mathcal{T}.
	\]
	Hence, the theorem is the consequence of Propositions \ref{prop:100} and \ref{prop:101}.
\end{proof}

\begin{corollary}
	\label{cor:2}
	Assume that there exists a function $W: M \times M \times [\inf \mathcal{T}, \infty) \rightarrow [0, \infty)$,
	such that for certain $\alpha > 0$,
	\[
		p(s; x, y)
		\leq 
		\min\left\{\frac{1}{ W\left(x,y,s^{1/\alpha}\right)},\, \frac{s}{W\left(x,y,d(x,y)\right) (d(x,y))^\alpha}\right\}
		\qquad\text{for all } x, y\in M,\, s \in \calT.
	\]
	Suppose that there are $C_1 \in [1, \infty)$ and $\delta > 0$ such that for each $x, y \in M$, the function
	$W(x, y, \, \cdot \,)$ is non-decreasing and belongs to $\WUSCINF{\delta}{C_1}{\inf \mathcal{T}}$.
	If $\varphi$ belongs to $\WLSC{\beta}{C_2}{1/\inf\mathcal{T}}$ for certain $\beta \in (0, 1]$ and 
	$C_2 \in [1, \infty)$, then there is $C \geq 1$ such that
	\[
		p_\varphi(t; x,y) 
		\leq  
		C \min\left\{\frac{1}{ W\left(x,y,\frac{1}{(\varphi^{-1}(t^{-1}))^{1/\alpha}}\right)},\, 
		\frac{t \varphi\left(\frac{1}{(d(x,y))^\alpha}\right)}{W\left(x,y,d(x,y)\right)}\right\}
		\qquad\text{for all } x,y \in M,\, t \in \calT.
	\]
	The constant $C$ depends only on $C_1$, $C_2$ and $\delta$, $\beta$.
\end{corollary}

\begin{corollary}
	\label{cor:11}
	Assume that there exists a function $W : M \times M \times [\inf \mathcal{T}, \infty) \rightarrow [0, \infty)$,
	such that
	\[
		p(s; x,y) 
		\leq
		\frac{1}{W(x,y,s)}e^{-g\left(\frac{f(d(x,y))}{s}\right)}
		\qquad\text{for all } x, y \in M, \, s \in \calT
	\]
	where $f$ is positive, and $g$ is positive and increasing function belonging to $\WLSCINF{\alpha}{c_1}{\inf \mathcal{T}}$
	for certain $c_1 \in (0, 1]$ and $\alpha > 0$. Suppose that there are $C_2 \in [1, \infty)$ and $\delta > 0$
	such that for every $x, y \in M$, the function $W(x, y, \, \cdot \,)$ is non-decreasing and belongs to 
	$\WUSCINF{\delta}{C_2}{\inf\mathcal{T}}$. If $\varphi$ belongs to $\WLSC{\beta}{C_3}{1/\inf\mathcal{T}}$ for certain
	$\beta \in (0, 1]$ and $C_3 \in [1, \infty)$, then there is $C \geq 1$ such that 
	\[
		p_\varphi(t; x,y)
		\leq
		C\min\left\{\frac{1}{ W\left(x,y,\frac{1}{\varphi^{-1}(t^{-1})}\right)},\, 
		\frac{t \varphi\left(\frac{1}{f(d(x,y))}\right)}{W\left(x,y,f(d(x,y))\right)}\right\}
		\qquad\text{for all } x,y \in M, \, t \in \calT.
	\]
	The constant $C$ depends only on $c_1, C_2, C_3$ and $\delta, \alpha, \beta$.
\end{corollary}
\begin{proof}
	If $ s \leq f(d(x,y))$ then by the upper scaling property of $W$ we have
	\[
		\frac{W(x,y,f(d(x,y)))}{W(x,y,s)} \leq C_2 \left(\frac{f(d(x,y))}{s}\right)^{\delta}.
	\]
	Furthermore, for $u\geq 1$, by the lower scaling property of $g$,
	\begin{align*}
		e^{-g(u)} 
		&\leq e^{-c_1 g(1)u^{\alpha}} \\
		&\leq \frac{C_4}{u^{1+\delta}}
	\end{align*}
	where $C_4 \geq 1$ depends on $c_1$, $g(1)$ and $\delta$. Hence, for $ s \leq f(d(x,y))$,
	\[
		p(s; x,y) \leq C_2 C_4 \frac{s}{W(x,y,f(d(x,y)))f(d(x,y))}.
	\]
	Now, by monotonicity of $W(x,y,\, \cdot \,)$ and $g$,
	\begin{equation}
		\label{eq:101}
		p(s; x,y)
		\leq  
		\min\left\{\frac{1}{(C_2 C_4)^{-1}W(x,y,s)},\frac{s }{(C_2 C_4)^{-1}W(x,y,f(d(x,y)))f(d(x,y))}\right\}
	\end{equation}
	for all $s \in \mathcal{T}$ and $x, y \in M$. To conclude the proof, it is enough to invoke Theorem \ref{thm:100}.
\end{proof}

\subsection{Heat kernel: Lower bounds}
\label{sec:4}

\begin{lemma}
	\label{lem:21}
	Fix $x_0 \in M$ and assume that there exist a non-increasing function $f:(0,\infty) \rightarrow (0,\infty)$, and
	a constant $C_1 \geq 1$ such that
	\[
		\PP^{x_0}(d(S_s, x_0)\geq r) \leq C_1 s f(r) \qquad\text{for all } r > 0, \, s\in \mathcal{T}.
	\]
	Then
	\[
		\PP^{x_0}(d(S^\varphi_t, x_0)\geq r)
		\leq 
		\frac{ e}{e-1} C_1 \,t\, \varphi\left(f(r)\right)
		\qquad\text{for all } r>0, \, t\in\mathcal{T}.
	\]
\end{lemma}
\begin{proof}
	The statement easily follows from Lemma \ref{lem:4}.
\end{proof}

\begin{theorem}
	\label{thm:101}
	Fix $x_0 \in M$, and assume that there exist a non-increasing function $f:(0,\infty) \rightarrow (0, 1/\inf \calT)$,
	and a constant $C_1 \geq 1$ such that
	\begin{equation}
		\label{eq:28}
		\PP^{x_0}(d(S_s, x_0) \geq r) \leq C_1 s f(r) \qquad \text{for all } r > 0, \, s \in \mathcal{T}.
	\end{equation}
	Suppose that $f$ belongs to $\WUSCINF{-\delta}{C_2}{0}$ for certain $C_2 \in [1, \infty)$ and $\delta > 0$. Let
	$g: \mathcal{T}\times (0,\infty) \rightarrow [0,\infty)$ be such that for all $t \in \mathcal{T}$, the function
	$g(t, \cdot)$ is non-increasing and satisfies
	\begin{equation}
		\label{eq:22}
		p(s; x_0, y) \geq g(s, d(x_0,y) ) \qquad \text{for all } y \in M, \, s\in\mathcal{T},
	\end{equation}
	and for certain $c_3 > 0$,
	\begin{equation}
		\label{eq:29}
		\int_{B(x_0,r)^c} g(s, d(x_0,y)) \, \mu({\rm d} y) \geq c_3 \qquad\text{whenever } s f(r)\geq 1.
	\end{equation}
	If $\varphi$ belongs to $\WUSC{\beta_2}{c_4}{1/\inf\mathcal{T}} \cap
	\WLSC{\beta_1}{C_5}{1/\inf\mathcal{T}}$ for some $1 > \beta_2 > \beta_1 > 0$, $c_4 \in (0, 1]$ and $C_5 \in [1, \infty)$,
	then there are $c > 0$ and $\lambda > 1$, such that
	\[
		p_\varphi(t; x_0, y)
		\geq 
		c \min\left\{\frac{1}{V(x_0,\lambda f^{-1}(\varphi^{-1}(1/t)))}, 
		\,\frac{t\, \varphi\left(f(d(x_0, y))\right)}{V(x_0, \lambda d(x_0,y))}\right\}
		\qquad\text{for all } y \in M,\, t\in\mathcal{T}.
	\]
	The constants $c$ and $\lambda$ depends only on $C_1, C_2, c_3, c_4, C_5$ and $\alpha, \beta$.
\end{theorem}
\begin{proof} 
	Let
	\[
		\tilde{g}(t,r) = \int_{\mathcal{T}} g(s, r) \: \PP(T_t \in {\rm d} s), \qquad t\in\mathcal{T}, \, r\geq 0.
	\]
	For $r > 0$ and $\lambda > 1$, we define $A=B(x_0, \lambda r)\setminus B(x_0, r)$. Since $g(s, \, \cdot \,)$ is 
	non-increasing, by \eqref{eq:22} and \eqref{eq:29}, we have
	\begin{align*}
		\mu(A) \tilde{g}(t,r)
		&\geq 
		\int_{\mathcal{T}}\int_{A} g(s, d(x_0, y)) \, \mu({\rm d} y) \: \PP(T_t \in {\rm d} s)\\
		&= 
		\int_{\mathcal{T}}\int_{B(x_0, r)^c} g(s, d(x_0, y)) \, \mu({\rm d} y) \: \PP(T_t \in {\rm d} s) \\
		&\phantom{=}- \int_{\mathcal{T}}\int_{B(x_0, \lambda r)^c} g(s, d(x_0, y)) \, \mu({\rm d} y)\: \PP(T_t\in {\rm d} s) \\
		&\geq
		c_3 \PP(T_t\geq 1/f(r))-\PP^{x_0}(d(S^\varphi_t, x_0)\geq \lambda r).
	\end{align*}
	Hence, thanks to the upper scaling property of $\varphi$, by Lemmas \ref{lem:4} and \ref{lem:21}, we obtain
	\[
		\mu(A) \tilde{g}(t,r)
		\geq  
		c_6\min\left\{1,t\varphi\big(f(r)\big)\right\} -c_7\frac{e}{e-1}t\varphi\big(f(\lambda r)\big).
	\]
	Now, let us consider the case when $t\varphi(f(r))\leq 1$. Since $\varphi \in \WLSC{\beta_1}{C_5}{1/{\inf \mathcal{T}}}$
	and $f \in \WUSCINF{-\delta}{C_2}{0}$, we obtain
	\begin{align*}
		\tilde{g}(t,r)
		&\geq \frac{t\varphi(f(r))}{\mu(A)}\left(c_6-c_7 \frac{e}{e-1}\frac{\varphi(f(\lambda r))}{\varphi(f( r))}\right)\\
		&\geq \frac{t\varphi(f(r))}{\mu(A)}\left(c_6-c_7 C_5 \frac{e}{e-1}\left(\frac{f(\lambda r)}{f( r)}
		\right)^{\beta_1}\right)\\
		&\geq \frac{t\varphi(f(r))}{\mu(A)}\left(c_6-c_7 C_5 \frac{e}{e-1}\left(C_2 \lambda^{-\delta}
		\right)^{\beta_1}\right).
	\end{align*}
	Taking
	\[
		\lambda = \max\left\{1, \left(\frac{c_6 (1-e^{-1})}{2 c_7 C_5 C_2^{\beta_1}}\right)^{-1/(\beta_1\delta)} \right\},
	\]
	we get
	\begin{align}
		\nonumber
		\tilde{g}(t,r)
		&\geq  \frac{c_6}{2}\frac{t\varphi(f(r))}{\mu(A)} \\
		\label{eq:24}
		&\geq \frac{c_6}{2}\frac{t\varphi(f(r))}{V(x_0, \lambda r)}.
	\end{align}
	If $t \varphi(f(r)) > 1$, then $0 < r < r_t$ where $r_t = f^{-1}(\varphi^{-1}(1/t))$. Hence, by the monotonicity of 
	$g(t, \,\cdot\,)$ we obtain
	\begin{align}
		\nonumber
		\tilde{g}(t,r)
		&\geq  
		\tilde{g}(t,r_t) \\
		\label{eq:23}
		&\geq 
		\frac{c_6}{2} \frac{1}{V(x_0, \lambda r_t)}.
	\end{align}
	By \eqref{eq:22}, we easily see that $p_\varphi(t; x_0, y) \geq \tilde{g}(t, d(x_0, y))$ which together with
	\eqref{eq:24} and \eqref{eq:23} concludes the proof.
\end{proof}

\subsection{Heat kernel estimates: Conclusion}

\begin{theorem}
	\label{thm:1}
	Assume that there exist $C_1 \geq 1$ and $\gamma_2 \geq \gamma_1 > 0$ such that for all $x\in M$, $\lambda > 1$ and
	$r>\frac{1}{2}\inf \mathcal{T}$,
	\begin{equation}
		\label{eq:27}
		C_1^{-1} \lambda^{\gamma_1} \leq \frac{V(x,\lambda r)}{V(x,r)}\leq C_1 \lambda^{\gamma_2}.
	\end{equation}
	Suppose that there exists $C_2 \geq 1$ such that for all $x,y\in M$ and $s\in \mathcal{T}$,
	\[
		C_2^{-1}
		\frac{1}{V(x,f^{-1}(t))} 
		\ind{\{f(d(x,y))\leq s\}} 
		\leq 
		p(s; x,y)
		\leq
		C_2
		\frac{1}{V(x,f^{-1}(s))} e^{-g\left(\frac{f(d(x,y))}{s}\right)}
	\]
	where $f$ is positive increasing function belonging to $\WLSCINF{\delta_1}{c_3}{0} \cap \WUSCINF{\delta_2}{C_4}{0}$
	for certain $c_3 \in (0, 1]$, $C_4 \in [1, \infty)$, and $\delta_2 \geq \delta_1 > 0$, and $g$ is positive increasing
	function belonging to $\WLSCINF{\alpha}{c_5}{\inf \calT}$ for some $c_5 \in (0, 1]$ and $\alpha > 0$. If $\varphi$ belongs
	to $\WLSC{\beta_1}{C_6}{1/\inf\mathcal{T}} \cap \WUSC{\beta_2}{c_7}{1/\inf\mathcal{T}}$ for some
	$1 > \beta_2 > \beta_1 > 0$, $C_6 \in [1, \infty)$ and $c_7 \in (0, 1]$, then
	\[
		p_\varphi(t; x,y) 
		\approx
		\min\left\{\frac{1}{V\left(x,f^{-1}\left(\frac{1}{\varphi^{-1}(1/t)}\right)\right)},
		t \frac{\varphi\left(\frac{1}{f(d(x,y))}\right)}{V(x,d(x,y))}\right\}
	\]
	uniformly with respect to $x, y \in M$ and $t \in \mathcal{T}$. 
\end{theorem}
\begin{proof}
	For $x, y \in M$ and $s \in \mathcal{T}$ we set
	\[
		W(x, y, s) = C_2^{-1} V\big(x, f^{-1}(s)\big).
	\]
	Since $f$ is increasing and belongs to $\WLSCINF{\delta_1}{c_3}{0}$, its inverse is increasing function which
	belongs to $\WUSCINF{1/\delta_1}{c_3'}{0}$. Hence, to get the upper estimate it is enough to invoke Corollary \ref{cor:11}.

	For the proof of the lower estimate, we fix $x_0 \in M$. We are going to apply Theorem \ref{thm:101} for
	\[
		g(s, r) = \frac{1}{C_2 V(x_0, f^{-1}(s))} \ind{\{f(r) \leq s\}}, \qquad s \in \mathcal{T}, \, r > 0,
	\]
	and
	\[
		\tilde{f}(r) = \frac{1}{f((2C_1)^{1/\gamma_1} r)}, \qquad r > 0.
	\]
	Let us start by checking \eqref{eq:29}. Suppose that $s \tilde{f}(r) \geq 1$. If $r > \inf \calT$, then by \eqref{eq:27},
	\begin{align*}
		\int_{B(x_0, r)^c} C_2 g(s,d(x_0, y)) \, \mu({\rm d} y)
		&\geq 1-\frac{V(x_0, r)}{V(x_0, f^{-1}(s))} \\
		&\geq 1-C_1\left(\frac{r}{f^{-1}(s)}\right)^{\gamma_1}\geq \frac{1}{2}
	\end{align*}
	which can be extended to all $r > 0$ with a help of monotonicity. 

	To show \eqref{eq:28}, it is enough to consider $s \tilde{f}(r) \leq \tilde{f}(1)$. If $\calT = \NN$, then we easily see
	that $r \geq 1$. Hence, in both cases $r > \frac{1}{2} \inf \calT$. Now, by \eqref{eq:101}
	\begin{align*}
		\PP^{x_0}(d(S_s, x_0) > r) 
		&=
		\int_{B(x_0, r)^c} p(s; x_0, y) \, \mu({\rm d} y) \\
		&
		\leq 
		C_8 \int_{B(x_0, r)^c} \frac{s}{V(x_0, d(x_0, y))f(d(x_0, y))} \, \mu({\rm d} y)\\
		&\leq 
		C_8 s \sum_{n=0}^\infty \int_{B(x_0, 2^{n+1}r)\setminus B(x_0, 2^nr)} \frac{1}{V(x_0, 2^n r) f(2^{n}r)}
		\, \mu({\rm d} y).
	\end{align*}
	Since
	\[
		\frac{V(x_0, 2^{n+1}r )}{V(x_0, 2^n r)} 
		\leq 
		C_1 2^{\gamma_2}
	\]
	and
	\[
		\frac{f((2C_1)^{1/\gamma_2} r)}{f(2^n r)} \leq (2C_1)^{\delta_1/\gamma_2} c_3^{-1} 2^{-n \delta_1},
	\]
	we get
	\begin{align}
		\PP^{x_0}(d(S_s, x_0) > r)  \nonumber
		&\leq 
		2^{\gamma_2+\delta_1/\gamma_2} C_1^{1+\delta_1/\gamma_2} c_3^{-1} C_8 s \tilde{f}(r) 
		\sum_{n=0}^\infty 2^{-n \delta_1} \\
		&\leq
		C_9 s \tilde{f}(r)\label{eq:tg1}
	\end{align}
	proving \eqref{eq:28}. 
	
	Having checked \eqref{eq:28}--\eqref{eq:29}, we are in the position to use Theorem \ref{thm:101}. We obtain
	\[
		p_\varphi(t; x_0, y)
		\geq 
		c 
		\min\left\{\frac{1}{V(x_0, \lambda \tilde{f}^{-1}(1/\varphi^{-1}(1/t)))}, 
		\,\frac{t\, \varphi\left(\tilde{f}(d(x_0, y))\right)}{V(x_0, \lambda d(x_0,y))}\right\}
	\]
	for certain $\lambda > 0$ and $c > 0$ which are independent of $x_0$. To complete the proof it is enough to use
	scaling properties of $V$, $f$ and $\varphi$.
\end{proof}
We show analogously the following theorem.
\begin{theorem}
	\label{thm:2}
	Assume that there are $C_1 \geq 1$ and $\gamma_2 \geq \gamma_1 > 0$ such that for all $x \in M$, $\lambda > 1$ and
	$r > \frac{1}{2} \inf \calT$,
	\[
		C_1^{-1} \lambda^{\gamma_1} \leq \frac{V(x, \lambda r)}{V(x, r)} \leq C_1 \lambda^{\gamma_2}.
	\]
	Suppose that there are $C_1 > 0$ and $\alpha > 0$ such that 
	\begin{equation}
		\label{eq:31}
		C_2^{-1} \frac{1}{V(x, s^{1/\alpha})} \ind{\{(d(x, y))^\alpha \leq s\}}
		\leq
		p(s; x, y) 
		\leq 
		C_2 \min\left\{\frac{1}{V(x, s^{1/\alpha})}, \frac{s}{V(x, d(x, y)) (d(x, y))^\alpha} \right\}
	\end{equation}
	for all $s \in \calT$ and $x, y \in M$.
	If $\varphi$ belongs to $\WLSC{\beta_1}{C_3}{1/\inf\mathcal{T}} \cap \WUSC{\beta_2}{c_4}{1/\inf\mathcal{T}}$
	for some $1 > \beta_2 \geq \beta_1 > 0$, $C_3 \in [1, \infty)$ and
	$c_4 \in (0, 1]$, then
	\[
		p_\varphi(t; x,y) 
		\approx
		\min\left\{\frac{1}{V\left(x, (\varphi^{-1}(t^{-1}))^{-1/\alpha}\right)}
		, 
		\frac{t \varphi\left((d(x, y))^{-\alpha}\right)}{V(x, d(x, y))}
		\right\}
	\]
	uniformly with respect to $t \in \calT$ and $x, y \in M$.
\end{theorem}
\begin{proof}
	For $x, y \in M$ and $s \in \calT$, we set
	\[
		W(x, y, s) = C_2^{-1} V(x, s).
	\]
	Then $W(x, y, \cdot)$ is non-decreasing function belonging to $\WUSCINF{\gamma_2}{C_1}{\inf \calT}$. Hence, the upper
	estimate is a consequence of Corollary \ref{cor:2}.

	To show the lower estimate we fix $x_0 \in M$ and use Theorem \ref{thm:101} with
	\[
		g(s, r) = \frac{1}{C_2 V(x_0, s^{1/\alpha})} \ind{\{r^\alpha \leq s \}},
		\qquad s \in \calT, \, r > 0,
	\]
	and
	\[
		\tilde{f}(r) = (2C_1)^\frac{\alpha}{\gamma_1} r^\alpha, \qquad r > 0.
	\]
	To check \eqref{eq:29} we reason as in the proof of Theorem \ref{thm:1}. To show \eqref{eq:28}, it is enough to consider
	$s \tilde{f}(r) \leq \tilde{f}(1)$. If $\calT = \NN$, then we easily
	see that $r \geq 1$. Hence, in both cases $r > \frac{1}{2} \inf \calT$. Now, by \eqref{eq:31}
	\begin{align*}
		\PP^{x_0}(d(S_s, x_0) > r) 
		&=
		\int_{B^c(x_0, r)} p(s; x_0, y) \, \mu({\rm d} y) \\
		&
		\leq 
		C_5 \int_{B^c(x_0, r)} \frac{s}{V(x_0, d(x_0, y))(d(x_0, y))^\alpha} \, \mu({\rm d} y)\\
		&\leq
		C_5 \frac{s}{r^\alpha} 
		\sum_{n=0}^\infty \int_{B(x_0, 2^{n+1}r)\setminus B(x_0, 2^nr)} \frac{1}{V(x_0, 2^n r)} 2^{-\alpha n}
		\, \mu({\rm d} y).
	\end{align*}
	Since
	\[
		\frac{V(x_0, 2^{n+1}r )}{V(x_0, 2^n r)} 
		\leq 
		C_1 2^{\gamma_2}
	\]
	we get
	\begin{align*}
		\PP^{x_0}(d(S_s, x_0) > r) 
		&\leq 
		C_1 C_5 
		2^{\gamma_2}
		\frac{s}{r^\alpha}
		\sum_{n=0}^\infty 2^{-n \alpha} \\
		&\leq
		C_6 s \tilde{f}(r)
	\end{align*}
	proving \eqref{eq:28}. 

	Now by Theorem \ref{thm:101}, there are $c > 0$ and $\lambda > 0$ such that
	\[
		p_\varphi(t; x_0, y)
		\geq 
		c 
		\min\left\{\frac{1}{V(x_0, \lambda \tilde{f}^{-1}(1/\varphi^{-1}(1/t)))}, 
		\,\frac{t\, \varphi\left(\tilde{f}(d(x_0, y))\right)}{V(x_0, \lambda d(x_0,y))}\right\}
	\]
	for all $t \in \calT$ and $y \in M$. Since $c$ and $\lambda$ are independent of $x_0$, to finish the proof it is enough
	to invoke scaling properties of $V$ and $\varphi$. 
\end{proof}

\subsection{Equivalence of estimates}
In this section we show a consequence of heat kernel estimates of the form as in Theorems \ref{thm:1} and \ref{thm:2}. 
\begin{lemma}
	\label{lem:301}
	Fix $x_0 \in M$ and assume that there exist an increasing function $f: [0,\infty) \rightarrow [0,\infty)$ and
	a constant $C_1 \geq 1$ such that 
	\[
		\PP^{x_0}(d(S_s, x_0)\geq r) \leq C_1 \frac{s^2}{ f^2(r)} \qquad \text{for all } r > 0,\, s\in \mathcal{T}.
	\]
	Then
	\[
		\EE^{x_0} \big[ f(d(x_0,S_t^\varphi))\ind{B(x_0,r)}(S^\varphi_t) \big] 
		\leq C_1 \frac{2e}{e-1}t f(r)\varphi\big(1/f(r)\big)
		\qquad \text{for all } r > 0, \, t \in \calT.
	\]
\end{lemma}
\begin{proof}
	We have
	\begin{align*}
		\EE^{x_0} \big[ f(d(x_0,S_t^\varphi)) \ind{B(x_0,r)}(S^\varphi_t) \big]
		&=
		\EE^{x_0} \bigg[ \ind{B(x_0,r)}(S^\varphi_t) 
		\int_0^{f(d(x_0,S_t^\varphi))} {\: \rm d} u\bigg] \\
		&= 
		\int_0^{f(r)} \PP^{x_0} \big(f^{-1}(u) \leq d(S^\varphi_t, x_0)\leq r\big) {\: \rm d} u\\
		&=
		\int_{[0,\infty)} \int_0^{f(r)} \PP^{x_0}\big(f^{-1}(u) \leq d(S_s, x_0)\leq r\big) {\: \rm d} u \: 
		\PP(T_t\in {\: \rm d}s).
	\end{align*}
	Now, the claim is a consequence of Lemma \ref{lem:4} and the following estimate
	\begin{align}
		\int_0^{f(r)} \PP^{x_0} \big(f^{-1}(u) \leq d(S_s, x_0)\leq r\big) {\: \rm d} u 
		&
		\leq C_1 \int_0^{f(r)} \left(1\wedge\frac{s^2}{u^2}\right) {\: \rm d} u \nonumber\\
		&
		\leq 2 C_1  (f(r)\wedge s ). \label{eq:tg3}
	\end{align}
\end{proof}

\begin{theorem}
	\label{thm:6}
	Assume that there exist $C_1 \geq 1$ and $\gamma_2 \geq \gamma_1 > 0$ such that for all $x\in M$, $\lambda > 1$ and
	$r > \frac{1}{2}$,
	\[
		C_1^{-1} \lambda^{\gamma_1} \leq \frac{V(x,\lambda r)}{V(x,r)}\leq C_1 \lambda^{\gamma_2}.
	\]
	Suppose that there exists $C_2 \geq 1$ such that for all $x,y\in M$ and $n \in \NN$,
	\[
		C_2^{-1}
		\frac{1}{V(x,f^{-1}(n))} 
		\ind{\{f(d(x,y))\leq n\}} 
		\leq 
		p(n; x,y)
		\leq
		C_2
		\frac{1}{V(x,f^{-1}(n))} e^{-g\left(\frac{f(d(x,y))}{n}\right)}
	\]
	where $f$ is positive increasing function belonging to $\WLSCINF{\delta_1}{c_3}{0} \cap \WUSCINF{\delta_2}{C_4}{0}$
	for certain $c_3 \in (0, 1]$, $C_4 \in [1, \infty)$, and $\delta_2 \geq \delta_1 > 0$, and $g$ is positive increasing
	function belonging to $\WLSCINF{\alpha}{c_5}{1}$ for some $c_5 \in (0, 1]$ and $\alpha > 0$. Let $\vphi$ be a Bernstein
	function such that $\vphi(0) = 0$ and $\vphi(1) = 1$. The following are equivalent:
	\begin{enumerate}[i)]
		\item
		\label{thm:6:1}
		$\vphi$ belongs to $\WLSC{\beta_1}{C_7}{1} \cap \WUSC{\beta_2}{c_6}{1}$ for certain $1 > \beta_2 \geq \beta_1 > 0$,
		$c_6 \in (0, 1]$ and $C_7 \in [1, \infty)$;
		\item we have
		\label{thm:6:2}
		\[
			p_\varphi(n; x,y)
			\approx
			\min\left\{\frac{1}{V\left(x,f^{-1}\left(\frac{1}{\varphi^{-1}(1/n)}\right)\right)},
			n \frac{\varphi\left(\frac{1}{f(d(x,y))}\right)}{V(x,d(x,y))}\right\}
		\]
		uniformly with respect to $n \in \NN$ and $x, y \in M$;
		\item we have
		\label{thm:6:3}
		\[
			p_\varphi(1; x,y) 
			\approx
			\min\left\{\frac{1}{V(x, f^{-1}(1))},
			\frac{\varphi\left(\frac{1}{f(d(x,y))}\right)}{V(x,d(x,y))}\right\}
		\]
		uniformly with respect to $x, y \in M$.
	\end{enumerate}
\end{theorem}
\begin{proof}
	By Theorem \ref{thm:1}, \eqref{thm:6:1} implies \eqref{thm:6:2}. Since \eqref{thm:6:2} implies \eqref{thm:6:3}, it is enough
	to show that \eqref{thm:6:3} leads to \eqref{thm:6:1}. Let 
	\[
		w(r)=\min\left\{\frac{1}{V(x,f^{-1}(1))},\frac{\varphi(1/f(r))}{V(x,r)}\right\}, \qquad r>0.
	\]
	By the reverse doubling property of $V$ there exists $\lambda>1$ such that $V(x,\lambda r) - V(x,r)\geq V(x,r)$.
	Therefore, for $r \geq \max\{f^{-1}(1),1/2\}$,
	\begin{align*}
		w(\lambda r)V(x,r) 
		&\leq 
		w(\lambda r)\mu\big(B(x,\lambda r)\setminus B(x,r)\big)\\
		&\leq 
		\int_{B(x,r)^c} 
		w(d(x,y)) \: \mu({\rm d}y) \\
		&\leq
		c_1 \PP\big(|S^\varphi_1|>r\big) \\
		&\leq  
		c_2 \varphi(1/f(r))\\
		&\leq 
		c_3 w(r) V(x,r).
	\end{align*}
	Hence, by \eqref{eq:10} and the doubling property of $V$ and $f$, 
	\[
		w(r)V(x,r) \approx \int_{B(x,r)^c} w(d(x,y)) \, \mu({\rm d} y)
	\]
	uniformly with respect to $x \in M$ and $r \geq \max\{f^{-1}(1),1/2\}$. Furthermore 
	\begin{align*}
		\int_{B(x,r)^c} w(d(x,y)) \mu({\rm d} y)
		&\approx 
		\sum_{k \geq 0} 
		w(\lambda^k r) \mu\big(B(x,\lambda^k r)\setminus B( x,\lambda^{k-1}r)\big)\\
		&\approx 
		\sum_{k \geq 0} 
		w(\lambda^k r) V(x,\lambda^k r)\\
		&\approx 
		\int_r^\infty w(u)V(x,u)\frac{{\rm d} u}{u}.
	\end{align*}
	That is
	\begin{equation}
		\label{eq:41}
		w(r)V(x,r) \approx \int_r^\infty w(u)V(x,u)\frac{{\rm d}u}{u}.
	\end{equation}
	In view of \cite[Proposition A.1]{MR3729529}, the function $r \mapsto \varphi(1/f(r))=w(r)V(x,r)$ satisfies
	the upper scaling condition at infinity with a negative exponent which implies that $\vphi$ has lower scaling at zero
	with a positive exponent due to the doubling of $f$ and \eqref{eq:10}.

	Similarly, using Lemma \ref{lem:301}, for $r\geq \lambda \max\{f^{-1}(1),1/2\}$,
	\begin{align}
		w(r)V(x,r)f(r) 
		&\approx 
		\int^r_0 f(s)w(s)V(x,s)\frac{{\rm d} s}{s} \nonumber\\ 
		&\approx  
		\EE \Big[ f(d(x,S_1^\varphi)) \ind{\{d(x,S_1^\varphi)<r\}}\Big].\label{eq:tg2}
	\end{align}
	By \cite[Proposition A.1]{MR3729529}, $\varphi(1/f(r))f(r)=w(r)V(x,r)f(r)$ satisfies the weak lower scaling property
	at infinity with positive index, that is $\varphi$ satisfies the weak upper scaling at $0$ with index strictly smaller
	than $1$.
\end{proof}

\begin{theorem}
	\label{thm:7}
	Assume that there exist $C_1 \geq 1$ and $\gamma_2 \geq \gamma_1 > 0$ such that for all $x\in M$, $\lambda > 1$, and
	$r > 0$,
	\[
		C_1^{-1} \lambda^{\gamma_1} \leq \frac{V(x,\lambda r)}{V(x,r)}\leq C_1 \lambda^{\gamma_2}.
	\]
	Suppose that there exists $C_2 \geq 1$ such that for all $x,y\in M$ and $s > 0$,
	\[
		C_2^{-1}
		\frac{1}{V(x,f^{-1}(s))} 
		\ind{\{f(d(x,y))\leq s\}} 
		\leq 
		p(s; x,y)
		\leq
		C_2
		\frac{1}{V(x,f^{-1}(s))} e^{-g\left(\frac{f(d(x,y))}{s}\right)}
	\]
	where $f$ is positive increasing function belonging to $\WLSCINF{\delta_1}{c_3}{0} \cap \WUSCINF{\delta_2}{C_4}{0}$
	for certain $c_3 \in (0, 1]$, $C_4 \in [1, \infty)$, and $\delta_2 \geq \delta_1 > 0$, and $g$ is positive increasing
	function belonging to $\WLSCINF{\alpha}{c_5}{1}$ for some $c_5 \in (0, 1]$ and $\alpha > 0$. Let $\vphi$ be a Bernstein
	function such that $\vphi(0) = 0$. The following are equivalent:
	\begin{enumerate}[i)]
		\item
		\label{thm:7:1}
		$\vphi$ belongs to $\WLSCINF{\beta_1}{c}{0} \cap \WUSCINF{\beta_2}{C}{0}$ for certain $1 > \beta_2 \geq \beta_1 > 0$,
		$c \in (0, 1]$ and $C \in [1, \infty)$;
		\item 
		\label{thm:7:2}
		\[
			p_\varphi(t; x,y)
			\approx
			\min\left\{\frac{1}{V\left(x,f^{-1}\left(\frac{1}{\varphi^{-1}(1/t)}\right)\right)},
			t \frac{\varphi\left(\frac{1}{f(d(x,y))}\right)}{V(x,d(x,y))}\right\}
		\]
		uniformly with respect to $t > 0$  and $x, y \in M$;
		\item 
		\label{thm:7:3}
		\[
			\int_{(0, \infty)} p(s; x, y) \, \nu({\rm d} s)
			\approx
			\frac{\varphi\left(\frac{1}{f(d(x,y))}\right)}{V(x,d(x,y))}
		\]
		uniformly with respect to $x, y \in M$.
	\end{enumerate}
\end{theorem}
\begin{proof}
	By Theorem \ref{thm:1}, \eqref{thm:7:1} implies \eqref{thm:7:2}. Next, let us observe that by \cite[Theorem 30.1]{Sato1999},
	we have
	\[
		\lim_{t \to 0} \frac{1}{t} p_\vphi(t; x, y) = \int_{(0, \infty)} p(s; x, y) \, \nu({\rm d} s).
	\]
	On the other hand, by \eqref{thm:7:2}, for small $t$,
	\[
		\frac{1}{t} p_\vphi(t; x, y) \approx
				\frac{\varphi\left(\frac{1}{f(d(x,y))}\right)}{V(x,d(x,y))}	
	\]
	which implies \eqref{thm:7:3}. Finally, let us assume that \eqref{thm:7:3} is satisfied. We proceed analogously to the proof
	of Theorem \ref{thm:6}. It suffices to justify \eqref{eq:41}, and $\eqref{eq:tg2}$ with
	\[
		w(r) = \frac{\varphi\left(\frac{1}{fr)}\right)}{V(x,r)}.
	\]
	By \eqref{eq:tg1}, we have
	\begin{align*}
		w(\lambda r) V(x, r) 
		&\leq w(\lambda r) \mu\big(B(x, \lambda r) \setminus B(x, r) \big) \\
		&\leq \int_{B(x, r)^c} w(d(x, y)) \, \mu({\rm d} y) \\
		&\approx \int_{(0,\infty)}\PP^x(d(S_s,x)>r)\, \nu({\rm d} s)\\
		&\leq c\varphi(1/f(r))
	\end{align*}
	which leads to \eqref{eq:41}. Similarly, we prove \eqref{eq:tg2} using \eqref{eq:tg3}. The rest of the proof is analogous.
\end{proof}

\subsection{Green function estimates}
In this section we study a behavior of the Green function $G_\varphi$ for the subordinate process
$(S^\varphi_t : t > 0)$, that is
\begin{align*}
	G_\varphi(x,y) 
	&= 
	\int_{\mathcal{T}}p^\varphi(t; x,y) \: {\rm d} t\\
	&= \int_{[0,\infty)}p(s; x,y) U({\rm d} s), \qquad x\neq y
\end{align*}
where $U$ is the potential measure of subordinator.

\begin{proposition}
	\label{prop:102}
	Fix $x_0, y_0 \in M$, $x_0 \neq y_0$, and assume that there exist $t \in \calT$ and a non-decreasing function
	$W: [\inf \calT, \infty) \rightarrow [0, \infty)$ belonging to $\WLSCINF{\delta}{c_1}{\inf \calT}$ for certain
	$c_1 \in (0, 1]$ and $\delta > 1$, such that
	\[
		p(s; x_0, y_0) \leq \min\left\{\frac{1}{W(s)}, \frac{1}{W(t)} \right\} \qquad\text{for all } s \in \calT.
	\]
	Then there is $c > 0$ such that
	\[
		G_\varphi(x_0, y_0) \leq c \frac{1}{W(t) \varphi(1/t)}.
	\]
	The constant $c$ depends only on $\delta$ and $c_1$.
\end{proposition}
\begin{proof}
	We write
	\begin{align*}
		W(t) G_\varphi(x_0, y_0)
		&=
		W(t) \int_{(0,\infty)} p(s; x_0, y_0) \: U({\rm d} s)\\
		&\leq  
		\int_{(0,\infty)} \min\left\{1, \frac{W(t)}{W(s)}\right\} \: U({\rm d} s)\\
		&\leq 
		c_1^{-1} \int_{(0,\infty)} \min\left\{ 1, \left(\frac{t}{s} \right)^\delta\right\}\: U({\rm d} s)\\
		&=
		\delta c_1^{-1} t^\delta \int^\infty_t \frac{U((0,u])}{u^{\delta+1}} {\: \rm d} u.
	\end{align*}
	Hence, by Lemma \ref{lem:22},
	\[
		W(t) G_\varphi(x_0, y_0)
		\leq 
		C_2 t^\delta \int^\infty_t \frac{1}{\varphi(1/s)s^{\delta+1}} {\: \rm d} s
	\]
	which together with \eqref{eq:10} completes the proof.
\end{proof}

\begin{remark}
	\label{rem:tg1}
	If $\varphi$ belongs to $\WUSC{\alpha}{C}{1/\inf\mathcal{T}}$ for certain
	$C \in [1, \infty)$ and $\alpha <1$, then the claim of Proposition \ref{prop:102} holds for $\delta>\alpha$.
\end{remark}

\begin{proposition}
	\label{prop:103}
	Assume that there are $C_1 > 0$ and $\gamma > 0$ such that for all $x \in M$,  $\lambda > 1$, and
	$r > \frac{1}{2}\inf \mathcal{T}$,
	\begin{equation}
		\label{eq:33}
		\frac{V(x,\lambda r)}{V(x,r)}\leq C_1 \lambda^{\gamma}.
	\end{equation}
	Suppose that there is $c_2>0$ such that
	\[
		c_2
		\frac{1}{V(x,f^{-1}(t))} \ind{\{ f(d(x,y))\leq t \}} \leq 
		p(t; x,y)
		\qquad\text{for all } x, y \in M, \, t \in \calT
	\]
	where $f$ is positive increasing function with $f(0) = 0$ and $f(1) = 1$, belonging to $\WLSCINF{\delta}{c_3}{0}$ 
	for some $c_3 \in (0, 1]$, and $\delta > 0$. If $\varphi$ belongs to $\WLSC{\beta}{C_4}{1/\inf\mathcal{T}}$ for a certain
	$C_4 \in [1, \infty)$ and $\beta > 0$, then there is $c > 0$ such that
	\[
		G_\varphi(x,y)\geq c 
		\min\left\{\frac{1}{V(x,d(x,y))\varphi(1/f(d(x,y)))},
		\, \frac{1}{V(x,f^{-1}(\inf\mathcal{T}))}\right\}
	\]
	for all $x, y \in M$, $x \neq y$.
\end{proposition}
\begin{proof}
	Let us first consider $f(d(x,y)) \geq \inf\mathcal{T}$, and $x\neq y$. Hence, $d(x, y) \geq \inf \calT$. Now,
	by \eqref{eq:33} we obtain
	\begin{align*}
		V(x,d(x,y)) G_\varphi(x,y)
		&=
		V(x,d(x,y)) \int_{[0,\infty)} p(u; x,y) U({\rm d} u) \\
		&\geq  
		c_2 \int_{[f(d(x,y)),\infty)} \frac{V(x,d(x,y))}{V(x,f^{-1}(u))} U({\rm d} u)\\
		&\geq 
		c_2 C_1^{-1}
		\int_{[f(d(x,y)),\infty)} \left(\frac{d(x,y)}{f^{-1}(u)}\right)^\gamma U({\rm d} u)\\
		&=
		\gamma c_2 C_1^{-1} d(x,y)^\gamma
		\int^\infty_{d(x,y)}\frac{U([f(d(x,y)),f(s)])}{s^{\gamma+1}}{\: \rm d} s.
	\end{align*}
	Next, by Lemma \ref{lem:22} for $s > d(x,y)$ we get
	\[
		U\big([f(d(x,y),f(s)]\big)
		\geq
		\frac{c_5}{\varphi(1/f(s))} - \frac{c_6}{\varphi(1/f(d(x, y)))}. 
	\]
	Since, by the lower scaling property of $\varphi$ and $f$, we have
	\begin{align*}
		\frac{\varphi(1/f(s))}{\varphi(1/f(d(x, y)))}
		&\leq C_4  \bigg(\frac{f(d(x, y))}{f(s)} \bigg)^\beta \\
		&\leq C_4 c_3^{-\beta} \bigg(\frac{d(x, y)}{s} \bigg)^{\beta \delta}
	\end{align*}
	we obtain
	\[
		U\big([f(d(x,y),f(s)]\big)
		\geq
		\frac{c_5}{\varphi(1/f(s))} \bigg(1 - C_4 c_3^{-\beta} \bigg(\frac{d(x, y)}{s} \bigg)^{\beta \delta} \bigg).
	\]
	Therefore, for $s \geq C_6 d(x, y)$, where
	\[
		C_6 = \big(2 C_4 c_3^{-\beta}\big)^{1/(\beta \delta)},
	\]
	we have
	\[
		U\big([f(d(x,y),f(s)]\big) \geq \frac{c_5}{2} \frac{1}{\varphi(1/f(s))}.
	\]
	Hence, by monotonicity of $f$ and $\varphi$, we obtain
	\begin{align*}
		V(x,d(x,y)) G_\varphi(x,y) 
		&\geq 
		c_7 d(x,y)^\gamma \int^\infty_{C_6 d(x,y)} \frac{1}{\varphi(1/f(s)) s^{\gamma+1}} {\: \rm d} s\\
		&\geq c_8 \frac{1}{\varphi(1/f(d(x,y)))}.
	\end{align*}
	In particular, this completes the proof in the case when $\calT = (0, \infty)$. If $\mathcal{T}=\mathbb{N}$, we need to
	consider $f(d(x,y))\leq 1$. In view of Lemma \ref{lem:22}, there is $C_9 \geq 1$, such that for all $\lambda \geq 1$,
	\[
		C_9^{-1} \frac{1}{\varphi(1/\lambda)} \leq U\big([0, \lambda]\big) \leq C_9
		\frac{1}{\varphi(1/\lambda)}.
	\]
	We select $\lambda \geq 1$ such that
	\[
		\varphi(1/\lambda) = \frac{2}{3} C_9^{-2}.
	\]
	Since $\varphi(1) = 1$, we get
	\begin{align*}
		U\big([1, \lambda]\big) 
		&=
		U\big([0, \lambda]\big) - U\big([0, 1)\big) \\
		&\geq 
		C_9^{-1} \frac{1}{\varphi(1/\lambda)} - C_9 \geq \frac{C_9}{2}.
	\end{align*}
	Hence,
	\begin{align*}
		G_\varphi(x,y)
		&\geq
		c_2\int_{[f(d(x,y)),\infty)} \frac{1}{V(x,f^{-1}(u))}U({\rm d}u)\\
		&\geq
		c_2\int_{[1, \lambda]} \frac{1}{V(x,f^{-1}(u))}U({\rm d}u)\\
		&\geq
		\frac{c_2 C_9}{2} \frac{1}{V(x,f^{-1}(\lambda))},
	\end{align*}
	which together with the scaling properties of $V$ and $f$ completes the proof.
\end{proof}

\begin{theorem}
	\label{thm:4}
	Assume that there exist $C_1 \geq 1$ and $\gamma_2 \geq \gamma_1 > 0$ such that for all $x \in M$,  $\lambda > 1$, and 
	$r > \frac{1}{2}\inf \mathcal{T}$,
	\[
		C_1^{-1} \lambda^{\gamma_1} \leq \frac{V(x,\lambda r)}{V(x,r)} \leq C_1 \lambda^{\gamma_2}
	\]
	Suppose that there exist $C_2 \geq 1$, such that for all $x, y \in M$ and $t \in \calT$,
	\[
		C_2^{-1} \frac{1}{V(x,f^{-1}(t))} \ind{\{ f(d(x,y)) \leq t\}} 
		\leq 
		p(t; x,y) 
		\leq
		C_2 \frac{1}{V(x, f^{-1}(t) \vee d(x, y) )} 
	\]
	where $f$ is positive increasing function with $f(0) = 0$, belonging to $\WLSCINF{\delta_1}{c_3}{0} \cap
	\WUSCINF{\delta_2}{C_4}{0}$ for certain $c_3 \in (0, 1]$, $C_4 \in [1, \infty)$ and $\delta_2 \geq \delta_1 > 0$.
	If $\gamma_1/\delta_2>1$ and $\varphi$ belongs to $\WLSC{\beta}{c_2}{1/\inf\mathcal{T}}$ for certain $C_5 \in [1, \infty)$
	and $\beta > 0$, then
	\[
		G_\varphi(x,y) 
		\approx
		\min\left\{\frac{1}{V(x,d(x,y))\varphi(1/f(d(x,y)))},\frac{1}{V(x,f^{-1}(\inf \calT))}\right\}.
	\]
\end{theorem}
\begin{proof}
	Thanks to Proposition \ref{prop:103} it remains to prove the upper estimates only. Fix $x_0, y_0 \in M$, $x_0 \neq y_0$.
	Let $t = f(d(x_0, y_0)) \vee (\inf \calT)$ and $W(s) = V(x_0, f^{-1}(s))$. By Proposition \ref{prop:102}, we get
	\begin{equation}
		\label{eq:35}
		G_\varphi(x_0, y_0) \leq C_6 \frac{1}{V(x_0, f^{-1}(t)) \varphi(1/t)}.
	\end{equation}
	Hence, for $f(d(x_0, y_0)) > \inf \calT$, we obtain
	\[
		G_\varphi(x_0, y_0) \leq C_6 \frac{1}{V(x_0, d(x_0, y_0)) \varphi(1/f(d(x_0, y_0)))}.
	\]
	Since the constant $C_6$ is independent of $x_0$ and $y_0$, the theorem follows in the case when $\calT = (0, \infty)$. 
	If $\calT = \NN$, by \eqref{eq:35}, for $f(d(x_0, y_0)) \leq 1$ we get
	\[
		G_\varphi(x_0, y_0) \leq C_6 \frac{1}{V(x_0, f^{-1}(1))}
	\]
	which completes the proof.
\end{proof}

\section{Stability of heat kernel estimates}
\label{sec:3}
In this section we consider stability of heat kernel estimates for Markov chains, that is when $\calT = \NN$. In this case
we need to assume that the state space is also discrete. To be more precise, let $(M, d, \mu)$ be a discrete metric measure
space that is: $(M, d)$ is discrete metric space, and $(M, \calB, \mu)$ is measure space where $\calB$ denote $\sigma$-algebra
of Borel sets and $\mu$ is a Radon measure on $(M, \calB)$. We assume that the support of $\mu$ equals $M$.

For $x \in M$ and $r > 0$, we set
\[
	V(x, r) = \mu\big(B(x, r)\big)
\]
where
\[
	B(x, r) = \big\{y \in M : d(x, y) \leq r \big\}.
\]
We assume that there are $C_V \geq 1$, $r_0 > 0$, and $\gamma_2 \geq \gamma_1 > 0$, such that for all $x \in M$
\begin{subequations}
\begin{equation}
	\label{eq:A:44}
	\frac{V(x, R)}{V(x, r)} \leq C_V \bigg(\frac{R}{r}\bigg)^{\gamma_2}
	\qquad \text{for all } R \geq r > 0,
\end{equation}
and 
\begin{equation}
	\label{eq:A:42}
	C_V^{-1} \bigg(\frac{R}{r} \bigg)^{\gamma_1} \leq \frac{V(x, R)}{V(x, r)}
	\qquad\text{for all } R \geq r \geq r_0,
\end{equation}
and
\begin{equation}
	\label{eq:A:23}
	V(x, r) = \mu(\{x\})
	\qquad \text{for all } r_0 > r \geq 0.
\end{equation}
\end{subequations}
Let $\psi: (0, \infty) \rightarrow (0, \infty)$ be increasing function such that $\psi(1) = 1$ and $\psi(0) = 1/2$, and
there are $C_\psi \geq 1$ and $\beta_2 \geq \beta_1 > 0$,
\begin{subequations}
	\begin{equation}
		\label{eq:A:46}
		\frac{\psi(R)}{\psi(r)} \leq C_\psi \bigg(\frac{R}{r}\bigg)^{\beta_2}
		\qquad\text{for all } R \geq r > 0,
	\end{equation}
	and
	\begin{equation}
		\label{eq:A:45}
		C_\psi^{-1}  \bigg(\frac{R}{r}\bigg)^{\beta_1} \leq 
		\frac{\psi(R)}{\psi(r)}
		\qquad\text{for all } R \geq r \geq 1.
	\end{equation}
\end{subequations}
A jump kernel is a function $J: M \times M \rightarrow [0, \infty)$ such that $J(x, y) = J(y, x)$ for all $x, y \in M$, and 
\[
	\int_M J(x, y)\, \mu({\rm d} y) = 1.
\]
We say that the jump kernel $J$ satisfies \eqref{jump} if 
\begin{equation}
	\label{jump}
	\tag{$J_\psi$}
	J(x, y) \approx \frac{1}{V(x, d(x, y)) \psi(d(x, y))}
\end{equation}
uniformly with respect to $x, y \in M$. It is easy to see that \eqref{jump} implies that there is $c > 0$ such that 
for all $x_0 \in M$ and $r > 0$,
\begin{equation}
	\label{eq:A:7}
	\int_{B(x_0, r)^c} J(x_0, y) \, \mu({\rm d} y) \leq \frac{c}{\psi(r)}.
\end{equation}
Given the jump function $J$ we introduce a Markov chain $\mathbf{X} = (X_n : n \in \NN_0)$ as
\[
	\PP(X_{n+1} = y \mid X_n = x) = J(x, y) \mu(\{y\}).
\]
The $n$-step transition function, called the discrete time heat kernel is defined as
\begin{align*}
	h(1; x, y) &= J(x, y), \\
	h(n; x, y) &= \int_M h(n-1; x, z) J(z, y) \, \mu({\rm d} z), \qquad n \geq 2.
\end{align*}
For each bounded function $f \in L^2(M, \mu)$, we define the Borel measure $\Gamma[f]$ on $M$ by setting
\[
	\Gamma[f](\{x\}) = \frac{1}{2}\int_M (f(x) - f(y))^2 J(x, y) \, \mu({\rm d} y)\mu(\{x\}).
\]
We say that cut-off Sobolev inequality holds true if for each $\epsilon > 0$ there is $C > 0$ such that 
for all $x_0 \in M$, $R \geq r > 0$ and $f \in L^2(M, \mu)$ there is a function $\phi: M \rightarrow [0, 1]$,
called cut-off function for $R+r > R > 0$,
\[
	\ind{B(x_0, R)} \leq \phi \leq \ind{B(x_0, R+r)}
\]
such that
\begin{equation}
	\label{cutoff}
	\tag{$\text{CS}_\psi$}
	\int_{B(x_0, R+2 r)} f^2 {\: \rm d}\Gamma[\phi]
	\leq
	\epsilon \int_{U'} \int_U (f(x) - f(y))^2 J(x, y) \, \mu({\rm d} x) \, \mu({\rm d} y) 
	+\frac{C}{\psi(r)} \int_{B(x_0, R+2r)} f^2 {\: \rm d} \mu
\end{equation}
where
\[
	U = \big\{y \in M : R \leq d(x_0, y) \leq R+r\big\},
	\qquad
	U' = \big\{y \in M : R- r \leq d(x_0, y) \leq R+2r\big\}. 
\]
In this section we show that the jump function satisfies \eqref{jump} and the cut-off Sobolev inequality \eqref{cutoff}
holds true if and only if the discrete time heat kernel $h$ satisfies
\begin{equation}
	\label{Dheat}
	\tag{$\text{DHK}_\psi$}
	h(n; x, y) 
	\approx
	\min\bigg\{\frac{1}{V(x, \psi^{-1}(n))}, \frac{n}{V(x, d(x, y)) \psi(d(x, y))} \bigg\}
\end{equation}
uniformly with respect to $x, y \in M$ and $n \in \NN$.

Let us first assume that the discrete time heat kernel satisfies \eqref{Dheat}. We can repeat the proof of
\cite[Lemma 2.7]{Chen2016} to show that there is $c > 0$ such that for all $x_0 \in M$, $r > 0$ and $k\in\NN$,
\[
	\PP^{x_0}\big(\tau_{B(x_0, r)} \leq k\big) \leq c \frac{k}{\psi(r)}.
\]
Using fact that $\tau_{B(x_0, r)}\geq 1$ we obtain that for any $\lambda > 0$,
\[
	\PP^{x_0}\big(\tau_{B(x_0, r)} \leq \lambda\big) \leq c \frac{\lambda}{\psi(r)}.
\]
Hence, following the proof of \cite[Lemma 3.4]{Chen2016} we show that for all $R \geq r > 0$ the function
\[
	h(x) = \EE^x \Big(\sum_{0 \leq n \leq \tau_{D_0}} \ind{D_1}(X_n) e^{-\lambda n}\Big)
\]
where $\lambda = \frac{1}{\psi(r)}$,
\[
	\begin{aligned}
		D_0 &= \big\{y \in M : R+ \tfrac{1}{10} r < d(x_0, y) < R + \tfrac{9}{10}r \big\} \\
		D_1 &= \big\{y \in M : R+ \tfrac{1}{5} r < d(x_0, y) < R + \tfrac{4}{5}r \big\} \\
		D_2 &= \big\{y \in M : R+ \tfrac{2}{5} r < d(x_0, y) < R + \tfrac{3}{5}r \big\} \\
	\end{aligned}
\]
satisfies $\supp h \subseteq \overline{D_0}$, $h(x) \leq 3\psi(r)$ for all $x \in M$, and there is $c > 0$ such that
\[
	h(x) \geq c \psi(r) \qquad\text{for all } x \in D_2.
\]
Now, the proof of \cite[Proposition 3.6]{Chen2016} shows that there are $C_1, C_2 > 0$ such that for all 
$x_0 \in M$ and $R \geq r > 0$ there is $\phi$ satisfying
\begin{equation}
	\label{eq:A:10}
	\ind{B(x_0, R)} \leq \phi \leq \ind{B(x_0, R+r)},
\end{equation}
and such that for all $f \in L^2(M, \mu)$,
\begin{equation}
	\label{eq:A:8}
	\int_{U'} f^2 {\:\rm d} \Gamma[\phi] 
	\leq C_1 \int_{U'} \int_U (f(x) - f(y))^2 J(x, y) \, \mu({\rm d} x) \, \mu({\rm d} y) 
	+
	\frac{C_2}{\psi(r)} \int_{U'} f^2 {\: \rm d} \mu
\end{equation}
where
\[
	U = \big\{y \in M : R \leq d(x_0, y) < R+r \big\},
	\qquad
	U' = \big\{y \in M : R - \tfrac{1}{2} r \leq d(x_0, y) < R + \tfrac{3}{2}r \big\}.
\]
Next, by combining the proofs of \cite[Proposition 2.3 (2) and (4)]{Chen2016} we see that \eqref{eq:A:8} implies that
there are $C_1, C_2 > 0$ such that for all $x_0 \in M$ and $R \geq r > 0$, there is $\phi$ satisfying 
\eqref{eq:A:10} and
\[
	\int_{B(x_0, R + 2r)} f^2 {\: \rm d} \Gamma[\phi] 
	\leq
	C_1 \int_{U''} \int_U (f(x) - f(y))^2 J(x, y) \mu({\rm d} x) \mu({\rm d} y)
	+
	\frac{C_2}{\psi(r)} \int_{B(x_0, R+2r)} f^2 {\: \rm d} \mu
\]
where
\[
	U'' = \big\{y \in M : R - r \leq d(x_0, y) < R + 2 r \big\}.
\]
Lastly, the reasoning as in the proof of \cite[Proposition 2.4(1)]{Chen2016} and \cite[Proposition 2.3(1)]{Chen2016}
leads to \eqref{cutoff}.

Now, let us consider the reverse implication. Let $(N_t : t \geq 0)$ be the standard Poisson process independent of the Markov
chain $\mathbf{X}=(X_n : n \in \NN_0)$. Let $p$ be the density of the transition probability of the process $Y_t = X_{N_t}$
that is
\begin{equation}
	\label{eq:A:18}
	\begin{aligned}
	p(t; x, y)
	&= \sum_{k = 0}^\infty h(k; x, y) \PP(N_t = k) \\
	&= \sum_{k = 0}^\infty h(k; x, y) \frac{e^{-t} t^k}{k!}.
	\end{aligned}
\end{equation}
Then $(P_t : t > 0)$  where
\[
	P_t f(x) = \int_M p(t; x, y) f(y) \mu({\rm d} y),
	\qquad\text{for } f \in L^2(M, \mu),
\]
is the semigroup associated to the Dirichlet form $(\calE, \calD)$, $\calD = L^2(M, \mu)$ which is defined as
\[
	\calE(f, g) = \frac{1}{2}\int_M \int_M
	(f(x) - f(y))(g(x) - g(y)) J(x, y) \, \mu({\rm d} x) \, \mu({\rm d} y).
\]
We claim that there is $c > 0$ such that for all $x, y \in M$, $t > 0$,
\begin{equation}
	\label{eq:A:48}
	p(t; x, y) \leq c \min\left\{\frac{1}{V(x, \psi^{-1}(1/2 \vee t))}, \frac{1/2 \vee t}{V(x, d(x, y)) 
	\psi(d(x, y))} \right\}.
\end{equation}
The first term on the right hand-side requires restriction to the range of $\psi$ which is $[1/2, \infty)$. Hence, $t$ cannot
be arbitrary close to zero. Indeed, otherwise for fixed $x \in M$ and $0 < t < 1/2$,
\[
	p(t; x, x) \leq 2 c \frac{t}{V(x, 0)},
\]
thus taking $t$ approaching zero we conclude that
\[
	\limsup_{t \to 0^+} p(t; x, x) = 0.
\]
This contradicts the estimate $p(t; x, x) \geq e^{-t} J(x,x)$ which is a consequence of \eqref{eq:A:18}. For the same reason
the estimate in \cite[Theorem 3.2]{Murugan2015} cannot be valid for all $t > 0$.

The proof of \eqref{eq:A:48} is closely related to \cite{Chen2016} and we express only the necessary changes. Since the jump
function satisfies \eqref{jump}, by the proof of \cite[Lemma 4.1]{Chen2016} we obtain that there is $c > 0$ such that for all
$r > 0$, and $f \in L^1(M, \mu)$,
\begin{equation}
	\label{eq:A:1}
	\|f\|_{L^2}^2 \leq c \bigg(\frac{\|f\|_{L^1}^2}{\inf_{x \in \supp f} V(x, r)} + \psi(r) \calE(f, f) \bigg).
\end{equation}
The proof of \cite[Proposition 3.1.4]{Boutayeb2015} shows that \eqref{eq:A:1} implies that there are $c > 0$ and
$\gamma > 0$ such that for all $x_0 \in M$, $r > 0$ and any function $f$ with support contained in $B(x_0, r)$,
\begin{equation}
	\label{eq:A:2}
	\|f\|_{L^2}^{2+2\gamma} \leq \frac{c}{V(x_0, r)^\gamma} \|f\|_{L^1}^{2\gamma}
	\Big(\|f\|_{L^2}^2 + \psi(r) \calE(f, f)\Big).
\end{equation}
Next, \eqref{eq:A:2} implies that there are $c > 0$ and $\gamma > 0$ such that for all $x_0 \in M$, $r > 0$ and any function
$f$ with support contained in $B(x_0, r)$,
\begin{equation}
	\label{eq:A:3}
	\|f\|_{L^2}^{2+2\gamma} \leq c \frac{\psi(r)}{V(x_0, r)^\gamma} \|f\|_{L^1}^{2\gamma} \calE(f, f).
\end{equation}
Indeed, for $r \geq r_0$ it follows by the proof of \cite[Proposition 3.4.1]{Boutayeb2015}. For $0 < r \leq r_0$ we can use
monotonicity of $V$ and $\psi$ and $\psi(0) = \frac{1}{2}$.

Then by the proof of \cite[Proposition 3.3.2]{Boutayeb2015} we can deduce the following variant of the Faber--Krahn 
inequality: there are $c > 0$ and $\gamma > 0$ such for all $x_0 \in M$, $r > 0$ and relatively compact open
set $U \subset B(x_0, r)$,
\begin{equation}
	\label{eq:A:4}
	\lambda_1(U) \geq \frac{c}{\psi(r)} \bigg(\frac{V(x_0, r)}{\mu(U)}\bigg)^\gamma
\end{equation}
where
\[
	\lambda_1(U) = \inf\bigg\{\frac{\calE(f, f)}{\|f\|_{L^2}^2} : f \in \calD \cap \calC_c(U), f \neq 0\bigg\}.
\]
Given $x_0 \in M$ and $r > 0$, set $B = B(x_0, r)$ and consider Dirichlet form $(\calE, \calD_B)$ where $\calD_B = L^2(B, \mu_B)$
and $\mu_B$ is the measure $\mu$ restricted to $B$. Let $(P^B_t : t > 0)$ denote the corresponding
semigroup. By \cite[Lemmas 5.5]{Grigoryan2014} from \eqref{eq:A:3} we can deduce that there is $c > 0$ such that for all 
$x_0 \in M$ and $r > 0$, the heat kernel $p_B$ for the Dirichlet form $(\calE, \calD_B)$ satisfies
\begin{equation}
	\label{eq:A:5}
	\max_{x, y \in B} p_B(t; x, y) \leq c \frac{1}{V(x_0, r)} \bigg(\frac{\psi(r)}{t}\bigg)^{\frac{1}{\gamma}},
\end{equation}
for all $t > 0$. Next, the proof of \cite[Lemma 4.14]{Chen2016} shows that \eqref{eq:A:5} implies that there is $c > 0$
such that for all $x_0 \in M$, and $r > 0$,
\begin{equation}
	\label{eq:A:6}
	\EE^{x_0}\big[\tau_{\mathbf{Y}}(x_0, r)\big] \leq c \psi(r)
\end{equation}
where $\tau_{\mathbf{Y}}(x_0, r)$ denote the exit time from the ball $B(x_0, r)$, that is
\[
	\tau_{\mathbf{Y}}(x_0, r) = \inf\big\{t > 0  : Y_t \notin B(x_0, r) \big\}.
\]
Our next task is to show the lower estimate on $\EE^{x_0}[\tau_{\mathbf{Y}}(x_0, r)]$. Here the reasoning depends on variants
of \cite[Proposition 2.3 (5)]{Chen2016} and \cite[Corollary 4.12]{Chen2016}. Given $\rho > 0$, let $(\calE^{(\rho)}, \calD)$
be the Dirichlet form corresponding to the jump function
\[
	J^{(\rho)}(x, y) = 
	\begin{cases}
		J(x, y) & \text{if } d(x, y) \leq \rho, \\
		0 & \text{otherwise.}
	\end{cases}
\]
By repeating the arguments used in the proof of \cite[Proposition 2.3(1)]{Chen2016} we can deduce
from \eqref{cutoff} that for each $\epsilon > 0$ there is $C > 0$ such that for all $x_0 \in M$ and
$R \geq r > 0$ there is a function $\phi$ satisfying
\begin{equation}
	\label{eq:A:12}
	\ind{B(x_0, R)} \leq \phi \leq \ind{B(x_0, R+r)},
\end{equation}
such that for all $f \in L^2(M, \mu)$ and $\rho > 0$,
\begin{equation}
	\label{eq:A:14}
	\begin{aligned}
	\int_{B(x_0, R+2r)} f^2 {\: \rm d} \Gamma^{(\rho)}[\phi]
	&\leq \epsilon \int_{U'} \int_U (f(x) - f(y))^2 J^{(\rho)}(x, y) \, \mu({\rm d} x) \, \mu({\rm d} y) \\
	&\phantom{\leq}+
	\frac{C}{\psi(r \wedge \rho)} \int_{B(x_0, R+2r)} f^2 {\: \rm d} \mu.
	\end{aligned}
\end{equation}
Next, by the proof of \cite[Proposition 2.3(5)]{Chen2016} we see that there is $c > 0$ such that for all $x_0 \in M$, 
all $R \geq r > 0$ and $\rho > 0$, and every function $\phi$ satisfying \eqref{eq:A:12} and \eqref{eq:A:14} we have
\begin{equation}
	\label{eq:A:15}
	\calE^{(\rho)}(\phi, \phi) \leq c \frac{V(x_0, R+r)}{\psi(r)}.
\end{equation}
Now, let us justify the second statement, that is \cite[Corollary 4.12]{Chen2016}. To do so we need to reprove 
\cite[Section 4.2]{Chen2016} under our hypothesis. A function $u$ is $\calE^{(\rho)}$-subharmonic in $D \subset M$, if
there is $c\in \RR$ such that $u-c\in \calD$ and
\[
	\calE^{(\rho)}(u, f) \leq 0
\]
for all nonnegative $f \in L^2(D, \mu_D)$. 
\begin{lemma}
	\label{lem:A:1}
	Let $x_0 \in M$. For $s > 0$ we set $B_s = B(x_0, s)$. Let $\theta > 0$, $R \geq 2 r > 0$, and $\rho > 0$. 
	Let $u$ be $\calE^{(\rho)}$-subharmonic function in $B_{R+r}$. We set $v = (u - \theta)_+$. Let $\phi$ be a cut-off 
	function satisfying \eqref{eq:A:14} for $R \geq R-r > 0$. Then there $c > 0$ such that
	\[
		\int_{B_{R+r}} {\rm d} \Gamma[v \phi]
		\leq 
		\frac{c}{\psi(r \wedge \rho)} 
		\bigg\{
		1
		+
		\frac{1}{\theta} 
		\bigg(\frac{R+\rho}{r}\bigg)^{\gamma_2} \frac{1}{\mu(B_{R+\rho})} \int_{B_{R+\rho}} |u| {\: \rm d}\mu
		\bigg\}
		\int_{B_{R+r}} u^2 {\: \rm d} \mu.
	\]
	The constant $c$ is independent of $x_0$, $R$, $r$, $\rho$ and $\theta$.
\end{lemma}
\begin{proof}
	Since $M$ is discrete the function $v \phi^2$ is bounded. Moreover, $u$ is $\calE^{(\rho)}$-subharmonic function in
	$B_{R+r}$ thus $\supp v \phi^2 \subset B_R$. We have
	\[
		0 \geq 2\calE^{(\rho)}(u, \phi^2 v) = I_1 + 2 I_2
	\]
	where
	\[
		I_1 = \int_{B_{R+r}} \int_{B_{R+r}} \big(u(x) - u(y)\big) 
		\big(v(x) \phi(x)^2 - v(y)\phi(y)^2\big) J^{(\rho)}(x, y) \, \mu({\rm d} y) \, \mu({\rm d} x)
	\]
	and
	\[
		I_2 = 
		\int_{B_{R+r}} \int_{B_{R+r}^c} \big(u(x) \phi(x) - u(y)\phi(y)\big) \phi^2(x) v(x) J^{(\rho)}(x, y) 
		\, \mu({\rm d} y) \, \mu({\rm d} x).
	\]
	We notice that if $u(x) \geq u(y)$
	\[
		(u(x) - u(y)) v(y) = (v(x) - v(y)) v(y) \qquad\text{for all } x, y \in M.
	\]
	To estimate $I_1$ we observe that if $u(x) \geq u(y)$ then
	\begin{align*}
		&\big(u(x) - u(y)\big)\big(\phi(x)^2 v(x) - \phi(y)^2 v(y)\big) \\
		&\qquad\qquad=
		\big(u(x) - u(y)\big) \phi(x)^2 \big(v(x) - v(y) \big) + 
		\big(u(x) - u(y)\big) \big(\phi(x)^2 - \phi(y)^2\big) v(y) \\
		&\qquad\qquad\geq
		\phi(x)^2 \big(v(x) - v(y)\big)^2 +\big(v(x) - v(y)\big)\big(\phi(x)^2 - \phi(y)^2\big) v(y). 
	\end{align*}
	Since $ab \geq -t\frac{1}{8} a^2 - 2 b^2$ and $(a+b)^2 \leq 2 a^2 + 2 b^2$, we have
	\begin{align*}
		&\phi(x)^2 \big(v(x) - v(y)\big)^2 +\big(v(x) - v(y)\big)\big(\phi(x)^2 - \phi(y)^2\big) v(y) \\
		&\qquad\qquad\geq
		\phi(x)^2 \big(v(x) - v(y)\big)^2 - \frac{1}{8} \big(\phi(x) + \phi(y)\big)^2 \big(v(x) - v(y)\big)^2 
		- 2 v(y)^2 \big(\phi(x) - \phi(y)\big)^2 \\
		&\qquad\qquad\geq
		\frac{3}{4} \phi(x)^2 \big(v(x) - v(y)\big)^2 - \frac{1}{4} \phi(y)^2 \big(v(x) - v(y)\big)^2 
		-2 v(y)^2\big(\phi(x) - \phi(y)\big)^2,
	\end{align*}
	thus
	\begin{align}
		&\big(u(x) - u(y)\big)\big(\phi(x)^2 v(x) - \phi(y)^2 v(y)\big) \nonumber\\
		&\qquad\qquad\geq
		\frac{3}{4} \phi(x)^2 \big(v(x) - v(y)\big)^2 - \frac{1}{4} \phi(y)^2 \big(v(x) - v(y)\big)^2 
		-2 v(y)^2\big(\phi(x) - \phi(y)\big)^2.\label{eq:A:71}
	\end{align}
	if $u(y) \geq u(x)\geq \theta$ we have
	\begin{align*}
		&\big(u(x) - u(y)\big)\big(\phi(x)^2 v(x) - \phi(y)^2 v(y)\big) \\
		&\qquad\qquad=
		\phi(x)^2 \big(v(x) - v(y)\big)^2 +\big(v(x) - v(y)\big)\big(\phi(x)^2 - \phi(y)^2\big) v(y), 
	\end{align*}
	and therefore \eqref{eq:A:71} holds as well in this case. If $u(y) \geq u(x)$ and $u(x)<\theta$,
	\begin{align*}
		&\big(u(x) - u(y)\big)\big(\phi(x)^2 v(x) - \phi(y)^2 v(y)\big) \\
		&\qquad\qquad=
		\big(u(y) - u(x)\big) \phi(y)^2 v(y) \\
		&\qquad\qquad\geq
		v^2(y) \phi(y)^2. 
	\end{align*}
	Since $a^2 \geq \tfrac{3}{4}b^2-\tfrac{1}{4}a^2-2(a-b)^2$ we obtain \eqref{eq:A:71}. Thus
	\begin{align*}
		I_1 &\geq 
		\frac{1}{2} \int_{B_{R+r}} \int_{B_{R+r}} \phi(x)^2 \big(v(x) - v(y)\big)^2 J^{(\rho)}(x, y) 
		\, \mu({\rm d} y) \, \mu({\rm d} x) \\
		&\phantom{\geq}
		- 2  \int_{B_{R+r}} \int_{B_{R+r}} v(x)^2 \big(\phi(x) - \phi(y)\big)^2 J^{(\rho)}(x, y) 
		\, \mu({\rm d} y) \, \mu({\rm d} x).
	\end{align*}
	Next, we estimate $I_2$. Since
	\[
		\big(u(x) - u(y)\big) \phi^2(x) v(x) \geq
		\big(v(x) - v(y)\big) \phi^2(x) v(x) \geq - v(y) v(x).
	\]
	Let $x \in B_R$. Suppose that $r > 0$. Then by \eqref{jump}
	\begin{align*}
		\int_{B_{R+r}^c} v(y) J^{(\rho)}(x, y) \, \mu({\rm d} y)
		&\leq
		C_1
		\int_{B_{R+r}^c \cap B(x, \rho)} |u(y)| \frac{1}{V(x, r) \psi(r)} \, \mu({\rm d} y) \\
		&\leq
		C_1 
		\frac{1}{\psi(r \wedge \rho)} \frac{V(x_0, R+\rho)}{V(x, r)} \frac{1}{\mu(B_{R+\rho})}
		\int_{B_{R+\rho}} |u(y)| \, \mu({\rm d} y).
	\end{align*}
	Therefore, by \eqref{eq:A:44},
	\begin{align*}
		\frac{V(x_0, R+\rho)}{V(x, r)}
		&\leq
		\frac{V(x, 2R+\rho)}{V(x, r)} \\
		&\leq 
		C_V
		\bigg(\frac{2R+\rho}{r} \bigg)^{\gamma_2},
	\end{align*}
	and so
	\[
		 \int_{B_{R+r}^c} v(y) J^{(\rho)}(x, y) \, \mu({\rm d} y)
		 \leq
		 C_2
		 \frac{1}{\psi(r \wedge \rho)}  \bigg(\frac{R+\rho}{r} \bigg)^{\gamma_2} 
		 \frac{1}{\mu(B_{R+\rho})} \int_{B_{R+\rho}} |u(y)| \, \mu({\rm d} y).
	\]
	%If $0 < r \leq r_0$, then by \eqref{jump},
	%\begin{align*}
	%	\int_{B_{R+r}^c} v(y) J^{(\rho)}(x, y) \mu({\rm d} y)
	%	&\leq
	%	C_3
	%	\int_{B_{R+\rho}} |u(y)| \frac{1}{V(x, r_0)} \mu({\rm d} y).
	%\end{align*}
	%Since $r_0 \leq 1$, by \eqref{eq:A:23} and \eqref{eq:A:44},
	%\begin{align*}
	%	\frac{V(x_0, R+\rho)}{V(x, r_0)} 
	%	&\leq
	%	\frac{V(x, 2R+\rho)}{V(x, r)} \\
	%	&\leq C_V
	%	\frac{1}{\psi(r \wedge \rho)}
	%	\bigg(\frac{2R+\rho}{r}\bigg)^{\gamma_2} 
	%	\psi(r \wedge \rho ) \\
	%	&\leq C_V
	%	\frac{1}{\psi(r \wedge \rho)}
	%	\bigg(\frac{2R+\rho}{r}\bigg)^{\gamma_2}.
	%\end{align*}
	Consequently,
	\begin{align*}
		-I_2 &\leq 
		\int_{B_R} v(x) 
		\int_{B_{R+r}^c} v(y) J^{(\rho)}(x, y) \, \mu({\rm d} y) \, \mu({\rm d} x) \\
		&\leq
		\frac{C_4}{2\theta} \frac{1}{\psi(r \wedge \rho)}
		\bigg(\frac{R+\rho}{r}\bigg)^{\gamma_2}
		\bigg(\int_{B_R} u^2 {\: \rm d} \mu\bigg) 
		\frac{1}{\mu(B_{R+\rho})} \int_{B_{R+\rho}} \abs{u} {\: \rm d} \mu.
	\end{align*}
	Therefore,
	\begin{align}
		\nonumber
		0 &\leq 
		4 \int_{B_{R+r}} \int_{B_{R+r}} v(x)^2 \big(\phi(x) - \phi(y)\big)^2 J^{(\rho)}(x, y) 
		\, \mu({\rm d} y) \, \mu({\rm d} x) \\
		\nonumber
		&\phantom{\leq}-
		\int_{B_{R}} \int_{B_{R+r}} \phi(x)^2 \big(v(x) - v(y)\big)^2 J^{(\rho)}(x, y) 
		\, \mu({\rm d} y) \, \mu({\rm d} x) \\
		\nonumber
		&\phantom{\leq}
		+\frac{C_4}{\theta} \frac{1}{\psi(r \wedge \rho)}
		\bigg(\frac{R+\rho}{r}\bigg)^{\gamma_2}
		\bigg(\int_{B_R} u^2 {\: \rm d} \mu\bigg) 
		\frac{1}{\mu(B_{R+\rho})} \int_{B_{R+\rho}} \abs{u} {\: \rm d} \mu \\
		&
		\begin{aligned}
		\label{eq:A:73}
		&\leq
		4 \int_{B_{R+r}} v^2 {\: \rm d} \Gamma^{(\rho)}[\phi] 
		-
		\int_{B_{R}} \int_{B_{R+r}} \phi(x)^2 \big(v(x) - v(y)\big)^2 J^{(\rho)}(x, y) 
		\, \mu({\rm d} y) \, \mu({\rm d} x) \\
		&\phantom{\leq}
		+\frac{C_4}{\theta} \frac{1}{\psi(r \wedge \rho)}
		\bigg(\frac{R+\rho}{r}\bigg)^{\gamma_2}
		\bigg(\int_{B_R} u^2 {\: \rm d} \mu\bigg)
		\frac{1}{\mu(B_{R+\rho})} \int_{B_{R+\rho}} \abs{u} {\: \rm d} \mu.
		\end{aligned}
	\end{align}
	Next, we have
	\begin{align*}
		\int_{B_{R+r}} {\rm d}\Gamma[v \phi]
		&\leq
		\int_{B_{R+r}} \int_M \big(v(x) \phi(x) - v(y)\phi(y)\big)^2 J^{(\rho)}(x, y) \, \mu({\rm d} y) \, \mu({\rm d} x) \\
		&\phantom{\leq}+
		2 \int_{B_R} v^2(x) \phi^2(x) \bigg(\int_{B(x, \rho)^c} J(x, y) \, \mu({\rm d} y)\bigg) \, \mu({\rm d} x) \\
		&\phantom{\leq}+
		2 \int_M v^2(y) \phi^2(y) \bigg(\int_{B(y, \rho)^c} J(x, y) \, \mu({\rm d} x) \bigg) \, \mu({\rm d} y).
	\end{align*}
	Hence, by \eqref{eq:A:7},
	\begin{align}
		\nonumber
		\int_{B_{R+r}} {\rm d}\Gamma[v \phi]
		&\leq
		\int_{B_{R+r}} \int_{B_{R+r}}
		\big(v(x) \phi(x) - v(y)\phi(y)\big)^2 J^{(\rho)}(x, y) \, \mu({\rm d} y) \, \mu({\rm d} x) \\
		\nonumber
		&\phantom{\leq}+
		\int_{B_R} v^2(x) \phi^2(x) \bigg(\int_{B_{R+r}^c} J^{(\rho)}(x, y) \, \mu({\rm d} y)\bigg) \, \mu({\rm d} x)
		+
		\frac{C_5}{\psi(\rho)} \int_{B_R} v^2 {\: \rm d} \mu \\
		\label{eq:A:72}
		&
		\begin{aligned}
		&\leq
		2 \int_{B_{R+r}} v^2 {\rm d}\Gamma^{(\rho)}[\phi]
		+
		2 \int_{B_R} \int_{B_{R+r}} \phi^2(x) \big(v(x) - v(y)\big)^2 J^{(\rho)}(x, y)
		\, \mu({\rm d} y) \, \mu({\rm d} x) \\
		&\phantom{\leq}+
		\frac{2C_5}{\psi(r \wedge \rho)} \bigg(\int_{B_R} u^2 {\: \rm d} \mu\bigg). 
		\end{aligned}
	\end{align}
	By \eqref{eq:A:72} and \eqref{eq:A:73}, for any $\eta > 0$,
	\begin{align*}
		\eta \int_{B_{R+r}} {\rm d}\Gamma[v \phi]
		&\leq
		(2\eta+ 4) \int_{B_{R+r}} v^2 {\: \rm d}\Gamma^{(\rho)}[\phi] \\
		&\phantom{\leq}
		+(2\eta-1) \int_{B_{R}} \int_{B_{R+r}} \phi(x)^2 \big(v(x) - v(y)\big)^2 J^{(\rho)}(x, y) 
		\, \mu({\rm d} y) \, \mu({\rm d} x)\\
		&\phantom{\leq}+
		\frac{C_4}{\theta} \frac{1}{\psi(r \wedge \rho)}
		\bigg(\frac{R+\rho}{r}\bigg)^{\gamma_2}
		\bigg(\int_{B_R} u^2 {\: \rm d} \mu\bigg)
		\frac{1}{\mu(B_{R+\rho})} \int_{B_{R+\rho}} \abs{u} {\: \rm d} \mu \\
		&\phantom{\leq}+\frac{2C_5\eta}{\psi(r \wedge \rho)}
		\bigg(\int_{B_R} u^2 {\: \rm d} \mu\bigg).
	\end{align*}
	Now by \eqref{cutoff} we obtain
	\begin{align*}
		\int_{B_{R+r}} v^2 {\: \rm d} \Gamma[\phi]
		&\leq \frac{1}{8}
		\int_{B_{R}} \int_{B_{R+r}} \phi(x)^2 \big(v(x) - v(y)\big)^2 J^{(\rho)}(x, y) \, \mu({\rm d} y) \, \mu({\rm d} x) \\
		&\phantom{\leq}
		+\frac{C_6}{\psi(r \wedge \rho)} \bigg(\int_{B_R} u^2 {\: \rm d} \mu\bigg).
	\end{align*}
	Thus taking $\eta = \frac{2}{9}$ we get
	\begin{align*}
		\frac{2}{9} \int_{B_{R+r}} {\rm d}\Gamma[v \phi]
		&\leq
		\frac{C_4}{\theta} \frac{1}{\psi(r \wedge \rho)} \bigg(\frac{R+\rho}{r}\bigg)^{\gamma_2}
		\bigg(\int_{B_R} u^2 {\: \rm d} \mu\bigg)
		\frac{1}{\mu(B_{R+\rho})} \int_{B_{R+\rho}} \abs{u} {\: \rm d} \mu \\
		&\phantom{\leq}
		+\frac{C_7}{\psi(r \wedge \rho)} \bigg(\int_{B_R} u^2 {\: \rm d} \mu\bigg)
	\end{align*}
	proving the lemma.
\end{proof}
Next, we prove the following lemma.
\begin{lemma}
	\label{lem:A:2}
	Let $x_0 \in M$. For $s > 0$ we set $B_s = B(x_0, s)$. Let $\theta > 0$, $r_1 \geq r_2/2 > 0$, and $\rho > 0$. 
	Let $u$ be $\calE^{(\rho)}$-subharmonic function on $B_{r_1+r_2}$. We set $v = (u - \theta)_+$. Then there $c > 0$ such that
	\begin{align*}
		\int_{B_{r_1}} v^2 {\: \rm d} \mu
		&\leq
		\frac{c}{\theta^{2\gamma} \mu(B_{r_1+r_2})^{\gamma}}
		\bigg(\frac{r_1+r_2}{r_2 \wedge \rho}\bigg)^{\beta_2}
		\bigg\{1 + \frac{1}{\theta} \bigg(\frac{r_1+r_2/2+\rho}{r_2}\bigg)^{\gamma_2} 
		\bigg(\frac{1}{\mu(B_{r_1+r_2/2+\rho})} \int_{B_{r_1+r_2/2+\rho}} |u| {\: \rm d}\mu\bigg)\bigg\} \\
		&\phantom{\leq}\times
		\bigg(\int_{B_{r_1+r_2}} u^2 {\: \rm d} \mu\bigg)^{1+ \gamma}.
	\end{align*}
	The constant $c$ is independent of $x_0$, $r_1$, $r_2$, $\rho$ and $\theta$.
\end{lemma}
\begin{proof}
	Let
	\[
		D = \{x \in B_{r_1+r_2/2} : u(x) > \theta\}.
	\]
	Let $\phi$ be a cut-off function satisfying \eqref{eq:A:14} for $r_1+r_2/2 > r_1 > 0$. Then, by Lemma \ref{lem:A:1},
	we get
	\[
		\int_{B_{r_1+r_2}} {\rm d} \Gamma[v \phi]
		\leq
		\frac{C_1}{\psi((r_2/2) \wedge \rho)} 
		\bigg\{
		1 + \frac{1}{\theta} \bigg(\frac{r_1+r_2/2+\rho}{r_2/2}\bigg)^{\gamma_2} 
		\frac{1}{\mu(B_{r_1+r_2/2+\rho})} \int_{B_{r_1+r_2/2+\rho}} |u| {\: \rm d} \mu
		\bigg\}
		I_0
	\]
	where
	\[
		I_0 = \int_{B_{r_1+r_2}} u^2 {\: \rm d} \mu.
	\]
	Moreover, by \eqref{eq:A:4},
	\[
		\lambda_1(D) \geq \frac{C_2}{\psi(r_1+r_2)} \bigg(\frac{V(x_0, r_1+r_2)}{\mu(D)} \bigg)^{\gamma},
	\]
	and since $v \phi$ is supported on $D$,
	\[
		\lambda_1(D) \int_D v^2 \phi^2 {\: \rm d} \mu
		\leq \calE(v \phi, v\phi) \leq 2 \int_D {\rm d}\Gamma[v \phi].
	\]
	Hence,
	\[
		\int_D {\rm d}\Gamma[v \phi]
		\geq
		\frac{C_2}{2} \frac{1}{\psi(r_1+r_2)} \mu(B_{r_1+r_2})^\gamma \mu(D)^{-\gamma}
		\int_D v^2 \phi^2 {\: \rm d} \mu.
	\]
	In view of the Markov inequality
	\[
		\mu(D) 
		\leq \frac{1}{\theta^2} \int_{B_{r_1+r_2/2}} u^2 {\: \rm d} \mu \leq \frac{I_0}{\theta^2}.
	\]
	Thus
	\begin{align*}
		\int_D {\rm d}\Gamma[v \phi]
		&\geq
		\frac{C_2}{2} \frac{1}{\psi(r_1+r_2)} \mu(B_{r_1+r_2})^\gamma \theta^{2\gamma} I_0^{-\gamma}
		\int_D v^2 \phi^2 {\: \rm d} \mu \\
		&\geq
		\frac{C_2}{2} \frac{1}{\psi(r_1+r_2)} \mu(B_{r_1+r_2})^\gamma \theta^{2\gamma} I_0^{-\gamma} 
		\int_{B_{r_1}} v^2 {\: \rm d} \mu.
	\end{align*}
	Since by \eqref{eq:A:46},
	\[
		\frac{\psi(r_1+r_2)}{\psi((r_2/2) \wedge \rho)} 
		\leq
		C'
		\bigg(\frac{r_1+r_2}{r_2 \wedge \rho} \bigg)^{\beta_2}
	\]
	the lemma follows.
\end{proof}

\begin{proposition}
	\label{prop:A:1}
	There is $c > 0$ such that for all $x_0 \in M$, $R > 0$, $\rho > 0$ and any $\calE^{(\rho)}$-subharmonic function on 
	$B(x_0, R)$,
	\begin{align*}
		\max_{B(x_0, R/2)}{u^2} 
		\leq
		c\bigg\{
		\bigg(&\frac{1}{\mu(B(x_0, R+\rho))} \int_{B(x_0, R+\rho)} |u| {\: \rm d} \mu\bigg)^2 \\
		&+
		\bigg(1 + \frac{R}{\rho}\bigg)^{\frac{\beta_2}{\gamma}} 
		\bigg(1+\frac{\rho}{R} \bigg)^{\frac{\gamma_2}{\gamma}}
		\frac{1}{\mu(B(x_0, R))} \int_{B(x_0, R)} u^2 {\: \rm d}\mu\bigg\}.
	\end{align*}
	The constant $\gamma$ is determined in \eqref{eq:A:4}.
\end{proposition}
\begin{proof}
	For $i \in \NN$ and $\theta > 0$ we set
	\[
		r_i = \frac{1}{2}(1 + 2^{-i}) R, \qquad
		\theta_i = (1 - 2^{-i}) \theta, \qquad
		B_{r_i} = B(x_0, r_i),
	\]
	and
	\[
		I_i = \int_{B_{r_i}} (u - \theta_i)^2_+ {\: \rm d} \mu.
	\]
	Since $u$ is subharmonic also $u-\theta_i$ is subharmonic. Hence, by \cite[Lemma 3.2]{GrigoryanHuHu2018} for each
	$i \in \NN$, the function $(u - \theta_i)_+$ is subharmonic. By Lemma \ref{lem:A:2},
	\begin{align*}
		&I_{i+1} 
		= \int_{B_{r_{i+1}}} \big(u - \theta_i - (\theta_{i+1}-\theta_i)\big)_+^2 {\: \rm d}\mu \\
		&\leq
		\frac{c}{(\theta_{i+1} - \theta_i)^{2\gamma} \mu(B_{r_i})^\gamma }
		\bigg(\frac{r_i}{\min\{r_i-r_{i+1}, \rho\}} \bigg)^{\beta_2}
		\bigg\{1 + \frac{1}{(\theta_{i+1} - \theta_i)} \bigg(\frac{r_i+\rho}{r_i - r_{i+1}}\bigg)^{\gamma_2}\
		\frac{1}{\mu(B_{r_i+\rho})} \int_{B_{r_i+\rho}} |u| {\: \rm d} \mu \bigg\}
		I_i^{1+\gamma} \\
		&\leq
		2^{2\gamma i} 
		\frac{c'}{\theta^{2\gamma} \mu(B_R)^\gamma}
		\bigg(\frac{R}{\min\{2^{-i} R, \rho\}}\bigg)^{\beta_2}
		\bigg\{1 + 2^{\gamma_2 i} \bigg(1 +2^i \frac{\rho}{R}\bigg)^{\gamma_2}
		\frac{A_{R+\rho}}{\theta} \bigg\} I_i^{1+\gamma} \\
		&\leq
		\frac{c''}{\theta^{2\gamma} \mu(B_R)^\gamma} 2^{(2\gamma+\beta_2+\gamma_2+1)i}
		\bigg(1+\frac{R}{\rho}\bigg)^{\beta_2} \bigg(1 + \frac{\rho}{R}\bigg)^{\gamma_2}
		\bigg(1 + \frac{A_{R+\rho}}{\theta}\bigg) I_i^{1+\gamma}
	\end{align*}
	where we have set
	\[
		A_{R+\rho} = \frac{1}{\mu(B_{R+\rho})} \int_{B_{R+\rho}} |u| {\: \rm d} \mu.
	\]
	Let
	\[
		b = 2^{2\gamma + \beta_2+\gamma_2+1},
	\]
	and
	\[
		\theta = 
		A_{R+\rho} + 
		(2c'')^{\frac{1}{2\gamma}} b^{\frac{1}{2\gamma^2}} \sqrt{\frac{I_0}{\mu(B_R)}}
		\bigg(1 + \frac{R}{\rho} \bigg)^{\frac{\beta_2}{2\gamma}}
		\bigg(1 + \frac{\rho}{R} \bigg)^{\frac{\gamma_2}{2\gamma}}.
	\]
	Then
	\[
		I_{i+1} \leq c_0 b^i I_i^{1+\gamma}, \qquad
		c_0 = b^{-\frac{1}{\gamma}} I_0^{-\gamma}.
	\]
	Therefore,  by \cite[Lemma 4.9]{Chen2016},
	\[
		I_i \leq b^{-\frac{i}{\gamma}} I_0,
	\]
	and since
	\[
		\int_{B_{R/2}} (u - \theta)_+^2 {\: \rm d}\mu \leq \inf_{i \in \NN} I_i,
	\]
	we obtain 
	\[
		\max_{B(x_0, R/2)} u^2 \leq \theta^2
		\leq
		c_0\bigg\{A_{R+\rho}^2 + 
		\bigg(1 + \frac{R}{\rho}\bigg)^{\frac{\beta_2}{\gamma}} \bigg(1+\frac{\rho}{R} \bigg)^{\frac{\gamma_2}{\gamma}}
		\frac{1}{\mu(B_R)} \int_{B_R} u^2 {\: \rm d}\mu\bigg\}
	\]
	which completes the proof.
\end{proof}
Having established Proposition \ref{prop:A:1} we can repeat the proof of \cite[Corollary 4.12]{Chen2016}. Consequently,  
there is $c > 0$ such that for all $x_0 \in M$, $R, \rho > 0$ and any non-negative $\calE^{(\rho)}$-subharmonic
function $u$ on $B(x_0, R)$,
\begin{equation}
	\label{eq:A:17}
	\max_{B(x_0, R/2)}{u}
	\leq
	\frac{c}{V(x_0, R)}
	\bigg(1 + \frac{R}{\rho}\bigg)^{\frac{\beta_2}{\gamma}} 
	\bigg(1+\frac{\rho}{R} \bigg)^{\frac{\gamma_2}{\gamma}}
	\int_{B(x_0, R+\rho)} u {\: \rm d}\mu.
\end{equation}
Let us denote by $\mathbf{Y}^{(\rho)}$ the process associated to $(\calE^{(\rho)}, \calD)$. Let $\xi_\lambda$ be the exponential
distributed random variable with mean $1/\lambda$ independent of $\mathbf{Y}^{(\rho)}$. Using \eqref{eq:A:15} and 
\eqref{eq:A:17}, we can repeat the arguments used in the proof of \cite[Lemma 4.15]{Chen2016} to show that there is $c > 0$
such that for all $x_0 \in M$ and $r > 0$,
\begin{equation}
	\label{eq:A:9}
	\EE^{x_0} \Big[\tau_{\mathbf{Y}^{(\rho)}}(x_0, r) \wedge \xi_\lambda\Big] \geq c \psi(r)
\end{equation}
where $\rho=r/2$, $\lambda = 1/\psi(r)$ and $\tau_{\mathbf{Y}^{(\rho)}}(x_0, r)$ is the exit time of $\mathbf{Y}^{(\rho)}$ from 
the ball $B(x_0, r)$. Finally, by following the proof of \cite[Lemma 4.17]{Chen2016} we can show that there is $c > 0$ such that
for all $x_0 \in M$, and $r > 0$,
\begin{equation}
	\label{eq:A:19}
	\EE^{x_0} \big[ \tau_{\mathbf{Y}}(x_0, r) \big] \geq c \psi(r).
\end{equation}
The reasoning in \cite[p. 553]{Grigoryan2014} proves that \eqref{eq:A:5} entails \eqref{eq:A:4}, thus one can show that 
\cite[Lemma 4.18]{Chen2016} holds true for all $r > 0$. Thanks to \eqref{eq:A:9} and \eqref{eq:A:19}, 
we can repeat the proofs of \cite[Lemmas 4.19, 4.20, 4.21 and Corollary 4.22]{Chen2016} to conclude that there are
$C, c_1, c_2 > 0$ such that for all $x_0 \in M$, $r, \rho > 0$, and $t > 0$,
\[
	\PP^{x_0}\big(\tau_{\mathbf{Y}^{(\rho)}}(x_0, r) \leq t\big)
	\leq C \exp\bigg(-c_1 \frac{r}{\rho} + c_2 \frac{t}{\psi(\rho)}\bigg).
\]
Now, we can follow the proof of \cite[Proposition 4.23]{Chen2016} to show that the semigroup
$(Q^{(\rho)}_t : t > 0)$ corresponding to the Dirichlet form $(\calE^{(\rho)}, \calD)$ has the heat kernel $q^{(\rho)}$.
Moreover, there is $c > 0$ such that for all $x_0 \in M$, and $\psi(\rho) \geq t > 0$,
\begin{equation}
	\label{eq:A:26}
	\max_{x, y \in B(x_0, \rho)} q^{(\rho)}(t; x, y) 
	\leq \frac{c}{V(x_0, \rho)} \bigg(\frac{\psi(\rho)}{t}\bigg)^{1/\gamma}.
\end{equation}
Next, by the Meyer's decomposition and \cite[Lemma 7.2(1)]{Chen2016}
\begin{equation}
	\label{eq:A:27}
	p(t; x, y) \leq q^{(\rho)}(t; x, y) + \int_0^t \int_M \EE^x\big[J_\rho(Y^{(\rho)}_s, z) p(t-s; z, y) \big]\, 
	\mu({\rm d} z) {\:\rm d} s
\end{equation}
for all $x, y \in M$, where
\[
	J_\rho(x, y) = J(x, y) \ind{\{d(x, y) > \rho\}}.
\]
Following the proof of \cite[Proposition 4.24]{Chen2016} we show that there is $c > 0$ such that for all 
$\rho, t > 0$ and $x \in M$,
\begin{equation}
	\label{eq:A:31}
	\int_0^t \int_M
	\EE^x\big[J_\rho(Y^{(\rho)}_s, z) p(t-s; z, y)\big] \, \mu({\rm d} z) {\: \rm d} s
	\leq
	\frac{c t}{V(x, \rho) \psi(\rho)} \exp\bigg(c\frac{t}{\psi(\rho)}\bigg).
\end{equation}
\begin{theorem}
	There is $c > 0$ such that for all $x \in M$ and $t > 0$,
	\[
		p(t; x, x) \leq \frac{c}{V(x, \psi^{-1}(1/2 \vee t))}.
	\]
\end{theorem}
\begin{proof}
	We start by proving the following claim.
	\begin{claim}
		\label{clm:A:2}
		There is $c > 0$ such that for all $x \in M$ and $n \in \NN$,
		\[
			h(n; x, x) \leq \frac{c}{V(x, \psi^{-1}(1/2))}.
		\]
	\end{claim}
	The proof is by induction on $n$. Since $V(x, \psi^{-1}(1/2)) = V(x, 0)$, in view of \eqref{jump}, the 
	statement holds true for $n = 1$. Assume that the estimate is true for $n \geq 1$. Then
	\begin{align*}
		h(n+1; x, x) 
		&= \int_M h(n; x, y) J(y, x) \, \mu({\rm d} y) \\
		&\leq 
		\frac{c}{V(x, \psi^{-1}(1/2))} \int_M h(n; x, y) \, \mu({\rm d} y) 
		= \frac{c}{V(x, \psi^{-1}(1/2))}
	\end{align*}
	and the claim follows.
	
	Now, observe that by \eqref{eq:A:18} and Claim \ref{clm:A:2}, we obtain
	\begin{align*}
		p(t; x, x) \leq \frac{c}{V(x, \psi^{-1}(1/2))}.
	\end{align*}
	For $t \geq 1/2$, we set
	\[
		\rho =  \psi^{-1}(t).
	\]
	Then by \eqref{eq:A:26}, we obtain
	\[
		q^{(\rho)}(t; x, x) \leq \frac{c}{V(x, \psi^{-1}(t))}.
	\]
	Now, the theorem is a consequence of \eqref{eq:A:27} and \eqref{eq:A:31}.
\end{proof}
Next, by repeating \cite[Lemma 5.1]{Chen2016} in the present context we can show that there is $c > 0$ such that for all 
$t > 0$ and $x, y \in M$, 
\begin{equation}
	\label{eq:A:29}
	p(t; x, y) \leq c \min\bigg\{\frac{1}{V(x, \psi^{-1}(1/2 \vee t))}, \frac{1}{V(y, \psi^{-1}(1/2 \vee t))} \bigg\}.
\end{equation}
Following the proof \cite[Lemma 5.2]{Chen2016} we show that there are $c_1, c_2, c_3 > 0$ such that for all
$x, y \in M$, $\rho, t > 0$,
\begin{equation}
	\label{eq:A:30}
	q^{(\rho)}(t; x, y) \leq c_1 \bigg(\frac{1}{V(x, \psi^{-1}(1/2 \vee t))} + \frac{1}{V(y, \psi^{-1}(1/2 \vee t))} \bigg)
	\exp\bigg(c_2\frac{t}{\psi(\rho)} - c_3 \frac{d(x, y)}{\rho} \bigg).
\end{equation}
Finally, we show the following theorem.
\begin{theorem}
	\label{thm:A:1}
	Suppose that a metric measure space $(M, d, \mu)$ satisfies \eqref{eq:A:44}--\eqref{eq:A:23}. Let $\psi: (0, \infty) 
	\rightarrow (0, \infty)$ be increasing function satisfying \eqref{eq:A:46} and \eqref{eq:A:45}. Assume that the jump
	function $J$ satisfies \eqref{jump} and the cut-off Sobolev inequality \eqref{cutoff} holds true. Then
	there is $c > 0$ such that for all $x, y \in M$ and $t > 0$,
	\begin{equation}
		\label{eq:A:32}
		p(t; x, y) \leq c
		\min\bigg\{\frac{1}{V(x, \psi^{-1}(1/2 \vee t))}, \frac{1/2 \vee t}{V(x, d(x, y)) \psi(d(x, y))}\bigg\}.
	\end{equation}
\end{theorem}
\begin{proof}
	The proof follows the arguments used in \cite[Proposition 5.3]{Chen2016}. Let 
	\[
		N = 1 + \frac{\gamma_2}{\beta_1}+\frac{1}{2\gamma_2}.
	\]
	We have the following statement.
	\begin{claim}
		\label{clm:A:3}
		There are $c > 0$ and $\delta > 0$ such that for all $x \in M$, $r > 0$ and $t > 0$,
		\[
			\int_{B(x, r)^c} p(t; x, y) \, \mu({\rm d} y) \leq 
			c \bigg(\frac{\psi^{-1}(1/2 \vee t)}{r}\bigg)^{\delta}.
		\]
	\end{claim}
	We can assume that $r > \psi^{-1}(1/2 \vee t)$. Moreover, it is enough to consider $r \geq r_0$. Let 
	\begin{equation}
		\label{eq:A:43}
		\max\bigg\{\frac{\gamma_2}{\gamma_2 + \beta_1}, \frac{1}{2}\bigg\}  < \alpha < 1.
	\end{equation}
	For $n \in \NN$, we set
	\[
		\rho_n = 2^{n\alpha} r^{1-1/N} (\psi^{-1}(1/2 \vee t))^{1/N}.
	\]
	Since $N > 2$, and $2 \alpha > 1$, we have
	\[
		2^n r \geq \rho_n \geq \psi^{-1}(1/2 \vee t) 
		\bigg(\frac{r}{\psi^{-1}(1/2 \vee t)}\bigg)^{1 - 1/N} \geq \psi^{-1}(1/2 \vee t),
	\]
	and
	\[
		\frac{\rho_n}{\psi^{-1}(1/2 \vee t)} \geq \frac{2^n r}{\rho_n}.
	\]
	Now, if $2^n r \leq d(x, y) < 2^{n+1} r$, then by \eqref{eq:A:30} and \eqref{eq:A:44} 
	\[
		q^{(\rho_n)}(t; x, y) \leq 
		C_V \frac{1}{V(x, \psi^{-1}(1/2 \vee t))} \bigg(\frac{2^n r}{\psi^{-1}(1/2 \vee t)}\bigg)^{\gamma_2}
		\exp\bigg(-c_2 \frac{2^n r}{\rho_n} \bigg),
	\]
	hence by \eqref{eq:A:27} and \eqref{eq:A:31},
	\begin{align}
		\nonumber
		\int_{B(x, r)^c}
		p(t; x, y) \, \mu({\rm d} y)
		&=
		\sum_{n = 0}^\infty \int_{B(x, 2^{n+1} r) \setminus B(x, 2^n r)} p(t; x, y) \, \mu({\rm d} y) \\
		\label{eq:A:49}
		&
		\begin{aligned}
		&\leq
		C_2
		\sum_{n = 0}^\infty 
		\frac{V(x, 2^n r)}{V(x, \psi^{-1}(1/2 \vee t))} \bigg(\frac{2^n r}{\psi^{-1}(1/2 \vee t)}\bigg)^{\gamma_2}
		\exp\bigg(-c_2 \frac{2^n r}{\rho_n} \bigg)  \\
		&\phantom{\leq}+
		C_3
		\sum_{n = 0}^\infty t \frac{V(x, 2^n r)}{V(x, \rho_n) \psi(\rho_n)}.
		\end{aligned}
	\end{align}
	By \eqref{eq:A:44} and \eqref{eq:A:45}
	\begin{align*}
		\frac{V(x, 2^n r)}{V(x, \rho_n)}\cdot\frac{t}{\psi(\rho_n)}
		&\leq
		C_4
		\bigg(\frac{2^n r}{\rho_n}\bigg)^{\gamma_2} 
		\bigg(\frac{\psi^{-1}(1/2 \vee t)}{\rho_n}\bigg)^{\beta_1}  \\
		&=
		C_4
		\bigg(\frac{\psi^{-1}(1/2 \vee t)}{r}\bigg)^{\beta_1 -(\beta_1+\gamma_2)/N}
		2^{n(\gamma_2-\alpha\gamma_2 - \alpha \beta_1)} .
	\end{align*}
	In view of \eqref{eq:A:43},
	\[
		\sum_{n = 0 }^\infty \frac{V(x, 2^n r)}{V(x, \rho_n)}\cdot\frac{t}{\psi(\rho_n)}
		\leq
		C_5 \bigg(\frac{\psi^{-1}(1/2 \vee t)}{r}\bigg)^{\beta_1 -(\beta_1+\gamma_2)/N}.
	\]
	To bound the first sum in \eqref{eq:A:49}, by \eqref{eq:A:44} and \eqref{eq:A:45} we get
	\begin{align*}
		&
		\frac{V(x, 2^n r)}{V(x, \psi^{-1}(1/2 \vee t))} 
		\bigg(\frac{2^n r}{\psi^{-1}(1/2 \vee t)}\bigg)^{\gamma_2} \\
		&\qquad\qquad\leq
		C_V \bigg(\frac{2^n r}{\psi^{-1}(1/2 \vee t)}\bigg)^{2 \gamma_2} \\
		&\qquad\qquad=
		C_V \Bigg(2^{n(1-\alpha)} \bigg(\frac{\psi^{-1}(1/2 \vee t)}{r}\bigg)^{-\frac{1}{N}} \Bigg)^{2\gamma_2/(1-\alpha)}
		\bigg(\frac{\psi^{-1}(1/2 \vee t)}{r}\bigg)^{-2 \frac{\gamma_2}{1-\alpha}(1-\alpha-1/N)}.
	\end{align*}
	Since
	\[
		\frac{2^n r}{\rho_n} = 2^{n(1-\alpha)} \bigg(\frac{\psi^{-1}(1/2 \vee t)}{r}\bigg)^{-\frac{1}{N}}
	\]
	we obtain
	\[
		\sum_{n = 0}^\infty
		\frac{V(x, 2^n r)}{V(x, \psi^{-1}(1/2 \vee t))} \bigg(\frac{2^n r}{\psi^{-1}(1/2 \vee t)}\bigg)^{\gamma_2}
		\exp\bigg(-c_2 \frac{2^n r}{\rho_n} \bigg)
		\leq
		C_6
		\bigg(\frac{\psi^{-1}(1/2 \vee t)}{r}\bigg)^{-2 \frac{\gamma_2}{1-\alpha}(1-\alpha-1/N)+1/N}
	\]
	proving the claim for
	\[
		\delta = \min\left\{\beta_1-\frac{\beta_1+\gamma_2}{N}, -2 \gamma_2 + \frac{1}{N}\bigg(1-\frac{2\gamma_2}{1-\alpha}\bigg)
		\right\}.
	\]
	Now, let us fix $x_0 \in M$ and $r > 0$. If $d(x, x_0) < r/4$, then for all $t > 0$,
	\begin{align*}
		\PP^x\big(\tau_\mathbf{Y}(x_0, r) \leq t\big)
		\leq
		\PP^x\big(\tau_\mathbf{Y}(x, r/2) \leq t\big).
	\end{align*}
	If $0 < t < 1$, then 
	\begin{align}
		\label{eq:A:75}
		\PP^x\big(\tau_{\mathbf{Y^{(\rho)}}}(x_0, r) > t\big)
		=1
		\leq
		C_7 \bigg(\frac{\psi^{-1}(1/2 \vee t)}{r}\bigg)^{\delta}.
	\end{align}
	If $t \geq 1$, then an application of Claim \ref{clm:A:3} leads to
	\begin{align*}
		\PP^x\big(\tau_\mathbf{Y}(x, r/2) \leq t\big) 
		&=
		\PP^x\big(\tau_\mathbf{Y}(x, r/2) \leq t; Y_{2t} \in B(x, r/4)^c\big) 
		+ \PP^x\big(\tau_\mathbf{Y}(x, r/2) \leq t; Y_{2t} \in B(x, r/4) \big) \\
		&\leq
		\PP^x\big(Y_{2t} \in B(x, r/4)^c \big) + \sup_{0 < s \leq t} \max_{y \notin B(x, r/2)}
		\PP^y\big(Y_{2t -s} \in B(y, r/4)^c\big) \\
		&\leq
		C_7 \bigg(\frac{\psi^{-1}(2t)}{r}\bigg)^{\delta}.
	\end{align*}
	Hence, by \cite[Proposition 4.6]{GrigoryanHuLau2014}, for all $\rho > 0$,
	\begin{align}
		\nonumber
		\PP^x\big(\tau_{\mathbf{Y^{(\rho)}}}(x_0, r) > t\big)
		&\leq
		\PP^x\big(\tau_{\mathbf{Y}}(x_0, r) > t\big)
		+
		C_8
		\frac{t}{\psi(\rho)} \\
		\nonumber
		&\leq
		C_9 \bigg(\frac{\psi^{-1}(2 t)}{r}\bigg)^{\delta} + C_8 \frac{t}{\psi(\rho)} \\
		\label{eq:A:68}
		&\leq
		C_{10} \bigg(\frac{\psi^{-1}(t)}{r}\bigg)^{\delta} + C_8 \frac{t}{\psi(\rho)}
	\end{align}
	where in the last estimate we have used \eqref{eq:A:45}. We claim the following holds true.
	\begin{claim}
		\label{clm:A:4}
		There is $c > 0$ such that for each $k \in \NN$, there is $C > 0$ such that for all $t > 0$, $\rho > 0$, $x, y \in M$, if
		$d(x_0, y_0) > 4 k \rho$, then
		\[
			q^{(\rho)}(t; x, y) 
			\leq
			C
			\bigg(\frac{1}{V(x, \psi^{-1}(1/2 \vee t))} + \frac{1}{V(y, \psi^{-1}(1/2 \vee t))} \bigg)
			\exp\bigg(c \frac{t}{\psi(\rho)}\bigg) \bigg(1 + \frac{\rho}{\psi^{-1}(1/2 \vee t)}\bigg)^{-(k-1)\delta'}
		\]
		for all $x \in B(x_0, \rho)$ and $y \in B(y_0, \rho)$ where $\delta' = \delta \wedge \beta_2$.
	\end{claim}
	In view of \eqref{eq:A:30}, it is enough to consider $\rho \geq \psi^{-1}(1/2 \vee t)$. Let us fix $k \in \NN$, $t \geq 1$,
	$x_0, y_0 \in M$. Set $r = d(x_0, y_0)/2 > 2 k \rho$. Using \cite[Theorem 3.1]{GrigoryanHuLau2014} from \eqref{eq:A:68}
	we deduce that
	\[
		\int_{B(x_0, r)^c} q^{(\rho)}(t; x, z) \, \mu({\rm d} z)
		\leq
		C_1 \bigg\{\bigg(\frac{\psi^{-1}(t)}{\rho}\bigg)^{\delta} 
		+ \frac{t}{\psi(\rho)} \bigg\}^{k-1},
	\]
	for all $x \in B(x_0, \rho)$. Next, by \eqref{eq:A:45},
	\begin{equation}
		\label{eq:A:70}
		\frac{t}{\psi(\rho)} 
		\leq C_\psi \bigg(\frac{\psi^{-1}(t)}{\rho} \bigg)^{\beta_1}
	\end{equation}
	thus
	\begin{equation}
		\label{eq:A:69}
		\int_{B(x_0, r)^c} q^{(\rho)}(t; x, z) \, \mu({\rm d} z)
		\leq C_2 \bigg(\frac{\rho}{\psi^{-1}(1/2 \vee t)}\bigg)^{-(k-1)\delta'},
	\end{equation}
	for all $x \in B(x_0, \rho)$. To extend \eqref{eq:A:69} for all $t > 0$, we use \cite[Theorem 3.1]{GrigoryanHuLau2014} 
	and \eqref{eq:A:75}. Next, we write
	\begin{align*}
		q^{(\rho)}(2t; x, y) 
		&= \int_M q^{(\rho)}(t; x, z) q^{(\rho)}(t; z, y) \, \mu({\rm d} z) \\
		&\leq
		\int_{B(x_0, r)^c} q^{(\rho)}(t; x, z) q^{(\rho)}(t; z, y) \, \mu({\rm d} z) +
		\int_{B(y_0, r)^c} q^{(\rho)}(t; x, z) q^{(\rho)}(t; z, y) \, \mu({\rm d} z).
	\end{align*}
	Using by \eqref{eq:A:30} and \eqref{eq:A:69}, we obtain
	\begin{align*}
		&
		\int_{B(x_0, r)^c} q^{(\rho)}(t; x, z) q^{(\rho)}(t; z, y) \, \mu({\rm d} z) \\
		&\qquad\leq
		\frac{c_1}{V(y, \psi^{-1}(1/2 \vee t))} \exp\bigg(c_2 \frac{t}{\psi(\rho)} \bigg)
		\int_{B(x_0, r)^c} q^{(\rho)}(t; x, z) \, \mu({\rm d} z) \\
		&\qquad\leq
		\frac{C_3}{V(y, \psi^{-1}(1/2 \vee t))} 
		\exp\bigg(c_2 \frac{t}{\psi(\rho)}\bigg) \bigg(1 + \frac{\rho}{\psi^{-1}(1/2 \vee t)} \bigg)^{-(k-1)\delta'}
	\end{align*}
	and the claim follows.

	We are now in the position to finish the proof of the theorem. In view of \eqref{eq:A:29}, it is enough to consider
	$r > \psi^{-1}(1/2 \vee t)$. We take
	\[
		k \geq \frac{\gamma_2 + 2 \beta_2}{\delta'} + 1,
		\qquad\text{and}\qquad \rho  =\frac{r}{8 k}.
	\]
	By \eqref{eq:A:27}, \eqref{eq:A:31} and Claim \ref{clm:A:4}, for all $x \in B(x_0, \rho)$ and $y \in B(y_0, \rho)$,
	\begin{align*}
		p(t; x, y) 
		&\leq \frac{C_4}{V(x, \psi^{-1}(1/2 \vee t))} 
		\bigg(1 + \frac{d(x, y)}{\psi^{-1}(1/2 \vee t)} \bigg)^{\beta_2} 
		\bigg(1 + \frac{\rho}{\psi^{-1}(1/2 \vee t)}\bigg)^{-(k-1)\delta'}
		+
		\frac{C_5 t}{V(x, \rho) \psi(\rho)} \\
		&\leq
		\frac{C_6}{V(x, \psi^{-1}(1/2 \vee t))} \bigg(\frac{r}{\psi^{-1}(1/2 \vee t)}\bigg)^{\beta_2-(k-1)\delta'}
		+
		\frac{C_7 t}{V(x, r) \psi(r)}. 
	\end{align*}
	By \eqref{eq:A:44} and \eqref{eq:A:46}
	\[
		\frac{V(x, r)}{V(x, \psi^{-1}(1/2 \vee t))} \leq C_\psi \bigg(\frac{r}{\psi^{-1}(1/2 \vee t))}\bigg)^{\gamma_2}
	\]
	and
	\[
		\frac{\psi(r)}{1/2 \vee t} \leq C_\psi \bigg(\frac{r}{\psi^{-1}(1/2 \vee t)} \bigg)^{\beta_2},
	\]
	thus by \eqref{eq:A:70} and the choice of $k$, we get
	\begin{align*}
		p(t; x, y)
		&\leq
		C_8
		\frac{1/2 \vee t}{V(x, r) \psi(r)} \bigg(\frac{r}{\psi^{-1}(1/2 \vee t)}\bigg)^{\gamma_2 + 2 \beta_2-(k-1)\delta'}
		+
		\frac{C_7 t}{V(x, r) \psi(r)} \\
		&\leq
		C_9 \frac{1/2 \vee t}{V(x, r) \psi(r)},
	\end{align*}
	and the theorem follows.
\end{proof}
Having proved \eqref{eq:A:32}, we can follow the arguments in \cite[Lemma 2.7]{Chen2016} to show that there is $c \geq 1$
such that for all $x_0 \in M$, $r > 0$ and $t \geq 1/2$,
\begin{equation}
	\label{eq:A:20}
	\PP^{x_0}\big(\tau_{\mathbf{Y}}(x_0, r) \leq t\big) \leq c \frac{t}{\psi(r)}.
\end{equation}
Now let us return to study the Markov chain $\mathbf{X}$. For $x_0 \in M$ and $r > 0$ we set
\[
	\tau_{\mathbf{X}}(x_0, r) = \inf\big\{n \in \NN : X_n \notin B(x_0, r) \big\}.
\]
To find the estimates on exit times $\tau_{\mathbf{X}}$ we compare it with $\tau_{\mathbf{Y}}$.
\begin{lemma}
	\label{lem:A:4}
	There is $c > 0$ such that for all $x_0 \in M$, $r > 0$ and $\lambda > 0$,
	\[
		\PP^{x_0}\big(\tau_X(x_0, r) \leq \lambda \big) \leq c \frac{\lambda}{\psi(r)} + \frac{1}{16 \lambda}.
	\]
\end{lemma}
\begin{proof}
	Let us denote by $(T_k : k \in \NN)$ an increasing process on $\NN_0$ such that
	\[
		N_t = k \qquad\text{for all}\quad t \in [T_k, T_{k+1}).
	\]
	Observe that $T_k$ is Gamma distribution with scale $1$ and mean $k$, independent of $\mathbf{X}$. We have
	\begin{align*}
		\PP^{x_0}\Big(\tau_{\mathbf{X}}(x_0, r) \leq K\Big) 
		&=
		\PP^{x_0}\Big(\max_{1 \leq k \leq K} d(X_k, x_0) \geq r\Big) \\
		&=
		\PP^{x_0}\Big(\max_{1 \leq k \leq K} d(X_k, x_0) \geq r; T_K \leq 5K \Big)
		+
		\PP^{x_0}\Big(\max_{1 \leq k \leq K} d(X_k, x_0) \geq r; T_K > 5K\Big),
	\end{align*}
	thus
	\begin{align*}
		\PP^{x_0}\big(\tau_{\mathbf{X}}(x_0, r) \leq K\big)
		\leq 
		\PP^{x_0}\big(\tau_{\mathbf{Y}}(x_0, r) \leq 5K \big)
		+
		\PP^{x_0}(T_K - K > 4K).
	\end{align*}
	By the Markov inequality
	\[
		\PP^{x_0}(T_K - K > 4K) \leq \frac{1}{16 K},
	\]
	hence by \eqref{eq:A:20},
	\[
		\PP^{x_0}\big(\tau_{\mathbf{X}}(x_0, r) \leq K\big)
		\leq
		c \frac{5K}{\psi(r)} + \frac{1}{16 K}
	\]
	as stated.
\end{proof}

\begin{corollary}
	\label{cor:A:1}
	There is $\delta_0 \in (0, 1)$ such that for all $\delta \in (0, \delta_0]$, $x_0 \in M$ and $r > 0$,
	\[
		\PP^{x_0}\big(\tau_{\mathbf{X}}(x_0, r/2) \leq \delta \psi(r) \big) \leq \frac{1}{8}.
	\]
\end{corollary}
\begin{proof}
	By Lemma \ref{lem:A:4} and \eqref{eq:A:46}
	\begin{align*}
		\PP^{x_0}\big(\tau_{\mathbf{X}}(x_0, r/2) \leq \delta \psi(r) \big)
		&\leq
		c \delta \frac{\psi(r)}{\psi(r/2)} + \frac{1}{16 \delta \psi(r)} \\
		&\leq
		c C_\psi 2^{\beta_2} \delta + \frac{1}{16 \delta \psi(r)}.
	\end{align*}
	Select $\delta_0 = (2^{4+\beta_2} c C_\psi)^{-1}$ and let $\delta \in (0, \delta_0]$. If $\delta \psi(r) < 1$ then
	\[
		\PP^{x_0}\big(\tau_{\mathbf{X}}(x_0, r/2) \leq \delta \psi(r) \big) = 0
	\]
	otherwise
	\[
		\PP^{x_0}\big(\tau_{\mathbf{X}}(x_0, r/2) \leq \delta \psi(r) \big)
		\leq \frac{1}{16} + \frac{1}{16} = \frac{1}{8}. \qedhere
	\]
\end{proof}
Let us fix
\begin{equation}
	\label{eq:A:67}
	B = \frac{3}{r_0},
	\qquad
	\delta = \min\left\{C_\psi^{-1} B^{-\beta_2}, C_\psi^{-1} \eta^{-\beta_1} B^{-\beta_1}, \delta_0\right\},
	\qquad
	b = \max\left\{1 + \frac{2}{B}, (3 \delta C_\psi)^{1/\beta_1}\right\}
\end{equation}
where $\eta$ is determined in Lemme \ref{lem:A:10}. Next, we introduce the hitting times, that is
\[
	T_{\mathbf{X}}(x_0, r) = \inf\big\{n \in \NN_0 : X_k \in B(x_0, r)\big\}
\]
and
\[
	T_{\mathbf{Y}}(x_0, r) = \inf\big\{t > 0 : Y_t \in B(x_0, r)\big\}.
\]
\begin{lemma}
	\label{lem:A:6}
	There is $C > 0$ such that for all $x, y \in M$, $x \neq y$, and $r > 0$,
	\[
		\PP^x\big(T_{\mathbf{X}}(y, r) \leq \psi(r) \big)
		\leq
		C \frac{V(y, r)}{V(x, d(x, y))} \cdot \frac{\psi(r)}{\psi(d(x, y))}.
	\]
\end{lemma}
\begin{proof}
	Since $x \neq y$ and for all $r \in (0, r_0)$, $\psi(r) < 1$, we have
	\[
		\PP^x\big( T_{\mathbf{X}}(y, r) \leq \psi(r) \big) = 0.
	\]
	Next, let us consider $r_0 \leq r < r' = \max\{1, \psi^{-1}(2 c)\}$ where $c$ is the constant in \eqref{eq:A:20}. Since
	\[
		\left\{ T_{\mathbf{X}}(y, r) \leq \psi(r) \right\}
		\subseteq
		\left\{ T_{\mathbf{X}}(y, r') \leq \psi(r') \right\}
	\]
	and by \eqref{eq:A:44} and \eqref{eq:A:46}
	\begin{align*}
		V(y, r') \psi(r') 
		&= \frac{V(y, r')}{V(y, r)} \cdot \frac{\psi(r')}{\psi(r)} \\
		&\leq C_V C_\psi \bigg(\frac{r'}{r}\bigg)^{\beta_2+\gamma_2} V(y, r) \psi(r).
	\end{align*}
	it is enough to treat the case $r \geq r'$. We start by showing the corresponding statement for the continuous
	time process.
	\begin{claim}
		\label{clm:A:1}
		There is $C > 0$ such that for all $x, y \in M$ and $r \geq r'$,
		\[
			\PP^x\Big(T_{\mathbf{Y}}(y, r) \leq \psi(r)\Big) 
			\leq
			C \frac{V(y, r)}{V(x, d(x, y))} \cdot \frac{\psi(r)}{\psi(d(x, y))}.
		\]
	\end{claim}
	By the monotonicity, the inequality holds true if $d(x, y) \leq 4r$, thus it is enough to consider $d(x, y) > 4 r$. 
	Let $t \geq 1$. To simplify notation we set $T = T_{\mathbf{Y}}(y, r)$. By the strong Markov
	property and \eqref{eq:A:20}
	\begin{align*}
		\PP^x\Big(T \leq t; \sup_{T \leq s \leq T+t} d(Y_s, Y_T) < r \Big)
		&=
		\EE^x\Big[\ind{\{T \leq t\}} \cdot \PP^{Y_T}\big[\tau_{\mathbf{Y}}(Y_T, r) >t \big] \Big] \\
		&\geq 
		\bigg(1 - c \frac{t}{\psi(r)}\bigg) \PP^x\big(T \leq t \big).
	\end{align*}
	Hence,
	\begin{align*}
		\bigg(1 - c \frac{t}{\psi(r)}\bigg)
		\PP^{x}\big(T \leq t \big) 
		&\leq
		\PP^x\Big(T \leq t; \sup_{T \leq s \leq T+t} d(Y_s, Y_T) < r \Big) \\
		&\leq
		\PP^x\Big(d(Y_t, y) < 2r \Big) \\
		&=
		\int_{B(y, 2 r)} p(t; x, w) \, \mu({\rm d} w).
	\end{align*}
	By \eqref{eq:A:32},
	\[
		p(t; x, w) \leq \frac{C t}{V(x, d(x, w)) \psi(d(x, w))},
	\]
	and since $d(x, y) > 4r$ for $w \in B(y, 2r)$, we have
	\[
		d(y, w) < 2r < \tfrac{1}{2} d(x, y)
	\]
	thus
	\[
		d(x, w) \geq d(x, y) - d(y, w) > \tfrac{1}{2} d(x, y).
	\]
	Hence, by \eqref{eq:A:32}, \eqref{eq:A:44} and \eqref{eq:A:46},
	\begin{align*}
		p(t; x, w) 
		&\leq \frac{C t}{V(x, d(x, w)) \psi(d(x, w))} \\
		&\leq \frac{C' t}{V(x, d(x, y)) \psi(d(x, y))}.
	\end{align*}
	Therefore,
	\[
		\bigg(1 - c\frac{t}{\psi(r)}\bigg)
		\PP^{x}\big(T \leq t \big)
		\leq
		C'' \frac{V(y, 2r)}{V(x, d(x,y))} \cdot \frac{t}{\psi(d(x, y))}.
	\]
	Taking $t = \frac{\psi(r)}{2c} \geq 1$, we conclude the proof of the claim.
	
	Now, let us consider the Markov chain $\mathbf{X}$. Since the arrival time satisfies
	\[
		\PP(T_n > 2n) \leq \frac{1}{n},
	\]
	by taking $n = \lceil \psi(r) \rceil$, we can write
	\begin{align*}
		\frac{1}{2} \PP^x\big(T_{\mathbf{X}}(y, r) \leq \psi(r)\big)
		&\leq
		\PP^x\big(T_{\mathbf{X}}(y, r) \leq \psi(r); T_n \leq 2n \big) \\
		&\leq
		\PP^x\big(T_{\mathbf{Y}}(y, r) \leq 2 \psi(r)\big).
	\end{align*}
	In view of Claim \ref{clm:A:1}, we conclude the proof.
\end{proof}
We prove the Harnack inequality using idea from \cite{BassLevin}. We define a Markov chain $\mathbf{Z} = \big((V_k, X_k) : 
k \in \NN_0\big)$ where $V_k = V_0 + k$, $V_0 \in \NN_0$. Let $\calF_k$ denote $\sigma$-field generated by
$\big(Z_j : 0 \leq j \leq k\big)$. For $x \in M$, $r > 0$ and $k \in \NN_0$, we set
\[
	Q(k; x, r) = \big([k, k+\delta \psi(r)] \cap \NN_0\big) \times B(x, r)
\]
where $\delta > 0$ is defined in \eqref{eq:A:67}. Let
\[
	\tau_{\mathbf{Z}}(k; x, r) = \inf\big\{j \in \NN_0 : Z_j \notin Q(k; x, r)\big\},
\]
and $\tau_{\mathbf{Z}}(x, r) = \tau_{\mathbf{Z}}(0; x, r)$. Observe that
\[
	\tau_{\mathbf{Z}}(x, r) \leq \lfloor \delta \psi(r) \rfloor + 1.
\]
For $U \subset \NN_0 \times M$, we set
\[
	T_{\mathbf{Z}}(U) = \inf\big\{k \in \NN_0: Z_k \in U \big\}.
\]
\begin{lemma}
	\label{lem:A:8}
	There is $\theta_1 \in (0, \delta^{-1})$ such that for all $x \in M$, $r \geq \psi^{-1}(\delta^{-1})$,
	and all subset $U \subseteq Q(0; x, r)$ with $U(0) = \emptyset$,
	\[
		\PP^{(0, x)}\Big(T_{\mathbf{Z}}(U) < \tau_{\mathbf{Z}}(x, r)\Big)
		\geq
		\frac{\theta_1}{V(x, r) \psi(r)} \sum_{1 \leq k \leq \delta \psi(r)} \mu(U(k))
	\]
	where
	\[
		U(k) = \{u \in M : (k, u) \in U\}.
	\]
\end{lemma}
\begin{proof}
	Notice that 
	\[
		\sum_{1 \leq k \leq \delta \psi(r)} \mu(U(k)) \leq \delta \psi(r) V(x, r),
	\]
	thus it is enough to consider the case when
	\begin{equation}
		\label{eq:A:21}
		\PP^{(0, x)}\Big(T_{\mathbf{Z}}(U) < \tau_{\mathbf{Z}}(x, r)\Big) < \frac{1}{4}.
	\end{equation}
	Let $S = \min\big\{ T_{\mathbf{Z}}(U), \tau_{\mathbf{Z}}(x, r)\big\}$. For $k \in \NN$ and $y \in M$, we set
	\[
		N(k, y)= 
		\begin{cases}
			\PP^{(k, y)}\big(X_1 \in U(k+1)\big) &\text{if } y \notin U(k), \\
			0 & \text{otherwise,}
		\end{cases}
	\]
	where $U(k) = \{u \in M : (k, u) \in U\}$. By \cite[Lemma 4.3]{Murugan2015}, $(J_{n \wedge T_{\mathbf{Z}}(U)} : n \in \NN_0)$
	where
	\[
		J_n = \ind{U}(Z_n) - \ind{U}(Z_0) - \sum_{k = 0}^{n-1} N(Z_k), 
	\]
	is a martingale with respect to $(\calF_n : n \in \NN_0)$. Hence, by the optional stopping theorem, we obtain
	\begin{align*}
		\PP^{(0, x)}\Big(T_{\mathbf{Z}}(U) < \tau_{\mathbf{Z}}(x, r)\Big)
		&=
		\EE^{(0, x)}\big[\ind{U}(S, X_S)\big] \\
		&\geq
		\EE^{(0, x)}\Big[
		\sum_{k = 0}^{S-1} N(k, X_k)
		\Big].
	\end{align*}
	Next, by \eqref{jump} for $y \in B(x, r) \setminus U(k)$,
	\begin{align*}
		N(k, y) 
		&= \int_{U(k+1)} J(y, u) \, \mu({\rm d} u) \\
		&\geq
		\int_{U(k+1)} \frac{C}{V(y, d(y, u)) \psi(d(y, u))} \, \mu({\rm d} u).
	\end{align*}
	Since for $U(k+1) \subset B(x, r)$, we have
	\begin{align*}
		N(k, y)
		&\geq
		C \frac{\mu(U(k+1))}{V(x, 2r) \psi(2r)} \\
		&\geq
		C' \frac{\mu(U(k+1))}{V(x, r) \psi(r)}
	\end{align*}
	where in the last estimate we have used \eqref{eq:A:44} and \eqref{eq:A:46}. Hence,
	\begin{align*}
		\ind{\big\{S \geq \lfloor \delta \psi(r) \rfloor \big\}} \sum_{k = 0}^{S - 1}
		N(k, X_k) \geq \frac{C''}{V(x, r) \psi(r)} \sum_{1 \leq k \leq \delta \psi(r)} \mu(U(k)).
	\end{align*}
	Since $\tau_{\mathbf{Z}}(x, r) \leq \lfloor \delta \psi(r) \rfloor + 1 $ and 
	$T_{\mathbf{Z}}(U) \neq \tau_{\mathbf{Z}}(x, r)$, by \eqref{eq:A:21} and Corollary \ref{cor:A:1} we obtain
	\begin{align*}
		\PP^{(0, x)}\big(S \geq \lfloor \delta\psi(r) \rfloor \big) 
		&\geq 
		1 - \PP^{(0, x)}\big(T_{\mathbf{Z}}(U) \leq \tau_{\mathbf{Z}}(x, r) \big) 
		- \PP^{(0, x)}\big(\tau_{\mathbf{Z}}(x, r) \leq \delta \psi(r) \big) \\
		&\geq
		\frac{5}{8}.
	\end{align*}
	Consequently,
	\begin{align*}
		\PP^{(0, x)}\Big(T_{\mathbf{Z}}(U) < \tau_{\mathbf{Z}}(x, r)\Big)
		&\geq
		\PP^{(0, x)} \big(S \geq \lfloor \delta\psi(r) \rfloor \big)
		\frac{C''}{V(x, r) \psi(r)} \sum_{1 \leq k \leq \delta \psi(r)} \mu(U(k)) \\
		&\geq 
		\frac{c_1}{V(x, r) \psi(r)} \sum_{1 \leq k \leq \delta \psi(r)} \mu(U(k))
	\end{align*}
	which completes the proof.
\end{proof}

\begin{lemma}
	\label{lem:A:7}
	There is $\theta_2 > 0$ such that for all $R > r_0$, $z \in M$, if $R/2 \geq r > 0$, $x \in B(z, R/2)$, and
	$k \in \NN \cap [\delta \psi(r) + 1, \delta \psi(R/2)]$ then
	\[
			\PP^{(0, z)}\Big(T_{\mathbf{Z}}(U(k; x, r)) < \tau_{\mathbf{Z}}(z, R)\Big) 
			\geq \theta_2 
			\frac{V(x, r/2)}{V(x, R)} \cdot \frac{\psi(r)}{\psi(R)}
	\]
	where
	\[
		U(k; x, r) = \{k\} \times B(x, r).
	\]
\end{lemma}
\begin{proof}
	Given $k \in \NN \cap [\delta \psi(r) + 1, \delta \psi(R/2)]$ we set 
	$Q' = (\NN \cap [k - \delta\psi(r), k]) \times B(x, r/2)$. Since $B(x, r/2) \subset B(z, R)$, we have
	\[
		Q' \subset (\NN \cap [1, \delta \psi(R)]) \times B(z, R).
	\]
	Since
	\[
		\sum_{1 \leq j \leq \delta \psi(R)} \mu(Q'(j))= \lceil \delta \psi(r) \rceil \cdot V(x, r/2)
	\]
	by Lemma \ref{lem:A:8}, we obtain
	\begin{align}
		\nonumber
		\PP^{(0, z)}\big(T_{\mathbf{Z}}(Q') < \tau_{\mathbf{Z}}(z, R) \big)
		&\geq C' \frac{V(x, r/2)}{V(z, R)} \cdot \frac{\psi(r)}{\psi(R)} \\
		\label{eq:A:25}
		&\geq C'' \frac{V(x, r/2)}{V(x, R)} \cdot \frac{\psi(r)}{\psi(R)}
	\end{align}
	where in the last inequality we have used \eqref{eq:A:44} thanks to $B(z, R) \subset B(x, 3R/2)$. Now, the strong Markov
	property gives
	\begin{align}
		\nonumber
		\PP^{(0, z)}\big(T_{\mathbf{Z}}(U(k; x, r)) < \tau_{\mathbf{Z}}(z, R)\big)
		&\geq
		\PP^{(0, z)}\big(T_{\mathbf{Z}}(U(k; x, r)) < \tau_{\mathbf{Z}}(z, R);
		T_{\mathbf{Z}}(Q') < \tau_{\mathbf{Z}}(z, R)\big)\\
		\label{eq:A:24}
		&=
		\EE^{(0, z)}\Big[ \ind{\{T_{\mathbf{Z}}(Q') < \tau_{\mathbf{Z}}(z, R)\}}
		\PP^{Z_{T_{\mathbf{Z}}(Q')}}
		\big(T_{\mathbf{Z}}(U(k; x, r)) < \tau_{\mathbf{Z}}(z, R)\big)\Big].
	\end{align}
	Each path that stays in
	\[
		\NN \times B\big(X_{T_{\mathbf{Z}}(Q')}, r/2\big)
	\]
	for at least $\delta \psi(r)$ steps corresponds to a path of $\mathbf{Z}$ that starts at $Z_{T_{\mathbf{Z}}(Q')}$
	hits $Q(0; z, R)$ before leaving $U(k; x, r)$. Hence, by Corollary \ref{cor:A:1},
	\[
		\PP^{Z_{T_{\mathbf{Z}}(Q')}}
		\big( T_{\mathbf{Z}}(U(k; x, r)) < \tau_{\mathbf{Z}}(x, r)\big)
		\geq
		\PP^{X_{T_{\mathbf{Z}}(Q')}}
		\big(\tau_{\mathbf{X}}(X_{T_{\mathbf{Z}}(Q')}, r/2) \geq T_{\mathbf{Z}}(Q') + \delta \psi(r)\big) \geq \frac{7}{8}.
	\]
	Therefore, by \eqref{eq:A:24} and \eqref{eq:A:25}, we obtain
	\[
		\PP^{(0, z)}\big(T_{\mathbf{Z}}(U(k; x, r)) < \tau_{\mathbf{Z}}(x, R)\big)
		\geq 
		\frac{7}{8} C''\frac{V(x, r/2)}{V(x, R)} \cdot \frac{\psi(r)}{\psi(R)} 
	\]
	proving the lemma.
\end{proof}

\begin{lemma}
	\label{lem:A:5}
	Let $f$ be a nonnegative integrable function supported on $\NN_0 \times B(x, 2r)^c$ for certain $x \in M$ and $r > 0$.
	There is $\theta_3 > 0$, such that for all $y \in B(x, r/2)$,
	\[
		\EE^{(0, x)}
		\big[f(Z_{\tau_{\mathbf{Z}}(x, r)}) \big]
		\leq
		\theta_3
		\EE^{(0, y)}
		\big[
		f(Z_{\tau_{\mathbf{Z}}(x, r)})
		\big].
	\]
	The constant $c$ is independent of $x$, $r$, and $f$.
\end{lemma}
\begin{proof}
	It is enough to consider simple functions. Since $f \geq 0$, we can also assume that $f$ is indicator of a 
	subset of $\NN_0 \times B(x, 2r)^c$. Let
	\[
		f = \delta_k \otimes \ind{U}
	\]
	for $k \in \NN_0$ and $U \subset M \setminus B(x, 2r)$. Since
	\[
		1 \leq \tau(x, r) \leq \lfloor \delta \psi(r) \rfloor + 1,
	\]
	we have $1 \leq k \leq \lfloor \delta \psi(r) \rfloor+1$. Next, we write 
	\begin{align*}
		\EE^{(0, y)} 
		\big[
		f(Z_{\tau_{\mathbf{Z}}(x, r)})
		\big]
		&=
		\EE^{(0, y)}
		\Big[\EE^{(0, y)}\big[
		\delta_k(V_{\tau_{\mathbf{Z}}(x, r)}) \ind{U}(X_{\tau_{\mathbf{Z}}(x, r)}) | \calF_{k-1} 
		\big]
		\Big] \\
		&=
		\int_U
		\EE^{(0, y)} \big[
		\ind{\{\tau_{\mathbf{Z}}(x, r) > k-1\}} J(X_{k-1}, u)
		\big]
		\, \mu({\rm d} u).
	\end{align*}
	Since $X_{k-1} \in B(x, r)$ on $\{\tau_{\mathbf{Z}}(x, r) > k-1\}$, we have
	\[
		d(X_{k-1}, u) \leq d(X_{k-1}, x) + d(x, u) \leq r + d(x, u)
	\]
	and
	\[
		d(X_{k-1}, u) \geq d(x, u) - d(X_{k-1}, x) \geq \frac{1}{2} d(x, u).
	\]
	Using $d(x, u) \geq 2r$, we obtain
	\[
		\frac{1}{2} d(x, u) \leq d(X_{k-1}, u) \leq \tfrac{3}{2} d(x, u).
	\]
	In view of \eqref{jump}, \eqref{eq:A:44} and \eqref{eq:A:46},
	\begin{align*}
		\EE^{(0, y)}
		\big[
		f(Z_{\tau_{\mathbf{Z}}(x, r)})
		\big]
		\approx
		\PP^{(0, y)}\big[\tau_{\mathbf{Z}}(x, r) > k-1\big] 
		\int_U \frac{1}{V(u, d(x, u)/2) \psi(d(x, u)/2)} 
		\, \mu({\rm d} u).
	\end{align*}
	Since $B(y, r/2) \subset B(x, r)$,
	\begin{align*}
		\PP^{(0, y)}\big( \tau_{\mathbf{Z}}(x, r) > k-1\big) 
		&\geq 
		\PP^y\big(\tau_{\mathbf{X}}(x, r) \geq k \big) \\
		&\geq
		\PP^y\big(\tau_{\mathbf{X}}(y, r/2) \geq k \big) \\
		&\geq
		\PP^y\big(\tau_{\mathbf{X}}(y, r/2) \geq \lfloor \delta \psi(r) \rfloor + 1\big) \\
		&=
		1 - \PP^y\big(\tau_{\mathbf{X}}(y, r/2) \leq \delta \psi(r) \big),
	\end{align*}
	and thus by Corollary \ref{cor:A:1},
	\[
		\PP^{(0, x)}\big( \tau_{\mathbf{Z}}(x, r) > k-1\big) \geq \frac{7}{8}
	\]
	and the lemma follows.
\end{proof}

\begin{lemma}
	\label{lem:A:10}
	There are $C > 0$ and $\eta \geq 2$ such that for all $x, y \in M$, if $d(x, y) \geq \eta \psi^{-1}(n)$
	then
	\[
		\PP^x\big(X_n \in B(y, \eta \psi^{-1}(n)\big) \geq 
		C
		\frac{V(y, \psi^{-1}(n))}{V(x, d(x, y))} \cdot \frac{n}{\psi(d(x, y))}.
	\]
\end{lemma}
\begin{proof}
	For $z \in M$, we set
	\[
		N(z) = 
		\begin{cases}
			\PP^{z}\big(X_1 \in B(y, \psi^{-1}(n))\big) & \text{if } z \notin B(y, \psi^{-1}(n) ), \\
			0 & \text{otherwise.}
		\end{cases}
	\]
	Let $S = \min\left\{n, \tau_{\mathbf{X}}(x, (\eta-1) \psi^{-1}(n)), T_{\mathbf{X}}(y, \psi^{-1}(n))\right\}$ and 
	$A = B(y, \psi^{-1}(n))$. Then by \cite[Lemma 4.3]{Murugan2015} the sequence $(J_{n \wedge T_{\mathbf{X}}(A)} : n \in \NN_0)$
	where
	\[
		J_n = \ind{A}(X_n) - \ind{A}(X_0) - \sum_{k = 0}^{n-1} N(X_k)
	\]
	is a martingale. Thus, by the optional stopping theorem, we obtain
	\[
		\PP^x\big(X_S \in B(y, \psi^{-1}(n))\big) \geq \EE^x\Big(\sum_{k = 0}^{S-1} N(X_k) \Big).
	\]
	For $z \in B(x, (\eta-1) \psi^{-1}(n)) \setminus B(y, \psi^{-1}(n))$, we have
	\begin{align*}
		N(z) 
		&= \int_{B(y, \psi^{-1}(n))} J(z, w) \, \mu({\rm d} w) \\
		&\geq \int_{B(y, \psi^{-1}(n))} \frac{C_1}{V(z, d(z, w)) \psi(d(z, w))} \, \mu({\rm d} w) \\
		&\geq C_1 \frac{V(y, \psi^{-1}(n))}{V(x, 2 d(x, y)) \psi(2 d(x, y))}
	\end{align*}
	since by \eqref{eq:A:46}, for $w \in B(y, \psi^{-1}(n))$ we have
	\begin{align*}
		d(z, w) 
		&\leq d(x, y) + d(z, x) + d(y, w) \\
		&\leq d(x, y) + (\eta-1) \psi^{-1}(n) + \psi^{-1}(n) \\
		&\leq 2 d(x, y)
	\end{align*}
	and
	\[
		B(z, d(z, w)) \subset B(z, 2 d(x, y)) \subset B(x, 3 d(x, y)).
	\]
	Consequently, by \eqref{eq:A:44} and \eqref{eq:A:46},
	\begin{equation}
		\label{eq:A:74}
		\PP^x\big(X_S \in B(y, \psi^{-1}(n)) \big) 
		\geq 
		C_4 \frac{V(y, \psi^{-1}(n))}{V(x, d(x, y)) \psi(d(x, y))}
		\EE^x S.
	\end{equation}
	Next, we estimate $\EE^x S$. By Lemma \ref{lem:A:6} together with \eqref{eq:A:44} and \eqref{eq:A:45},
	\begin{align*}
		 \PP^x\big(  T_{\mathbf{X}}(y, \psi^{-1}(n)) < n \big)
		 &\leq
		 C_5 \frac{V(y, \psi^{-1}(n))}{V(x, d(x, y))} \cdot \frac{n}{\psi(d(x, y))} \\
		 &\leq
		 C_5 \frac{V(x, d(x, y) + \psi^{-1}(n))}{V(x, d(x, y))} \cdot \frac{n}{\psi(d(x, y))} \\
		 &\leq
		 C_5 C_V C_\psi \bigg(1 + \frac{\psi^{-1}(n)}{d(x, y)}\bigg)^{\gamma_2} 
		 \bigg(\frac{\psi^{-1}(n)}{d(x, y)}\bigg)^{\beta_1} \\
		 &\leq
		 C C_V C_\psi 2^{\gamma_2} \eta^{-\beta_1}
	\end{align*}
	which is smaller than $\frac{1}{4}$, provided that
	\[
		\eta \geq \big(C_5 C_V C_\psi 2^{2 + \gamma_2}\big)^{\frac{1}{\beta_1}}.
	\]
	Moreover, by Lemma \ref{lem:A:4} and \eqref{eq:A:45}
	\begin{align*}
		\PP^x\Big( \tau_{\mathbf{X}}(x, (\eta-1) \psi^{-1}(n) ) < n\Big)
		&\leq
		c \frac{n}{\psi((\eta-1) \psi^{-1}(n))}
		+ \frac{1}{16 n} \\
		&\leq
		c C_\psi (\eta-1)^{-\beta_1} + \frac{1}{16}
	\end{align*}
	which is smaller than $\frac{1}{4}$, provided that
	\[
		\eta \geq \left( \tfrac{16}{3} c C_\psi \right)^{\frac{1}{\beta_1}} + 1.
	\]
	Hence,
	\[
		 \PP^x\Big(\tau_{\mathbf{X}}(x, (\eta-1) \psi^{-1}(n) ) \geq n; T_{\mathbf{X}}(y, \psi^{-1}(n)) \geq n\Big)
		 \geq 
		 \frac{1}{2},
	\]
	and thus, by the Markov inequality, we obtain
	\begin{align*}
		\frac{1}{2} 
		&\leq \PP^x\Big(\tau_{\mathbf{X}}(x, (\eta-1) \psi^{-1}(n) ) \geq n;
		T_{\mathbf{X}}(y, \psi^{-1}(n)) \geq n\Big) \\
		&\leq
		\PP^x(S \geq n) \\
		&\leq
		\frac{1}{n}
		\EE^x S.
	\end{align*}
	Consequently, by \eqref{eq:A:74}
	\[
		\PP^x\Big(T_{\mathbf{X}}(y, \psi^{-1}(n)) \leq n \Big) \geq C_6
		\frac{V(y, \psi^{-1}(n))}{V(x, d(x, y))} \cdot \frac{n}{\psi(d(x, y))}.
	\]
	Next, by Lemma \ref{lem:A:4}, for all $z \in B(y, \psi^{-1}(n))$ we have
	\[
		\PP^z\Big(\tau_{\mathbf{X}}(y, \eta \psi^{-1}(n)) \geq n\Big)
		\geq
		\PP^z\Big(\tau_{\mathbf{X}}(z, (\eta-1) \psi^{-1}(n)) \geq n\Big)
		\geq
		\frac{3}{4}.
	\]
	Therefore, by the strong Markov property,
	\begin{align*}
		\PP^x\Big(X_n \in B(y, \eta \psi^{-1}(n))\Big)
		&=
		\EE^x\Big[ \ind{\big\{T_{\mathbf{X}}(y, \psi^{-1}(n)) \leq n\big\}}
		\PP^{X_{T_{\mathbf{X}}(y, \psi^{-1}(n))}}
		\Big(\tau_{\mathbf{X}}(y, \eta \psi^{-1}(n)) \geq n \Big)
		\Big] \\
		&\geq
		C_7 \frac{V(y, \psi^{-1}(n))}{V(x, d(x, y))} \cdot \frac{n}{\psi(d(x, y))}
	\end{align*}
	and the lemma follows.
\end{proof}

Let us recall that a function $u$ is parabolic on $D \subset \NN_0 \times M$ if
\[
	\Big( u(Z_{k \wedge \tau_{\mathbf{Z}}(D)}) : k \in \NN \Big)
\]
is a martingale. Let $B$ and $b$ be constants defined in \eqref{eq:A:67}.
\begin{theorem}
	\label{thm:A:3}
	There is $K_0 > 0$ such that for all $R > B$, $z \in M$ and every nonnegative function $u$ 
	which is parabolic on $Q(0; z, bR)$, the following inequality holds
	\[
		\max_{(k, y) \in Q(\delta \psi(R); z, R/B)}
		u(k, y)
		\leq K_0 \min_{y \in B(z, R/B)} u(0, y).
	\]
\end{theorem}
\begin{proof}
	Suppose, contrary to our claim, that there is $K_0 \geq 2$ such that for all $K \geq K_0$ there are $R > B$, $z \in M$
	and a nonnegative function $u$ which is parabolic in $Q(0; z, bR)$, such that
	\begin{equation}
		\label{eq:A:33}
		\min_{y \in B(z, R/B)} u(0, y) = 1
	\end{equation}
	and
	\begin{equation}
		\label{eq:A:34}
		\max_{(k, y) \in Q(\delta\psi(R); z, R/B)} u(k, y) > K.
	\end{equation}
	We are going to show that if $K_0$ is sufficiently large then we can construct a sequence of points 
	$((k_k, x_k) : k \in \NN) \subset Q(\delta \psi(R); z, R)$ such that
	\begin{equation}
		\label{eq:A:57}
		u(k_k, x_k) \geq (1+\rho)^{k-1} K
	\end{equation}
	for all $k \in \NN$ for certain universal $\rho > 0$. 
	
	First, let us observe that, by \eqref{eq:A:34}, there is $(k_1, y_1) \in Q(\delta \psi(R); z, R/B)$ such that
	\[
		u(k_1, x_1) \geq K.
	\]
	We set $K_1 = K$. Now, let us suppose that we have already constructed $((k_k, x_k) : 1 \leq k \leq j)$ belonging to 
	$Q(\delta \psi(R); z, R)$ and satisfying \eqref{eq:A:57} for
	\[
		\rho = \frac{\delta \theta_1}{2(2-\delta \theta_1)} > 0
	\]
	where $\theta_1$ is from Lemma \ref{lem:A:8}. Let
	\[
		U_j = \{k_j\} \times B(x_j, r_j/B)
	\]
	where
	\[
		r_j = \eta R K_j^{-\frac{1}{\gamma_2+\beta_2}},
		\qquad\text{and}\qquad K_j = u(k_j, x_j)
	\]
	for certain $\eta > 1$ satisfying \eqref{eq:A:40} and \eqref{eq:A:55}.  In view of \eqref{eq:A:33}, there is
	$x_0 \in B(z, R/B)$ such that
	\[
		u(0, x_0) = 1,
	\]
	thus
	\begin{align*}
		1 =  u(0, x_0) 
		&= \EE^{(0, x_0)}\Big[u\big(Z_{T_{\mathbf{Z}}(U_j) \wedge \tau_{\mathbf{Z}}(z, R)}\big)\Big] \\
		&\geq
		\EE^{(0, x_0)}\bigg[u\big(Z_{T_{\mathbf{Z}}(U_j)} \big) 
		\ind{\big\{T_{\mathbf{Z}}(U_j) \leq \tau_{\mathbf{Z}}(z, R)\big\}}
		\bigg] \\
		&\geq
		\Big(
		\min_{y \in B(x_j, r_j/B)} u(k_j, y) \Big)
		\PP^{(0, x_0)}
		\Big(T_{\mathbf{Z}}(U_j) \leq \tau_{\mathbf{Z}}(x_0, R/B)\Big)
	\end{align*}
	because  
	\[
		Q(0; x_0, R/B) \subset Q(0; z, R).
	\]
	Hence, by Lemma \ref{lem:A:7}
	\begin{align*}
		\min_{y \in B(x_j, r_j/B)} u(k_j, y) 
		&\leq
		\left(\PP^{(0, x_0)}  \Big(Z_{T_{\mathbf{Z}}(U_j) \leq \tau_{\mathbf{Z}}(x_0, R/B)}\Big) \right)^{-1} \\
		&\leq
		\frac{1}{\theta_2} C_j
	\end{align*}
	where 
	\[
		C_j = \frac{V(x_j, R/B)}{V(x_j, r_j/(2B))} \cdot \frac{\psi(R/B)}{\psi(r_j/B)}.
	\]
	Observe that by \eqref{eq:A:44} and \eqref{eq:A:46}
	\begin{equation}
		\label{eq:A:40}
		C_j
		\leq
		C_V C_\psi 2^{\gamma_2} \eta^{-\gamma_2-\beta_2} K_j,
	\end{equation}
	thus for
	\[
		\eta \geq \left(C_V C_\psi 2^{2+\gamma_2} \theta_2^{-1}\right)^{\frac{1}{\gamma_2+\beta_2}},
	\]
	we obtain
	\[
		\min_{y \in B(x_j, r_j/B)} u(k_j, y) 
		\leq 
		\frac{1}{4} K_j.
	\]
	In particular,
	\begin{equation}
		\label{eq:A:53}
		r_j > B r_0 \geq 3,
	\end{equation}
	and there is $y_j \in B(x_j, r_j/B)$ such that
	\begin{equation}
		\label{eq:A:37}
		u(k_j, y_j) \leq \frac{1}{\theta_2} C_j.
	\end{equation}
	Next, by Lemma \ref{lem:A:5}
	\begin{align*}
		\EE^{(k_j, x_j)}\bigg[u\big(Z_{\tau_{\mathbf{Z}}(k_j; x_j, r_j)} \big)
		\ind{\big\{X_{\tau_{\mathbf{Z}}(k_j; x_j, r_j)} \notin B(x_j, 2 r_j) \big\}}
		\bigg]
		&\leq
		\theta_3
		\EE^{(k_j, y_j)}\bigg[u\big(Z_{\tau_{\mathbf{Z}}(k_j; x_j, r_j)} \big)
		\ind{\big\{X_{\tau_{\mathbf{Z}}(k_j; x_j, r_j)} \notin B(x_j, 2 r_j) \big\}}
		\bigg] \\
		&\leq
		\theta_3 
		\EE^{(k_j, y_j)}\big[u\big(Z_{\tau_{\mathbf{Z}}(k_j; x_j, r_j)} \big)\big] \\
		&=
		\theta_3 u(k_j, y_j).
	\end{align*}
	Therefore, by \eqref{eq:A:37},
	\[
		\EE^{(k_j, x_j)}\bigg[u\big(Z_{\tau_{\mathbf{Z}}(k_j; x_j, r_j)} \big)
		\ind{\big\{X_{\tau_{\mathbf{Z}}(k_j; x_k, r_j)} \notin B(x_j, 2 r_j) \big\}}
		\bigg]
		\leq
		\frac{\theta_3}{\theta_2} C_j.
	\]
	Let us consider
	\[
		A_j = \left\{
		(k, y) \in Q(k_j+1; x_j, r_j/B) : u(k, y) \geq \frac{2 \sigma}{\theta_2} C_j 
		\right\}
	\]
	for certain $\sigma > 1$ satisfying \eqref{eq:A:56}. Observe that $A_j(0) = A_j(k_j) = \emptyset$. Moreover, for 
	$(k, y) \in A_j$, we have
	\begin{align*}
		d(y, x_0) 
		&\leq d(y, x_j) + d(x_j, z) + d(z, x_0) \\
		&\leq (1+2/B) R \leq b R,
	\end{align*}
	and
	\begin{align*}
		k_j + \delta \psi(r_j/B) + 1 
		&\leq 
		3 \delta \psi(R) \leq \psi(b R),
	\end{align*}
	i.e. $A_j \subset Q(0; x_0, b R)$. Therefore, by Lemma \ref{lem:A:8},
	\begin{align*}
		1 = u(0, x_0) &= \EE^{(0, x_0)}
		\Big[u\big(Z_{T_{\mathbf{Z}}(A_j) \wedge \tau_{\mathbf{Z}}(x_0, b R)} \big)\Big] \\
		&\geq
		\EE^{(0, x_0)}
		\bigg[
		u\big(Z_{T_{\mathbf{Z}}(A_j)} \big)
		\ind{\big\{T_{\mathbf{Z}}(A_j) \leq \tau_{\mathbf{Z}}(x_0, b R)\big\}} \bigg] \\
		&\geq
		\frac{2 \sigma}{\theta_2} C_j \cdot
		\PP^{(0, x_0)} \Big(T_{\mathbf{Z}}(A_j) \leq \tau_{\mathbf{Z}}(x_0, b R)\Big) \\
		&\geq
		\frac{2 \sigma}{\theta_2} C_j
		\frac{\theta_1}{V(x_0, bR) \psi(bR)} 
		\sum_{1\leq k \leq \delta \psi(bR)} \mu(A_j(k)).
	\end{align*}
	Consequently,
	\[
		\frac{\theta_1}{V(x_j, r_j/B)\psi(r_j/B)} \sum_{1\leq k \leq \delta \psi(bR)} \mu(A_j(k))
		\leq
		\frac{\theta_2}{2 \sigma} \frac{C'_j}{C_j}
	\]
	where
	\[
		C'_j = \frac{V(x_j, 2 bR)}{V(x_j, r_j/B)} \cdot \frac{\psi(2 b R)}{\psi(r_j/B)}.
	\]
	Observe that by \eqref{eq:A:44} and \eqref{eq:A:46},
	\begin{equation}
		\label{eq:A:39}
		\frac{C'_j}{C_j} 
		= \frac{V(x_j, 2b R)}{V(x_j, R/B)} \cdot \frac{\psi(2b R)}{\psi(R/B)} 
		\leq C_V C_\psi (2 b B)^{\gamma_2+\beta_2}.
	\end{equation}
	Let
	\[
		D_j = Q(k_j+1; x_j, r_j/B) \setminus A_j, \qquad M_j = \max_{(k, y) \in Q(k_j+1; x_j, 2r_j)} u(k, y).
	\]
	Then
	\begin{align*}
		K_j &= \EE^{(k_j, x_j)} \bigg[u\big(Z_{T_\mathbf{Z}(D_j)}\big) 
		\ind{\big\{T_{\mathbf{Z}}(D_j) \leq \tau_{\mathbf{Z}}(k_j; x_j, r_j)\big\}}\bigg] \\
		&\phantom{=}
		+\EE^{(k_j, x_j)} \bigg[u\big(Z_{\tau_\mathbf{Z}(k_j; x_j, r_j)}\big) 
		\ind{\big\{T_{\mathbf{Z}}(D_j) > \tau_{\mathbf{Z}}(k_j; x_j, r_j)\big\}} 
		\ind{\big\{X_{\tau_{\mathbf{Z}}(k_j; x_j, r_j)} \notin B(x_j, 2 r_j)\big\}}
		\bigg] \\
		&\phantom{=}
		+\EE^{(k_j, x_j)} \bigg[u\big(Z_{\tau_\mathbf{Z}(k_j; x_j, r_j)}\big) 
		\ind{\big\{T_{\mathbf{Z}}(D_j) > \tau_{\mathbf{Z}}(k_j; x_j, r_j)\big\}} 
		\ind{\big\{X_{\tau_{\mathbf{Z}}(k_j; x_j, r_j)} \in B(x_j, 2 r_j)\big\}}\bigg] \\
		&\leq
		\frac{1+\theta_3}{\theta_2} C_j + M_j 
		\left(
		1-
		\PP^{(k_j, x_j)}
		\Big(T_{\mathbf{Z}}(D_j) \leq \tau_{\mathbf{Z}}(k_j; x_j, r_j)\Big)
		\right).
	\end{align*}
	Thanks to Lemma \ref{lem:A:8} we have
	\begin{align*}
		1 - \PP^{(k_j, x_j)} \Big(T_{\mathbf{Z}}(D_j) \leq \tau_{\mathbf{Z}}(k_j; x_j, r_j)\Big)
		&\leq
		1 - \frac{\theta_1}{V(x_j, r_j/B) \psi(r_j/B)} \sum_{1 \leq k \leq \delta \psi(r_j)} \mu(D_j(k))  \\
		&\leq
		1 - \theta_1 \delta + \frac{\theta_2}{2 \sigma} \frac{C'_j}{C_j}.
	\end{align*}
	Therefore, we arrive at
	\[
		K_j\bigg(1 - \frac{1+\theta_3}{\theta_2} \frac{C_j}{K_j}\bigg)
		\leq
		M_j \bigg(1 - \theta_1 \delta + \frac{\theta_2}{2 \sigma} \frac{C'_j}{C_j}\bigg).
	\]
	Now, taking 
	\begin{equation}
		\label{eq:A:55}
		\eta \geq \left(
		2^{2 + \gamma_2} \frac{1+\theta_3}{\delta \theta_1 \theta_2} C_V C_\psi \right)^{\frac{1}{\gamma_2+\beta_2}}
	\end{equation}
	by \eqref{eq:A:44} and \eqref{eq:A:46}, we obtain
	\begin{align*}
		1 - \frac{2(1+\theta_3)}{\theta_2} \frac{C_j}{K_j}
		&\geq 1 - 2^{\gamma_2} C_V C_\psi \frac{1+\theta_3}{\theta_2} \eta^{-\gamma_2-\beta_2} \\
		&\geq 1 - \frac{\theta_1 \delta}{4}.
	\end{align*}
	Moreover, if
	\begin{equation}
		\label{eq:A:56}
		\sigma \geq C_V C_\psi \frac{\theta_2}{\delta \theta_1} (2 b B)^{\gamma_2+\beta_2} 
	\end{equation}
	then by \eqref{eq:A:39} we get
	\begin{align*}
		1 - \theta_1 \delta + \frac{\theta_2}{2 \sigma} \frac{C'_j}{C_j} 
		&\leq
		1 - \theta_1 \delta + C_V C_\psi \frac{\theta_2}{2 \sigma} (2 b B)^{\gamma_2+\beta_2} \\
		&\leq 1 - \frac{\theta_1 \delta}{2}.
	\end{align*}
	Consequently, $K_j (1 + \rho) \leq M_j$, and hence there is $(k_{j+1}, x_{j+1}) \in Q(k_j+1; x_j, 2r_j)$ 
	such that $K_{j+1} = M_j = u(k_{j+1}, x_{j+1})$. Thus
	\[
		K_{j+1} \geq (1 + \rho) K_j.
	\]
	It remains to check that $(k_{j+1}, x_{j+1}) \in Q(\delta \psi(R); z, R)$. Observe that
	\begin{align*}
		k_{j+1} 
		&\leq k_j + 1 + \delta \psi(2 r_j) \\
		&\leq k_j + (2+\delta) \psi(2 r_j) \\
		&\leq k_1 + (2 + \delta) \sum_{k = 1}^j \psi(2 r_k).
	\end{align*}
	In view of \eqref{eq:A:53}, for all $j \geq k \geq 2$, we have $2 r_k \geq 1$, thus by \eqref{eq:A:45},
	\begin{align*}
		\frac{\psi(2 r_k)}{\psi(2 r_{k-1})}
		\leq
		C_\psi \bigg(\frac{r_k}{r_{k-1}}\bigg)^{\beta_1}
		&=
		C_\psi\bigg(\frac{K_{k-1}}{K_k}\bigg)^{\frac{\beta_1}{\gamma_2+\beta_2}} \\
		&\leq
		C_\psi (1+\rho)^{-\frac{\beta_1}{\gamma_2+\beta_2}}.
	\end{align*}
	Hence,
	\begin{align*}
		k_{j+1} 
		&\leq \delta \psi(R) + \delta \psi(R/B) 
		+ C_\psi  (2+\delta) \psi(2 r_1) \sum_{k = 1}^j (1 + \rho)^{-k \frac{\beta_1}{\gamma_2+\beta_2}} \\
		&\leq
		\delta \psi(R) + \delta \psi(R/B)
		+ C_\psi (2 + \delta) \frac{\psi(2 r_1)}{1 + (1+\rho)^{-\frac{\beta_1}{\gamma_2+\beta_2}}}.
	\end{align*}
	If
	\[
		K_0 \geq \eta^{\gamma_2+\beta_2},
	\]
	then $2 r_1 \leq R$, thus by \eqref{eq:A:46}
	\[
		\psi(2 r_1) \leq C_\psi \psi(R) K^{-\frac{\beta_1}{\gamma_2+\beta_2}} \eta^{\beta_1} 2^{\beta_1},
	\]
	therefore, by taking
	\[
		K_0 \geq
		\left(
		\frac{2+\delta}{1-\delta} \cdot \frac{2^{\beta_1} \eta^{\beta_1} C_\psi^2}
		{1-(1+\rho)^{-\frac{\beta_1}{\gamma_2+\beta_2}}} 
		\right)^{\frac{\gamma_2 +\beta_2}{\beta_1}}
	\]
	we obtain
	\[
		k_j+1 \leq k_{j+1} \leq 2 \delta \psi(R).
	\]
	Moreover, 
	\begin{align*}
		d(z, x_{j+1}) 
		&\leq d(z, x_0) + d(x_0, x_1) + \ldots + d(x_j, x_{j+1}) \\
		&\leq R B^{-1} + 2 \eta R K^{-\frac{1}{\gamma_2+\beta_2}} 
		\sum_{k = 0}^\infty (1+\rho)^{-\frac{k}{\gamma_2+\beta_2}} \\
		&\leq R 
		\left(B^{-1} + K^{-\frac{1}{\gamma_2+\beta_2}} \frac{2 \eta}{1-(1+\rho)^{-\frac{1}{\gamma_2+\beta_2}}} \right),
	\end{align*}
	thus taking
	\[
		K_0
		\geq
		\bigg(\frac{2 \eta B}{1-(1+\rho)^{-\frac{1}{\gamma_2+\beta_2}}}\bigg)^{\gamma_2+\beta_2}
	\]
	we guarantee that $x_{j+1} \in B(z, R)$. Hence, $(k_{j+1}, x_{j+1}) \in Q(\delta \psi(R); z, R)$. This completes the
	induction. 

	Now, we notice that the constructed sequence $((k_j, x_j) : j \in \NN)$ is such that $(k_j : j \in \NN)$ is increasing
	and belongs to $[\delta \psi(R), 2 \delta \psi(R)] \cap \NN_0$ which leads to contradiction.
\end{proof}

We start by showing the upper near diagonal estimates on $h$.
\begin{theorem}
	\label{thm:A:2}
	There is $C > 0$ such that for all $x, y \in M$ and $n \in \NN$, if $d(x, y) \leq \psi^{-1}(n)$ then
	\[
		h(n; x, y) \leq \frac{C}{V(x, \psi^{-1}(n))}.
	\]
\end{theorem}
\begin{proof}
	First, let us show that
	\begin{equation}
		\label{eq:A:50}
		h(2n; x, x) \leq \frac{C}{V(x, \psi^{-1}(2 n))}.
	\end{equation}
	Since $h(n; x, z) = h(n; z, x)$, we have
	\begin{align*}
		h(2n; x, x) 
		&= \int_M h(n; x, z) h(n; z, x)  \, \mu({\rm d} z) \\
		&= \int_M h(n; x, z)^2 \, \mu({\rm d} z).
	\end{align*}
	Hence, for each $x \in M$, the function $n \mapsto h(2n; x, x)$ is decreasing. If $t \leq 2 k$, then
	\[
		3 \frac{e^{-t} t^k}{k!} \geq \frac{e^{-t} t^k}{k!} + \frac{e^{-t} t^{k+1}}{(k+1)!}
	\]
	thus, by \eqref{eq:A:18}, for $t \leq 2 n$ we get
	\begin{align*}
		p(t; x, x)
		&\geq
		\sum_{n < k \leq 2n} 
		h(k; x, x) \frac{e^{-t} t^k}{k!} \\
		&\geq
		h(2 n; x, x) \sum_{n/2 < k \leq n} \frac{e^{-t} t^{2k}}{(2k)!} \\
		&\geq
		\frac{1}{3} h(2n; x, x) \sum_{n < k \leq 2n} \frac{e^{-t} t^k}{k!} \\
		&\geq
		\frac{1}{3}
		h(2n; x, x) \PP\big(n < N_t \leq 2n \big).
	\end{align*}
	Now, by the Markov's inequality, for $t = \frac{3}{2} n$ we get
	\[
		1 - \PP\Big(\big|N_t - \tfrac{3}{2} n\big| \leq \tfrac{n}{2} \Big) 
		=
		\PP\Big(\big|N_t - \tfrac{3}{2} n\big| \geq \tfrac{n}{2} \Big)
		\leq 
		\frac{6}{n}.
	\]
	Hence, by Theorem \ref{thm:A:1}, we obtain
	\begin{align*}
		h(2n; x, x) 
		&\leq C_1 p\Big(\tfrac{3}{2} n; x, x\Big) \\
		&\leq \frac{C_2}{V(x, \psi^{-1}(3n/2))}. 
	\end{align*}
	Now, \eqref{eq:A:44} and \eqref{eq:A:45} imply that
	\begin{align*}
		\frac{V(x, \psi^{-1}(2n))}{V(x, \psi^{-1}(3n/2))} 
		&\leq
		C_V
		\bigg(\frac{\psi^{-1}(2n)}{\psi^{-1}(3n/2)}\bigg)^{\gamma_2} \\
		&\leq
		C_V \Big(\tfrac{4}{3}C_\psi \Big)^{\gamma_2/\beta_1},
	\end{align*}
	thus we get \eqref{eq:A:50}.

	Next, let us observe that by the Cauchy--Schwarz inequality
	\begin{align*}
		h(2n+1; x, x) 
		&= 
		\int_M h(n+1; x, z) h(n; z, x) \, \mu({\rm d} z) \\
		&\leq
		\bigg(\int_M h(n+1; x, z)^2 \, \mu({\rm d} z)\bigg)^{\frac{1}{2}}
		\bigg(\int_M h(n; x, z)^2 \, \mu({\rm d} z) \bigg)^{\frac{1}{2}} \\
		&=
		\sqrt{h(2n+2; x, x)} \sqrt{h(2n; x, x)} \\
		&\leq
		h(2n; x, x)
	\end{align*}
	where in the last estimate we have used monotonicity of $n \mapsto h(2n; x, x)$. Hence, \eqref{eq:A:50}, \eqref{eq:A:44}
	and \eqref{eq:A:45}, leads to
	\begin{align*}
		h(2n+1; x, x) 
		&\leq \frac{C}{V(x, \psi^{-1}(2n+1))} \frac{V(x, \psi^{-1}(2n+1))}{V(x, \psi^{-1}(2n))} \\
		&\leq \frac{CC_V}{V(x, \psi^{-1}(2n+1))} \bigg(\frac{\psi^{-1}(2n+1)}{\psi^{-1}(2n)}\bigg)^{\gamma_2} \\
		&\leq \frac{CC_V}{V(x, \psi^{-1}(2n+1))} \bigg(C_\psi \Big(1+\tfrac{1}{2n}\Big)\bigg)^{\gamma_2/\beta_1}.
	\end{align*}
	Thus there is $C_1 > 0$ such that for all $x \in M$ and $n \in \NN$,
	\begin{equation}
		\label{eq:A:51}
		h(n; x, x) \leq \frac{C_1}{V(x, \psi^{-1}(n))}.
	\end{equation}
	Next, we observe that
	\begin{align*}
		h(n; x, y) 
		&= \int_M h\Big(\big\lfloor \tfrac{n}{2} \big\rfloor ; x, z\Big)
		h\Big(n - \big\lfloor \tfrac{n}{2} \big\rfloor; z, y\Big) \, \mu({\rm d} z)  \\
		&\leq 
		\bigg(\int_M h\Big(\big\lfloor \tfrac{n}{2} \big\rfloor; x, z\Big)^2 \, \mu({\rm d} z)\bigg)^{\frac{1}{2}}
		\bigg(\int_M h\Big(n-\big\lfloor \tfrac{n}{2} \big\rfloor; y, z\Big)^2 \, \mu({\rm d} z)\bigg)^{\frac{1}{2}} \\
		&\leq
		C_1 \frac{1}{\sqrt{V(x, \psi^{-1}(n))}} \frac{1}{\sqrt{V(y, \psi^{-1}(n-1))}}
	\end{align*}
	where in the last inequality we have used \eqref{eq:A:51} and
	\[
		n-1 \leq 2 \big\lfloor \tfrac{n}{2}  \big\rfloor \leq n.
	\]
	Now, if $d(x, y) \leq \psi^{-1}(n)$, then
	\[
		V(x, \psi^{-1}(n)) 
		\leq 
		V(y, 2 \psi^{-1}(n)).
	\]
	Thus by \eqref{eq:A:44} and \eqref{eq:A:45},
	\begin{align*}
		\frac{1}{V(y, \psi^{-1}(n-1))} 
		&\leq \frac{C_2}{V(x, \psi^{-1}(n))}
	\end{align*}
	and the theorem follows.
\end{proof}

\begin{theorem}
	\label{thm:A:5}
	There is $C > 0$ such that for all
	\[
		h(n; x, y) \leq C \min\bigg\{\frac{1}{V(x, \psi^{-1}(n))}, \frac{n}{V(x, d(x, y)) \psi(d(x, y))}\bigg\}.
	\]
\end{theorem}
\begin{proof}
	First, by Lemma \ref{lem:A:6} we have
	\begin{align*}
		\int_{B(y, \psi^{-1}(n))} h(n; x, z)\, \mu({\rm d} z) 
		&=
		\PP^x\big(X_{n} \in B(y, \psi^{-1}(n)) \big) \\
		&\leq
		\PP^x\big(T_{\mathbf{X}}(y, \psi^{-1}(n)) \leq n \big) \\
		&\leq
		C \frac{V(y, \psi^{-1}(n))}{V(x, d(x, y))} \cdot \frac{n}{\psi(d(x, y))}.
	\end{align*}
	Hence, for any $k \in \NN$,
	\[
		\int_{B(y, \psi^{-1}(n))} h(kn; x, z) \mu({\rm d} z)
		\leq
		C \frac{V(y, \psi^{-1}(k n))}{V(x, d(x, y))} \cdot \frac{k n}{\psi(d(x, y))},
	\]
	and so by \eqref{eq:A:44} and \eqref{eq:A:45},
	\begin{align}
		\nonumber
		\min_{z \in B(y, \psi^{-1}(n))} h(k n; x, z) 
		&\leq C
		\frac{V(y, \psi^{-1}(kn))}{V(y, \psi^{-1}(n))}
		\cdot
		\frac{k n}{V(x, d(x, y)) \psi(d(x, y))} \\
		\nonumber
		&\leq C C_V
		\bigg(\frac{\psi^{-1}(kn)}{\psi^{-1}(n)}\bigg)^{\gamma_2}
		\frac{k n}{V(x, d(x, y)) \psi(d(x, y))} \\
		\label{eq:A:61}
		&\leq C C_V C_\psi k^{\frac{\beta_1+\gamma_2}{\beta_1}} \frac{n}{V(x, d(x, y)) \psi(d(x, y))}.
	\end{align}
	Next, we set
	\[
		R = \psi^{-1} ( \delta^{-1} n ),
		\qquad\text{and}\qquad
		k = \big\lfloor C_\psi b^{\beta_2} \big\rfloor + 1.
	\]
	Hence, by \eqref{eq:A:46},
	\[
		\delta \psi(b R) 
		= \frac{\psi(b R)}{\psi(R)} n \leq k n.
	\]
	By the standard argument, see for example \cite[Lemma 4.2]{Murugan2015}, the function
	\[
		u(j, z) = h(kn-j; x, z)
	\]
	is parabolic on $([0, \delta \psi(b R) ] \cap \NN_0) \times M$. Since $\delta \leq C_\psi^{-1} B^{-\beta_2}$, 
	\begin{align}
		\nonumber
		R \geq \frac{\psi^{-1}(\delta^{-1} n)}{\psi^{-1}(n)} \psi^{-1}(n) 
		&\geq \big(C_\psi^{-1} \delta^{-1}\big)^{\frac{1}{\beta_2}} \psi^{-1}(n) \\
		\label{eq:A:64}
		&\geq B \psi^{-1}(n).
	\end{align}
	Therefore, by \eqref{eq:A:61},
	\begin{align*}
		\min_{z \in B(y, R/B)} u(0, z) 
		&\leq \min_{z \in B(y, \psi^{-1}(n))} h(kn; x, z) \\ 
		&\leq C_2 \frac{n}{V(x, d(x, y)) \psi(d(x, y))}.
	\end{align*}
	By Theorem \ref{thm:A:3}, we obtain
	\begin{align}
		\nonumber
		h((k-1) n; x, y) 
		= u(n, y) &\leq K_0 \min_{z \in B(y, R/B)} u(0, z) \\
		\label{eq:A:62}
		&\leq K_0 C_2 \frac{n}{V(x, d(x, y)) \psi(d(x, y))}.
	\end{align}
	Since by \eqref{jump}
	\begin{align}
		\nonumber
		h(n; x, y) 
		&= 
		\int_M h(n-1; x, z) J(z, y) \, \mu({\rm d} z) \\
		\nonumber
		&\geq
		h(n-1; x, y) J(y, y) V(y, 0) \\
		\label{eq:A:63}
		&\geq
		C_3 h(n-1; x, y), 
	\end{align}
	from \eqref{eq:A:62} we easily deduce the theorem.
\end{proof}

\begin{theorem}
	\label{thm:A:4}
	For each $\eta \geq 1$, there is $c > 0$ such that for all $x, y \in M$ and $n \in \NN$, if 
	$d(x, y) \leq \eta \psi^{-1}(n)$, then
	\[
		h(n; x, y) \geq \frac{c}{V(x, \psi^{-1}(n))}.
	\]
\end{theorem}
\begin{proof}
	We start by showing that there is $C > 0$ such that for all $n \in \NN$ and $x \in M$,
	\begin{equation}
		\label{eq:A:52}
		h(n; x, x) \geq \frac{C}{V(x, \psi^{-1}(n))}.
	\end{equation}
	Let
	\[
		\xi = \max\left\{1, 4 c\right\}
	\]
	where $c$ is the constant from Lemma \ref{lem:A:4}. Then by Lemma \ref{lem:A:4},
	\begin{align*}
		1 - \int_{B(x, \psi^{-1}(\xi n))} h(2 n; x, y) \, \mu({\rm d} y)
		&=
		\PP^x\big(X_{2n} \notin B(x, \psi^{-1}(\xi n)) \big) \\
		&\leq
		\PP^x\big(\tau_{\mathbf{X}}(x, \psi^{-1}(\xi n)) \leq 2n \big)\\
		&\leq
		\frac{2 c n}{\psi(\psi^{-1}(\xi n))} + \frac{1}{32 n} \leq \frac{3}{4},
	\end{align*}
	thus by the Cauchy--Schwarz inequality
	\begin{align*}
		h(2n; x, x) 
		&\geq \int_{B(x, \psi^{-1}(\xi n))} h(n; x, y)^2 \, \mu({\rm d} y) \\
		&\geq \frac{1}{V(x, \psi^{-1}(\xi n))} \bigg( \int_{B(x, \psi^{-1}(\xi n))} h(n; x, y) \, \mu({\rm d} y)\bigg)^2 \\
		&\geq \frac{C_1}{V(x, \psi^{-1}(2n))} 
	\end{align*}
	where in the last estimate we have also used \eqref{eq:A:45}. Next, by \eqref{eq:A:63}, we have
	\begin{align*}
		h(2n+1; x, x) 
		&\geq
		C_3 h(2n; x, x) \\
		&\geq
		\frac{C_4}{V(x, \psi^{-1}(2n+1))}
	\end{align*}
	which completes the proof of \eqref{eq:A:52}. 

	To extend \eqref{eq:A:52} near the diagonal we use the parabolic Harnack inequality, that is Theorem \ref{thm:A:3}. Let
	\[
		R = \psi^{-1}(\delta^{-1} n), \qquad\text{and}\qquad
		k = \lfloor C_\psi b^{\beta_2} \rfloor + 1.
	\]
	Then the function 
	\[
		u(j, z) = h(kn - j; x, z), \qquad 0 \leq j \leq kn,\, \quad z \in M,
	\]
	is parabolic on $([0, \delta \psi(b R)] \cap \NN_0) \times M$. Hence, by Theorem \ref{thm:A:3},
	\begin{align*}
		h((k-1)n; x, x) = u(n, x) 
		&\leq K_0 \min_{z \in B(y, R/B)} u(0, z) \\
		&\leq K_0 \min_{z \in B(y, R/B)} h(kn; x, z).
	\end{align*}
	Given $\eta \geq 1$ we set
	\[
		\ell = \big\lfloor C_\psi \eta^{\beta_2} \big\rfloor.
	\]
	Then by \eqref{eq:A:45}, $\psi^{-1} (\ell n) \geq \eta \psi^{-1}(n)$, and hence, by \eqref{eq:A:52} and \eqref{eq:A:64}, 
	\begin{align*}
		\min_{z \in B(y, \eta \psi^{-1}(n))} h(k \ell n; x, z) 
		&\geq
		\min_{z \in B(y, \psi^{-1}(\ell n))} h(k \ell n; x, z)\\ 
		&\geq
		\min_{z \in B(y, R/B)} h(k\ell n; x, z) \\
		&\geq \frac{C_5}{V(x, \psi^{-1}(k\ell n))}.
	\end{align*}
	Now, by \eqref{eq:A:63}, \eqref{eq:A:42} and \eqref{eq:A:46} we easily conclude the proof.
\end{proof}

\begin{theorem}
	\label{thm:A:6}
	There is $C > 0$ such that for all $x, y \in M$ and $n \in \NN$,
	\[
		h(n; x, y) \geq C \min\bigg\{\frac{1}{V(x, \psi^{-1}(n))}, \frac{n}{V(x, d(x, y))\psi(d(x, y))} \bigg\}.
	\]
\end{theorem}
\begin{proof}
	We are going to use the parabolic Harnack inequality. We set
	\[
		R = \psi^{-1}(\delta^{-1} n), \qquad
		k = \lfloor C_\psi b^{\beta_2} \rfloor + 1.
	\]
	Then for $\eta$ determined in Lemma \ref{lem:A:10}, by \eqref{eq:A:45} and \eqref{eq:A:67}, we get
	\begin{align*}
		\eta \psi^{-1}(n) 
		&\leq R \eta\frac{\psi^{-1}(n)}{\psi^{-1}(\delta^{-1} n)} \\
		&\leq \eta B \big(C_\psi \delta \big)^{\frac{1}{\beta_1}} R/B \leq R/B
	\end{align*}
	where $\eta \geq 1$ is determined in Lemma \ref{lem:A:10}. Then the function 
	\[
		u(j, z) = h(kn - j; x, z), \qquad 0 \leq j \leq kn, \, \quad z \in M,
	\]
	is parabolic on $([0, \delta \psi(b R)] \cap \NN_0) \times M$. Hence, by Lemma \ref{lem:A:10},
	\begin{align*}
		\int_{B(y, \eta \psi^{-1}(n))} h(n; x, w) \, \mu({\rm d} w)
		&= 
		\PP^x\left(X_{n} \in B(y, \eta \psi^{-1}(n)) \right) \\
		&\geq
		C \frac{V(y, \psi^{-1}(n))}{V(x, d(x, y))} \cdot \frac{n}{\psi(d(x, y))}.
	\end{align*}
	Hence,
	\[	
		\max_{w \in B(y, R/B)} h(n; x, w) \geq 
		C_1 \frac{n}{V(x, d(x, y))\psi(d(x, y))}.
	\]
	which together with Theorem \ref{thm:A:3} implies that
	\[
		h(n; x, y) \geq C_2 \frac{n}{V(x, d(x, y))\psi(d(x, y))}.
	\]
	Lastly, in view of Theorem \ref{thm:A:4}, if $d(x, y) \leq \eta \psi^{-1}(n)$ then
	\[
		h(n; x, y) \geq \frac{C_4}{V(x, \psi^{-1}(n))}
	\]
	which completes the proof.
\end{proof}

By Theorems \ref{thm:A:5} and \ref{thm:A:6} we obtain the following statement.
\begin{theorem}
	\label{thm:A:7}
	Suppose that a discrete metric measure space $(M, d, \mu)$ satisfies \eqref{eq:A:44}--\eqref{eq:A:23}. Let $\psi: (0, \infty) 
	\rightarrow (0, \infty)$ be increasing function satisfying \eqref{eq:A:46} and \eqref{eq:A:45}. Then the jump
	function $J$ satisfies \eqref{jump} and the cut-off Sobolev inequality \eqref{cutoff} holds true if and only if
	\begin{equation}
		\label{eq:40}
		h(n; x, y) \approx
		\min\bigg\{\frac{1}{V(x, \psi^{-1}(n))}, \frac{n}{V(x, d(x, y)) \psi(d(x, y))}\bigg\}
	\end{equation}
	uniformly with respect to $x, y \in M$ and $n \in \NN$.
\end{theorem}

\section{Heat kernel estimates: discrete time}
In this section we apply the results about subordinate processes (Section \ref{sec:2}) together with the stability of heat kernels
estimates (Section \ref{sec:3}) on discrete measure metric space and for discrete time setting. For the counterpart in the case of
continuous time processes we refer \cite{bae2019}. Recall that $h$ is a transition density of the Markov chain defined by the jumping kernel $J$. By $G$ we denote the Green function of this process.
\begin{lemma}
	\label{lem:tg1}
	Assume that there exist $C_1 \geq 1$ and $\gamma_2 \geq \gamma_1 > 0$ such that for all $x\in M$, $\lambda > 1$ and
	$r > 0$,
	\[
		C_1^{-1} \lambda^{\gamma_1} \leq \frac{V(x,\lambda r)}{V(x,r)}\leq C_1 \lambda^{\gamma_2}.
	\]
	Suppose that there exists $C_2 \geq 1$ such that for all $x,y\in M$ and $s > 0$,
	\[
		C_2^{-1}
		\frac{1}{V(x,f^{-1}(s))} 
		\ind{\{f(d(x,y))\leq s\}} 
		\leq 
		p(s; x,y)
		\leq
		C_2
		\frac{1}{V(x,f^{-1}(s))} e^{-g\left(\frac{f(d(x,y))}{s}\right)}
	\]
	where $f$ is positive increasing function belonging to $\WLSCINF{\delta_1}{c_3}{0} \cap \WUSCINF{\delta_2}{C_4}{0}$
	for certain $c_3 \in (0, 1]$, $C_4 \in [1, \infty)$, and $\delta_2 \geq \delta_1 > 0$, and $g$ is positive increasing
	function belonging to $\WLSCINF{\alpha}{c_5}{1}$ for some $c_5 \in (0, 1]$ and $\alpha > 0$. Let $\vphi$ be a Bernstein
	function such that $\vphi(0) = 0$ and $\vphi(1) = 1$. Assume that the L\'{e}vy measure of $\vphi$ has non-increasing density
	which belongs to $\WLSCINF{\eta_1}{c_6}{1}$ and $\WUSCINF{\eta_2}{C_7}{1}$ for some $\eta_1 \leq \eta_2\leq-1$,
	$c_6 \in (0, 1]$ and $C_7 \in [1, \infty)$. Then
	\[
		p_{\varphi}(1; x,y) \approx \min\left\{\frac{1}{V(x,f^{-1}(1))},\frac{f(d(x, y))}{V(x, d(x, y))} \nu\big(f(d(x, y))\big)\right\}
	\]
	uniformly with respect to $x, y \in M$.
\end{lemma}
\begin{proof}
%	Since $\nu$ is integrable on $(1,\infty)$, we conclude that $\eta_1 \leq -1 $. And without lost of generality we may assume that $\eta_2\geq -1$.
	By Proposition \ref{prop:1}, we have, for $x,y\in M$ such that $f(d(x,y))\geq 1$,
	\begin{align*}
		p_{\varphi}(1; x, y)
		&\geq 
		C_2^{-1}c_1
		\sum_{k \geq f(d(x,y))} 
		\frac{1}{V(x,f^{-1}(k))} \nu(k) \\
		&= 
		C_2^{-1}c_1 \frac{1}{V(x,d(x,y))} \nu(f(d(x,y))) 
		\sum_{k \geq f(d(x,y))}
		\frac{V(x,f^{-1}(f(d(x,y)))}{V(x,f^{-1}(k))} \cdot \frac{\nu(k)}{\nu(f(d(x,y)))}\\
		&\geq
		c_8 \frac{1}{V(x,d(x,y))}\nu(f(d(x,y)))
		\sum_{k\geq f(d(x,y))} 
		\left(\frac{k}{f(d(x,y))}\right)^{\eta_1-\gamma_2/\delta_1}\\
		&\geq
		c_9
		\frac{f(d(x,y))}{V(x,d(x,y))}\nu(f(d(x,y))).
	\end{align*}
	Moreover, by the lower scaling property of $g$,  there is $C_4 > 0$, such that for $\tau = \gamma_2/\delta_1-\eta_1 +1$, and
	$u > 1$,
	\[
		e^{-g(u)} 
		\leq
		C_4 u^{-\tau}.
	\]
	Therefore, for $x,y\in M$ such that $f(d(x,y))\geq 1$,
		\begin{align*}
		p_{\varphi}(1; x,y)
		&\leq 
		C_2 c_2 
		\sum_{k\geq f(d(x,y))}
		\frac{1}{V(x,f^{-1}(k))}\nu(k)
		+
		C_2c_2\sum_{k< f(d(x,y))}\frac{1}{V(x,f^{-1}(k))}e^{-g(f(d(x,y))/k)}\nu(k) \\
		&=
		C_2c_2\frac{1}{V(x,d(x,y))}\nu(f(d(x,y)))\bigg(
		\sum_{k\geq f(d(x,y))}\frac{V(x,f^{-1}(f(d(x,y)))}{V(x,f^{-1}(k))} \frac{\nu(k)}{\nu(f(d(x,y)))} \\
		&\phantom{\leq}+
		\sum_{k < f(d(x,y))}\frac{V(x,f^{-1}(f(d(x,y)))}{V(x,f^{-1}(k))} \frac{\nu(k)}{\nu(f(d(x,y)))} 
		e^{-g(f(d(x,y))/k)}
		\bigg) \\
		&\leq C_5  \frac{1}{V(x,d(x,y))}\nu(f(d(x,y))) 
		\bigg(\sum_{k\geq f(d(x,y))}
		\left(\frac{k}{f(d(x,y))}\right)^{\eta_2-\gamma_1/\delta_2} \\
		&\phantom{\leq}+
		\sum_{k < f(d(x, y))}
		\left(\frac{k}{f(d(x,y))}\right)^{\eta_1-\gamma_2/\delta_1+\tau} 
		\bigg) \\
		&\leq C_6 \frac{f(d(x,y))}{V(x,d(x,y))}\nu(f(d(x,y)))
	\end{align*}
	and the proposition follows in this case. If $f(d(x,y))< 1$
\[ C_2^{-1} \frac{1}{V(x,f^{-1}(1))} c(\varphi,1)	\leq p_{\varphi}(1; x,y)\leq C_2 \frac{1}{V(x,f^{-1}(1))}\]
which ends the proof.
\end{proof}

\begin{theorem}
	\label{thm:5}
	Let $(M, d, \mu)$ be a discrete measure metric space satisfying \eqref{eq:A:44}--\eqref{eq:A:23} for $r_0 = \frac{1}{2}$. 
	Assume that $\mathbf{S} = (S_n : n \in \NN_0)$ is a Markov chain on $M$ with the transition function
	$p: \NN_0 \times M \times M \rightarrow [0, \infty)$ such that for all $x, y \in M$ and $n \in \NN$,
	\begin{equation}
		\label{eq:62}
		C_2^{-1}
		\frac{1}{V(x,f^{-1}(n))} 
		\ind{\{f(d(x,y))\leq n\}} 
		\leq 
		p(n; x,y)
		\leq
		C_2
		\frac{1}{V(x,f^{-1}(n))} e^{-g\left(\frac{f(d(x,y))}{n}\right)}
	\end{equation}
	for certain $C_2 \geq 1$, where $f$ is positive increasing function which belongs to $\WLSCINF{\delta_1}{c_3}{0}$
	and $\WUSCINF{\delta_2}{C_4}{0}$ for certain $c_3 \in (0, 1]$, $C_4 \in [1, \infty)$, and $\delta_2 \geq \delta_1 > 0$,
	and $g$ is positive increasing function belonging to $\WLSCINF{\alpha}{c_5}{1}$ for some $c_5 \in (0, 1]$ and
	$\alpha > 0$. 
	Let $J: M \times M \rightarrow [0, \infty)$ be a jump kernel satisfying \eqref{jump} for some increasing function 
	$\psi: [0, \infty) \rightarrow (0, \infty)$ such that $\psi(0) = \frac{1}{2}$, $\psi(1) = 1$ satisfying \eqref{eq:A:46}
	and \eqref{eq:A:45} for some $ \beta_2 \geq \beta_1 > 0$. Then $\psi\circ f^{-1}\in \WLSCINF{\eta}{c}{1}$ for some $\eta<1$,
	if and only if
	\begin{equation}
		\label{eq:37}
		h(n; x, y) \approx  \min\left\{\frac{1}{V(x, \psi^{-1}(n))}, \frac{n}{V(x, d(x, y)) \psi(d(x, y))} \right\}
	\end{equation}
	uniformly with respect to $n \in \NN_0$ and $x, y \in M$. Furthermore, if $\psi\circ f^{-1}\in \WLSCINF{\eta}{c}{1}$ 
	for some $\eta<1$ and $\gamma_1 > \eta \delta_2$, then
	\begin{equation}
		\label{eq:tg8}
		G(x,y)\approx \frac{\psi(d(x,y))}{V(x,d(x,y))}, 
		\qquad x,\,y\in M.
	\end{equation}
\end{theorem}
\begin{proof}
	Our first aim is to construct Bernstein function $\vphi$ such that 
	\begin{equation}
		\label{eq:36}
		\vphi(u) \approx \frac{1}{\psi\big(f^{-1}(1/u)\big)}
	\end{equation}
	uniformly with respect to $u \in (0, 1]$. To do so, we define a measure on $[0, \infty)$ by the formula
	\begin{equation}
		\label{eq:38}
		\nu({\rm d} s) = \frac{1}{s \psi(f^{-1}(s))} {\rm d} s. 
	\end{equation}
	Observe that $\psi \circ f^{-1}$ is increasing and belongs to $\WLSCINF{\beta_1 \delta_2^{-1}}{c_7}{1}$ %\cap
	%\WUSCINF{\beta_2 \delta_1^{-1}}{C_8}{1}$
	for certain $c_7 \in (0, 1]$. %, $C_8 \in [1, \infty)$. 
	This and $\psi(f^{-1}(0))=1/2$ imply that the formula \eqref{eq:38} defines a L\'evy measure. Let 
	\[
		\vphi(u) = \int_{[0, \infty)} \big(1 - e^{-s u}\big) \, \nu({\rm d} s).
	\]
	Then \eqref{eq:36} is an easy consequence of $1-e^{-s}\approx \min\{1,s\}$ and scaling properties of $\psi \circ f^{-1}$.

	Now, by Theorem \ref{thm:1}, we obtain
	\begin{align}
		\label{eq:tg7}
		p_\vphi(n; x, y) 
		&\approx \min \left\{
		\frac{1}{V\Big(x, f^{-1}\Big(\frac{1}{\vphi^{-1}(1/n)}\Big)\Big)},
		n \frac{\vphi\Big(\frac{1}{f(d(x, y))}\Big)}{V(x, d(x, y))} \right\}.
	\end{align}
	In view of Theorem \ref{thm:A:7}, the jump function
	\[
		J_\vphi(x, y) = p_\vphi(1; x, y)
	\]
	gives rise to the quadratic form $\Gamma_\vphi$ which satisfy the cuf-off Sobolev inequality. By 
	\eqref{eq:A:44}--\eqref{eq:A:23} and \eqref{eq:36}, we have
	\[
		J_\vphi(x, y) \approx J(x, y)
	\]
	uniformly with respect to $x, y \in M$. Therefore, the quadratic form $\Gamma$ determined by $J$ satisfies the cut-off
	Sobolev inequality. Consequently, by Theorem \ref{thm:A:7}, we conclude \eqref{eq:37}.

	To prove the reverse implication, let us assume that \eqref{eq:37} holds true. We define $\nu$ by the formula \eqref{eq:38}.
	Then the density of $\nu$ is non-increasing and satisfies weak scaling conditions at infinity with exponents
	$-1-\beta_2/\delta_1$ and $-1-\beta_1/\delta_2$. Hence, by Lemma \ref{lem:tg1}
	\[
		p_\varphi(1;x,y) \approx \min\left\{\frac{1}{V(x,1)},\frac{1}{V(x,d(x,y))\psi(d(x,y))}\right\}.
	\]
	Next, Theorem \ref{thm:A:7} implies that
	\begin{equation}
		\label{eq:30}
		p_\varphi(n;x,y)\approx \min\left\{\frac{1}{V(x,\psi^{-1}(n))},\frac{n}{V(x,d(x,y))\psi(d(x,y))}\right\}.
	\end{equation}
	In view of Proposition \ref{prop:101}, the estimates \eqref{eq:30} and \eqref{eq:62} gives
	\[
		\frac{1}{V(x,\psi^{-1}(n))} \lesssim \frac{1}{V(x,f^{-1}(1/\varphi^{-1}(1/n)))}.
	\]
	Consequently, we get
	\[
		\varphi(u) \lesssim \frac{1}{\psi(f^{-1}(1/u))}, \qquad u < 1.
	\]
	Moreover, by Corollary \ref{cor:11}, the estimates \eqref{eq:30} and \eqref{eq:62} leads to
	\[
		\frac{c}{\psi(r)} \lesssim \varphi(1/f(r)), \qquad r>1.
	\]
	which gives
	\[
		\varphi(u) \gtrsim c \frac{1}{\psi(f^{-1}(1/u))},\qquad u<1.
	\]
	Therefore,
	\[
		p_\varphi(1;x,y)
		\approx 
		\min\left\{\frac{1}{V(x,1)}, \frac{\varphi(1/f(d(x,y))}{V(x,d(x,y))}\right\}.
	\]
	By Theorem \ref{thm:6} there is $\eta < 1$ such that $\psi\circ f^{-1}$ satisfies the upper scaling condition with
	$\eta$, which finishes the proof. 

	Lastly, let us observe that \eqref{eq:tg7} together with \eqref{eq:36} implies that
	\[
		h(n;x,y) \approx p_\varphi(n;x,y), 
	\]
	for all $n \in \NN_0$ and $x, y \in M$. Therefore \eqref{eq:tg8} is a consequence of Theorem \ref{thm:4} and Remark
	\ref{rem:tg1}.
\end{proof}

It is remarkable that similar results holds true also for the continuous time counterpart. Let $h$ and $G$ be the heat kernel
and the Green function corresponding to Markov process associated to a jumping kernel $J$. The implication
that the scaling condition on $\psi\circ f^{-1}$ gives the heat kernel estimates is proved in \cite[Theorem 2.19]{bae2019}.
The reverse implication follows from the next theorem.
\begin{theorem}
	\label{thm:tg5}
	Let $(M, d, \mu)$ be a  measure metric space satisfying \eqref{eq:A:44}--\eqref{eq:A:42} for $r_0 = 0$. 
	Assume that $\mathbf{S} = (S_t : t \in [0,\infty))$ is a Markov process on $M$ with the transition function
	$p: [0,\infty) \times M \times M \rightarrow [0, \infty)$ such that for all $x, y \in M$ and $t>0$,
	\[
		C_2^{-1}
		\frac{1}{V(x,f^{-1}(t))} 
		\ind{\{f(d(x,y))\leq t\}} 
		\leq 
		p(t; x,y)
		\leq
		C_2
		\frac{1}{V(x,f^{-1}(t))} e^{-g\left(\frac{f(d(x,y))}{t}\right)}
	\]
	for certain $C_2 \geq 1$, where $f$ is positive increasing function which belongs to $\WLSCINF{\delta_1}{c_3}{0}$
	and $\WUSCINF{\delta_2}{C_4}{0}$ for certain $c_3 \in (0, 1]$, $C_4 \in [1, \infty)$, and $\delta_2 \geq \delta_1 > 0$,
	and $g$ is positive increasing function belonging to $\WLSCINF{\alpha}{c_5}{1}$ for some $c_5 \in (0, 1]$ and
	$\alpha > 0$. 
	Let $J: M \times M \rightarrow [0, \infty)$ be a jump kernel satisfying \eqref{jump} for some increasing function 
	$\psi: [0, \infty) \rightarrow (0, \infty)$ such that  $\psi(1) = 1$ satisfying \eqref{eq:A:46}
	and \eqref{eq:A:45} for all $R\geq r>0$, for some $ \beta_2 \geq \beta_1 > 0$ and
	\[
		\int^1_0 \frac{{\rm d} s}{\psi(f^{-1}(s))} <\infty.
	\]
	Furthermore, $\psi\circ f^{-1}\in \WLSCINF{\eta}{c}{0}$ for some $\eta<1$, if and only if
	\[
		h(t; x, y) \approx  \min\left\{\frac{1}{V(x, \psi^{-1}(t))}, \frac{t}{V(x, d(x, y)) \psi(d(x, y))} \right\}
	\]
	uniformly with respect to $t>0$ and $x, y \in M$. If $\gamma_1 > \eta\delta_2$, then
	\[
		G(x,y)\approx \frac{\psi(d(x,y))}{V(x,d(x,y))}, 
		\qquad x,\,y\in M.
	\]
\end{theorem}
\begin{proof}
	The proof is the same as above, except that we use \cite[Theorem 1.13]{Chen2016} instead of Theorem \ref{thm:A:7},
	\cite[Lemma 4.2]{bae2019} instead of Lemma \ref{lem:tg1} and Theorem \ref{thm:7} instead of Theorem \ref{thm:6}.
\end{proof}

% remark dla f(x) = x^\alpha

Similarly, based on Theorem \ref{thm:2} one can prove the following statement.
\begin{theorem}
	\label{thm:3}
	Let $(M, d, \mu)$ be a discrete measure metric space satisfying \eqref{eq:A:44}--\eqref{eq:A:23} for $r_0 = \frac{1}{2}$. 
	Assume that $\mathbf{S} = (S_n : n \in \NN_0)$ is a Markov chain on $M$ with the transition function
	$p: \NN_0 \times M \times M \rightarrow [0, \infty)$ such that for all $x, y \in M$ and $n \in \NN$,
	\[
		C_2^{-1}
		\frac{1}{V(x, n^{1/\alpha})} 
		\ind{\{d(x,y)^\alpha \leq n\}} 
		\leq 
		p(n; x,y)
		\leq
		C_2
		\min\bigg\{
		\frac{1}{V(x, n^{1/\alpha})},
		\frac{n}{V(x, d(x,y)) (d(x, y))^\alpha} 
		\bigg\}
	\]
	for certain $C_2 \geq 1$ and $\alpha > 0$. Let $J: M \times M \rightarrow [0, \infty)$ be a jump kernel satisfying
	\eqref{jump} for some increasing function $\psi: [0, \infty) \rightarrow (0, \infty)$ such that $\psi(0) = \frac{1}{2}$,
	$\psi(1) = 1$ satisfying \eqref{eq:A:46} and \eqref{eq:A:45} for some $\alpha > \beta_2 \geq \beta_1 > 0$. Then
	\[
		h(n; x, y) \approx  \min\left\{\frac{1}{V(x, \psi^{-1}(n))}, \frac{n}{V(x, d(x, y)) \psi(d(x, y))} \right\}
	\]
	uniformly with respect to $n \in \NN$ and $x, y \in M$.
	Furthermore, if $\gamma_1>\beta_2$
	\[
		G(x,y)\approx \frac{\psi(d(x,y))}{V(x,d(x,y))}, 
		\qquad x,\,y\in M.
	\]
\end{theorem}

\section{Applications}
\label{sec:5}
In this section we provide examples of the metric measure space and a Markov chain $\mathbf{S}$ satisfying the assumptions
of Theorems \ref{thm:3} and \ref{thm:5}. 

The first example are infinite weighted graphs. In particular, we improve the results of Murugan--Saloff-Coste \cite{Murugan2015}.
\begin{theorem}
	\label{thm:8}
	Let $(G, E)$ be an infinite weighted graph which is Ahlfors $\gamma$-regular. Let $J: G \times G \rightarrow [0, \infty)$ 
	be a jump kernel satisfying \eqref{jump} for some increasing function $\psi: [0, \infty) \rightarrow [0, \infty)$ such
	that $\psi(0) = \frac{1}{2}$ and $\psi(1) = 1$ satisfying \eqref{eq:A:46} and \eqref{eq:A:45} for some 
	$\beta_2 \geq \beta_1 > 0$. Then
	\[
		h(n; x, y)
		\approx
		\min\bigg\{
		\frac{1}{V(x, \psi^{-1}(n))}, \frac{n}{V(x, d(x, y)) \psi(d(x, y))}\bigg\}
	\]
	uniformly with respect to $n \in \NN$, $x, y \in G$.
\end{theorem}
\begin{proof}
	We are going to use Theorem \ref{thm:5}. To do so we need to construct a Markov chain satisfying \eqref{eq:62}.
	By \cite[Theorem 4.39]{Barlow2017}, the transition function of the simple random walk $\mathbf{S} = (S_n : n \in \NN_0)$
	satisfies 
	\begin{equation}
		\label{eq:60}
		p(n; x, y) \leq c_1 n^{-\alpha/\beta} \exp\bigg(-c_2 \Big(\frac{d(x, y)^\beta}{n} \Big)^{\frac{1}{\beta-1}} \bigg)
	\end{equation}
	and
	\begin{equation}
		\label{eq:61}
		p(n; x, y) + p(n+1; x, y) \geq 
		c_3 n^{-\alpha/\beta} 
		\exp\bigg(-c_4 \Big(\frac{d(x, y)^\beta}{n} \Big)^{\frac{1}{\beta-1}} \bigg) 
	\end{equation}
	for all $n \in \NN$, $x, y \in G$ and certain $c_1, c_2, c_3, c_4 > 0$. Let $\tilde{\mathbf{S}} = 
	(\tilde{S}_n : n \in \NN_0)$ be a Markov chain driven by the transition function
	\[
		\tilde{p}(x, y) = \frac{1}{2} p(1; x, y) + \frac{1}{2} p(2; x, y), \qquad x, y \in G.
	\]
	Observe that
	\[
		\tilde{p}(n; x, y) = 2^{-n} \sum_{j = 0}^n {n \choose j} p(n+j; x, y),
	\]
	thus by \eqref{eq:61}
	\begin{align*}
		\tilde{p}(n; x, y) &\geq 
		2^{-n-1} \sum_{j = 0}^{n-1}
		{n \choose j} p(n+j; x, y) + {n \choose j+1} p(n+j+1; x, y)  \\
		&\gtrsim
		2^{-n} \sum_{j = 0}^{n-1} {n \choose j} 
		(n+j)^{-\alpha/\beta} \exp\bigg(-c_4 \Big(\frac{d(x, y)^\beta}{n+j} \Big)^{\frac{1}{\beta-1}} \bigg) \\
		&\geq
		c_5
		n^{-\alpha/\beta}
		\exp\bigg(-c_4 \Big(\frac{d(x, y)^\beta}{n} \Big)^{\frac{1}{\beta-1}} \bigg).
	\end{align*}
	Moreover, by \eqref{eq:60}, we easily conclude that
	\[
		\tilde{p}(n; x, y) \leq
		c_6 n^{-\alpha/\beta} \exp\bigg(-c_7 \Big(\frac{d(x, y)^\beta}{n} \Big)^{\frac{1}{\beta-1}} \bigg).
	\]
	Now, to complete the proof it is enough to invoke Theorem \ref{thm:5}.
\end{proof}
Let us notice that Theorem \ref{thm:8} applied to $\ZZ^d$ improves results obtained by Cygan and {\v S}ebek \cite{Cygan2021},
not only by dropping the CBF assumption on $\vphi$ but also by covering the cases when the first step is comparable to
$J_\vphi$.

There are metric measure spaces which fail to possess a canonical diffusion process, however they have processes stable-like
estimates instead. See for example \cite[Example 2.24]{Bendikova2014}. In this cases we can apply Theorem \ref{thm:3}
to study Markov chains and the corresponding discrete time heat kernels.

\begin{bibliography}{dirichlet}
	\bibliographystyle{amsplain}

\providecommand{\bysame}{\leavevmode\hbox to3em{\hrulefill}\thinspace}
\providecommand{\MR}{\relax\ifhmode\unskip\space\fi MR }
\providecommand{\MRhref}[2]{%
  \href{http://www.ams.org/mathscinet-getitem?mr=#1}{#2}
}
\providecommand{\href}[2]{#2}
\begin{thebibliography}{10}

\bibitem{MR0167642}
M.~Abramowitz and I.A. Stegun, \emph{Handbook of mathematical functions: with
  formulas, graphs, and mathematical tables}, Dover Books on Mathematics, 1965.

\bibitem{Barlow2017}
M.T. Barlow, \emph{Random walks and heat kernels on graphs}, London
  Mathematical Society Lecture Note Series, vol. 438, Cambridge University
  Press, 2017.

\bibitem{BassLevin}
R.F. Bass and D.A. Levin, \emph{Transition probabilities for symmetric jump
  processes}, Trans. Amer. Math. Soc. \textbf{354} (2002), no.~7, 2933--2953.

\bibitem{bendikov_saloff}
A.~Bendikov and L.~Saloff-Coste, \emph{Random walks on groups and discrete
  subordination}, Math. Nachr. \textbf{285} (2012), no.~5--6, 580--605.

\bibitem{Bendikova2014}
A.D. Bendikova, A.A Grigor'yan, Ch. Pittet, and W.~Woess, \emph{Isotropic
  {M}arkov semigroups on ultra-metric spaces}, Uspekhi Mat. Nauk \textbf{69}
  (2014), 3--102.

\bibitem{berg_forst}
Ch. Berg and G.~Forst, \emph{Infinitely divisible probability measures and
  potential kernels}, Probability measures on groups ({P}roc. {F}ifth {C}onf.,
  {O}berwolfach, 1978), Lecture Notes in Math., vol. 706, 1979, pp.~22--35.

\bibitem{MR1406564}
J.~Bertoin, \emph{L\'evy processes}, Cambridge Tracts in Mathematics, vol. 121,
  Cambridge University Press, Cambridge, 1996.

\bibitem{bochner1}
S.~Bochner, \emph{Diffusion equation and stochastic processes}, Proc. Nat.
  Acad. Sci. U.S.A. \textbf{35} (1949), 368--370.

\bibitem{bochner2}
\bysame, \emph{Harmonic analysis and the theory of probability}, University of
  California Press, Berkeley and Los Angeles, 1955.

\bibitem{MR3165234}
K.~Bogdan, T.~Grzywny, and M.~Ryznar, \emph{Density and tails of unimodal
  convolution semigroups}, J. Funct. Anal. \textbf{266} (2014), no.~6,
  3543--3571.

\bibitem{MR2013738}
K.~Bogdan, A.~St\'{o}s, and P.~Sztonyk, \emph{Harnack inequality for stable
  processes on {$d$}-sets}, Studia Math. \textbf{158} (2003), no.~2, 163--198.

\bibitem{Boutayeb2015}
S.~Boutayeb, T.~Coulhon, and A.~Sikora, \emph{A new approach to pointwise heat
  kernel upper bounds on doubling metric measure spaces}, Adv. Math.
  \textbf{270} (2015), 302--374.

\bibitem{carasso}
A.S. Carasso and T.~Kato, \emph{On subordinated holomorphic semigroups}, Trans.
  Amer. Math. Soc. \textbf{327} (1991), no.~2, 867--878.

\bibitem{MR2443765}
Z.-Q. Chen, P.~Kim, and T.~Kumagai, \emph{Weighted {P}oincar\'{e} inequality
  and heat kernel estimates for finite range jump processes}, Math. Ann.
  \textbf{342} (2008), no.~4, 833--883.

\bibitem{MR4249776}
Z.-Q. Chen, P.~Kim, T.~Kumagai, and J.~Wang, \emph{Heat kernel upper bounds for
  symmetric {M}arkov semigroups}, J. Funct. Anal. \textbf{281} (2021), no.~4,
  Paper No. 109074, 40.

\bibitem{MR2357678}
Z.-Q. Chen and T.~Kumagai, \emph{Heat kernel estimates for jump processes of
  mixed types on metric measure spaces}, Probab. Theory Related Fields
  \textbf{140} (2008), no.~1-2, 277--317.

\bibitem{MR4157572}
Z.-Q. Chen, T.~Kumagai, and J.~Wang, \emph{Heat kernel estimates and parabolic
  {H}arnack inequalities for symmetric {D}irichlet forms}, Adv. Math.
  \textbf{374} (2020), 107269, 71.

\bibitem{Chen2016}
\bysame, \emph{Stability of heat kernel estimates for symmetric non-local
  {D}irichlet forms}, Mem. Amer. Math. Soc. \textbf{271} (2021), no.~1330,
  v+89.

\bibitem{cont}
R.~Cont and P.~Tankov, \emph{Financial modelling with jump processes}, Chapman
  \& Hall/CRC Financial Mathematics Series, Chapman \& Hall/CRC, Boca Raton,
  FL, 2004.

\bibitem{MR2039952}
T.~Coulhon, \emph{Heat kernel and isoperimetry on non-compact {R}iemannian
  manifolds}, Heat kernels and analysis on manifolds, graphs, and metric spaces
  ({P}aris, 2002), Contemp. Math., vol. 338, Amer. Math. Soc., Providence, RI,
  2003, pp.~65--99.

\bibitem{Cygan2021}
W.~Cygan and S.~{\v S}ebek, \emph{Transition probability estimates for
  subordinate random walks}, Math. Nachr. \textbf{294} (2021), 518--558.

\bibitem{MR1736868}
A.~Grigor'yan, \emph{Estimates of heat kernels on {R}iemannian manifolds},
  Spectral theory and geometry ({E}dinburgh, 1998), London Math. Soc. Lecture
  Note Ser., vol. 273, Cambridge Univ. Press, Cambridge, 1999, pp.~140--225.

\bibitem{MR2569498}
\bysame, \emph{Heat kernel and analysis on manifolds}, AMS/IP Studies in
  Advanced Mathematics, vol.~47, American Mathematical Society, Providence, RI;
  International Press, Boston, MA, 2009.

\bibitem{GrigoryanHuHu2018}
A.~Grigor'yan, E.~Hu, and J.~Hu, \emph{Two-sided estimates of heat kernels of
  jump type {D}irichlet forms}, Adv. Math. \textbf{330} (2018), 433--515.

\bibitem{GrigoryanHuLau2014}
A.~Grigor'yan, J.~Hu, and K.-S. Lau, \emph{Estimates of heat kernel for
  non-local regular {D}irichlet forms}, Trans. Amer. Math. Soc. \textbf{366}
  (2014), 6397--6441.

\bibitem{Grigoryan2014}
A.~Grigor’yan and J.~Hu, \emph{Upper bounds of heat kernels on doubling
  spaces}, Mosc. Math. J. \textbf{14} (2014), 505--563.

\bibitem{MR3729529}
T.~Grzywny and M.~Kwa\'{s}nicki, \emph{Potential kernels, probabilities of
  hitting a ball, harmonic functions and the boundary {H}arnack inequality for
  unimodal {L}\'{e}vy processes}, Stoch. Proc. Appl. \textbf{128} (2018),
  1--38.

\bibitem{TGLLBT}
T.~Grzywny, \L{}. Le\.{z}aj, and B.~Trojan, \emph{Transition densities of
  subordinators of positive order}, To appear in Journal of the Institute of
  Mathematics of Jussieu, 2021.

\bibitem{MR4039012}
T.~Grzywny, M.~Ryznar, and B.~Trojan, \emph{Asymptotic behaviour and estimates
  of slowly varying convolution semigroups}, Int. Math. Res. Not. (2019),
  no.~23, 7193--7258.

\bibitem{MR4140542}
T.~Grzywny and K.~Szczypkowski, \emph{L\'evy processes: {C}oncentration
  function and heat kernel bounds}, Bernoulli \textbf{26} (2020), 3191--3223.

\bibitem{bae2019}
B.~Joohak, K.~Jaehoon, K.~Panki, and Jaehun L., \emph{Heat kernel estimates and
  their stabilities for symmetric jump processes with general mixed polynomial
  growths on metric measure spaces}, arXiv: 1904.10189, 2019.

\bibitem{Kulik}
V.~Knopova, A.~Kulik, and Schilling R., \emph{Construction and heat kernel
  estimates of general stable-like markov processes}, arxiv: 2005.08491, 2020.

\bibitem{MR3965398}
V.P. Knopova, A.N. Kochubei, and A.M. Kulik, \emph{Parametrix methods for
  equations with fractional {L}aplacians}, Handbook of fractional calculus with
  applications. {V}ol. 2, De Gruyter, Berlin, 2019, pp.~267--297.

\bibitem{MR3098066}
M.~Kwaśnicki, J.~Małecki, and M.~Ryznar, \emph{Suprema of {L}\'evy
  processes}, Ann. Probab. \textbf{41} (2013), 2047--2065.

\bibitem{Mimica2019}
A.~Mimica and S.~{\v S}ebek, \emph{Harnack inequality for subordinate random
  walks}, J. Theor. Probab. \textbf{32} (2019), 737--764.

\bibitem{MuruganSaloff2015}
M.~Murugan and L.~Saloff-Coste, \emph{Anomalous threshold behavior of long
  range random walks}, Electron. J. Probab. \textbf{20} (2015), 1--21.

\bibitem{Murugan2015}
\bysame, \emph{Transition probability estimates for long range random walks},
  New York J. Math. \textbf{21} (2015), 723--757.

\bibitem{Murugan2019}
\bysame, \emph{Heat kernel estimates for anomalous heavy-tailed random walks},
  Ann. Inst. Henri Poincar\'{e} Probab. Stat. \textbf{55} (2019), no.~2,
  697--719.

\bibitem{phillips}
R.S. Phillips, \emph{On the generation of semigroups of linear operators},
  Pacific J. Math. \textbf{2} (1952), 343--369.

\bibitem{MR1872526}
L.~Saloff-Coste, \emph{Aspects of {S}obolev-type inequalities}, London
  Mathematical Society Lecture Note Series, vol. 289, Cambridge University
  Press, Cambridge, 2002.

\bibitem{MR2508847}
\bysame, \emph{Sobolev inequalities in familiar and unfamiliar settings},
  Sobolev spaces in mathematics. {I}, Int. Math. Ser. (N. Y.), vol.~8,
  Springer, New York, 2009, pp.~299--343.

\bibitem{Sato1999}
K.-I. Sato, \emph{L\'evy processes and infinitely divisible distributions},
  Cambridge Studies in Advanced Mathematics, Cambridge University Press, 1999.

\bibitem{ssv}
R.~Schilling, R.~Song, and Z.~Vondra\v{c}ek, \emph{Bernstein functions}, De
  Gruyter Studies in Mathematics, vol.~37, Walter de Gruyter \& Co., 2010.

\end{thebibliography}
\end{bibliography}

\end{document}